\documentclass{amsart}

\usepackage{tikz}
\usepackage{placeins}

\usepackage{amssymb,amsmath,amsfonts,amsthm,amscd}
\usepackage{graphicx} 

\newtheorem{theorem}{Theorem}[section]
\newtheorem{lemma}[theorem]{Lemma}
\newtheorem{proposition}[theorem]{Proposition}
\newtheorem{corollary}[theorem]{Corollary}
\newtheorem{remark}[theorem]{Remark} 
\newtheorem{example}[theorem]{Example}
\newtheorem{question}[theorem]{Question}
\newtheorem{problem}[theorem]{Problem}

\newtheorem{definition}[theorem]{Definition}

\newcommand{\R}{\mathbb{R}}
\newcommand{\RR}{\mathcal{R}}

\newcommand{\C}{\mathbb{C}}

\title{On a magneto-spectral invariant on finite graphs}

\author[Hu]{Chunyang Hu}
\address{School of Mathematical Sciences, University of Science and Technology of China, Hefei, China}
\email{chunyanghu@mail.ustc.edu.cn}

\author[Hua]{Bobo Hua}
\address{School of Mathematical Sciences, LMNS, Fudan University, Shanghai, China}
\email{bobohua@fudan.edu.cn}

\author[Kamtue]{Supanat Kamtue}
\address{Department of Mathematics
and Computer Science, Faculty of Science, Chulalongkorn University, Bangkok, Thailand}
\email{supanat.k@chula.ac.th}

\author[Liu]{Shiping Liu}
\address{School of Mathematical Sciences, University of Science and Technology of China, Hefei, China}
\email{spliu@ustc.edu.cn}

\author[M\"unch]{Florentin M\"unch}
\address{Institute of Mathematics, Leipzig University, Leipzig, Germany}

\email{cfmuench@gmail.com}
\author[Peyerimhoff]{Norbert Peyerimhoff}
\address{Department of Mathematical Sciences, Durham University, Durham, UK}
\email{norbert.peyerimhoff@durham.ac.uk}

\begin{document}

\begin{abstract} In this paper, we introduce a magneto-spectral invariant for finite graphs. This invariant vanishes on trees and is maximized by complete graphs. We compute this invariant for cycles, complete graphs, wheel graphs, hypercubes, complete bipartite graphs and suspensions of trees and derive various lower and upper bounds. In particular, we provide a sharp upper bound for regular bipartite graphs and derive a direct relation between the class of graphs assuming this upper bound and the class of unit weighing matrices, which are generalizations of complex Hadamard matrices. Moreover, this class of bipartite graphs has non-negative magnetic Bakry-\'Emery curvature and is preserved under both the Cartesian product and a partial tensor product for bipartite graphs. The study of our invariant for certain pairs of cospectral graphs indicates also that this invariant allows us to distinguish between them. Finally, we discuss the behaviour of this invariant under various graph operations and investigate relations to the spectral gap. 
\end{abstract}

\maketitle

\tableofcontents

\section{Introduction}

In this paper, we introduce a new magneto-spectral invariant (see Definition \ref{def:magspecheight} below) and discuss various of its properties.  
We begin by providing the notational background.

Let $G=(V,E)$ be a \emph{simple} (that is, without loops and multiple edges) finite graph. For adjacent vertices $x, y \in V$ we use the notation $x \sim y$ and we call $x,y \in V$ \emph{neighbours}. For $x \in V$, the \emph{vertex degree} $d_x$ is defined as the number of neighbours of $x$. The \emph{minimal} and \emph{maximal vertex degree} of $G$ is defined by
$$ d_{\rm{min}}(G) = \min_{x \in V} d_x, \quad d_{\rm{max}}(G) = \max_{x \in V} d_x. $$
The set of oriented edges $E^{or}(G)$ consists of ordered pairs $(x,y)$ with $x \sim y$. A \emph{magnetic potential} is a map $\sigma: E^{or}(G) \to S^1$ with $\sigma(y,x) = \sigma(x,y)^{-1}$. Such a potential induces, for every walk $\gamma: x_1 \sim x_2 \sim \cdots \sim x_k$, a potential of $\gamma$, given by
$$ \sigma(\gamma) = \sigma(x_1,x_2) \sigma(x_2,x_3) \cdots \sigma(x_{k-1},x_k). $$

In the case that $\sigma$ assumes only the values $\pm 1$, we also refer to it as a \emph{signature}.
Let $C(G,\C) = C(V,\C) = \{f: V \to \C\}$ the vector space of $\C$-valued functions on the vertices of $G$.
The \emph{magnetic Laplacian} associated to a magnetic potential $\sigma$ is a linear operator $\Delta^\sigma = \Delta^\sigma_G: C(V,\C) \to C(V,\C)$, given by
\begin{equation} \label{eq:magnLap} 
\Delta^\sigma f(x) = \sum_{y \sim x} (f(x) - \sigma(x,y)f(y)). 
\end{equation}

In the case of $\sigma \equiv 1$, we recover the standard non-normalized Laplacian $\Delta=\Delta_G$. The space $C(V,\C)$ has the following inner product: Let $f,g \in C(V,\C)$. Then
$$ \langle f,g \rangle = \sum_{x \in V} f(x) \overline{g(x)}. $$
$\Delta^\sigma$ is symmetric and has non-negative real spectrum. Its eigenvalues are denoted by 
$$ 0 \le \lambda_{\min}(\Delta_G^\sigma) = \lambda_1^\sigma(G) \le \lambda_2^\sigma(G) \le \cdots \le \lambda_n^\sigma(G) = \lambda_{\max}(\Delta_G^\sigma) $$ 
with $n=|V|$, and the variational characterization of the smallest eigenvalue $\lambda_1^\sigma(G)$ is given by
$$ \lambda_1^\sigma(G) = \inf_{f \neq 0} \RR^\sigma(f) $$
with
$$ \RR^\sigma(f) = \RR^\sigma_G(f) = \frac{\sum_{\{x,y\} \in E} |f(x)-\sigma(x,y)f(y)|^2}{\sum_{x \in V} |f(x)|^2}. $$
Similarly, the eigenvalues of the standard Laplacian $\Delta=\Delta_G$ on $G$ are denoted by
$$ 0 = \lambda_1(G) \le \lambda_2(G) \le \cdots \le \lambda_n(G)=\lambda_{\max}(\Delta_G) $$
and the corresponding Rayleigh quotient is denoted by $\RR = \RR_G$.
The \emph{spectral gap} of a graph $G$ is defined as $\lambda_2(G)$.

An important transformation of magnetic potentials is given in the following definition.

\begin{definition} Let $G=(V,E)$ be a finite graph and $\sigma: E^{or}(G) \to S^1$. For any function $\tau: V \to S^1$, we define the associated \emph{gauge transformation} $\sigma^\tau$ of $\sigma$ as
$$ \sigma^\tau(x,y) = \tau(x)^{-1} \sigma(x,y) \tau(y). $$
The function $\tau$ is called a \emph{gauge function}. We say that two magnetic potentials $\sigma_1,\sigma_2$ are \emph{gauge equivalent}, if they are related via a suitable gauge function. 
\end{definition}

Note that magnetic Laplacians with gauge equivalent magnetic potentials are unitarily equivalent and have therefore the same spectrum. In particular, if $f$ is a $\Delta^{\sigma^\tau}$-eigenfunction to the eigenvalue $\lambda$, then $\tau f$ is a $\Delta^\sigma$-eigenfunction to the same eigenvalue.
Note also that any magnetic potential can be gauged to be trivial on a given spanning forest and that two different potentials which are both trivial on such a spanning forest cannot be gauge euqivalent.

In the case of signatures, we also refer to gauge functions $\tau: V \to \{-1,1\}$ as \emph{switching functions} and refer to two signatures $\sigma_1,\sigma_2$ which are related via a switching function as being \emph{switching equivalent}. 
Signatures which are switching equivalent to the trivial signature are called \emph{balanced}. We denote the trivial signature by $\sigma_0 \equiv 1$ in this paper. The signature assuming only values $-1$ is referred to as the \emph{anti-balanced} signature. It is denoted by $\overline{\sigma} \equiv -1$. We refer to magnetic Laplacians corresponding to signatures as \emph{signed Laplacians}.

Signed graphs and signed Laplacians have been studied extensively (see, e.g. \cite{AL20,H57,Z82,Z98,Z13} and references therein). On the other hand, there is a vast literature on discrete magnetic Laplacians: a small collection of some papers on this topic is \cite{ABG25,B13,BGKLM20,CdV13,DM06,FCLP18,G14,HSh01,KS23,LLPP15,Shu94,Su94}. 

\medskip

Now we introduce the following magneto-spectral invariant of finite graphs.

\begin{definition}[Magneto-spectral height] \label{def:magspecheight} Let $G=(V,E)$ be a finite graph. The \emph{magneto-spectral height} of $G$ is defined to be
\begin{equation} \label{eq:magspecheight} 
\nu(G) := \sup_{\sigma} \lambda_1^\sigma(G), 
\end{equation}
where the supremum is taken over all magnetic potentials $\sigma: E^{or}(G) \to S^1$. 

A magnetic potential is called \emph{maximal}, if $\lambda_1^\sigma(G)$ agrees with $\nu(G)$.
\end{definition}

Henceforth, all our graphs $G$ are assumed to be finite, unless stated otherwise.

\subsection{Structure of the paper} 

In Section \ref{sec:examples}, we first compute the magneto-spectral height for various examples, namely, we have $\nu(T)=0$ for trees (Lemma \ref{lem:treenu}), $\nu(C_n)=2-2\cos(\pi/n)$ for cycles $C_n$ (Example \ref{ex:cycle}), $\nu(K_n)=n-2$ for complete graphs $K_n$ (Example \ref{ex:completegraph}), $\nu(W_n) = 3-2\cos(\pi/n)$ for wheel graphs $W_n$ (Proposition \ref{prop:wheelnu}), and $\nu(Q^d) = \nu(K_{d,d}) = d-\sqrt{d}$ for hypercubes $Q^d$ (Proposition \ref{prop:nuhypercube}) and complete bipartite graphs $K_{d,d}$ (Example \ref{ex:complbipart}). For many of these examples, we describe a maximal potential explicitly, and we find also that maximal potentials are often unique up to gauge equivalence. Subsection \ref{subsec:cospectral} is concerned with pairs of cospectral graphs. It is known (see \cite{ZLY09}) that $W_6$ is not determined (up to isomorphisms) by its Laplace spectrum and that there exists a different cospectral graph $\widehat G$. We show that the magneto-spectral height distinguishes $W_6$ and $\widehat G$ (see Example \ref{ex:wheelcospec}). Since the magneto-spectral height vanishes for trees, it does not distinguish between cospectral pairs of trees (see Example \ref{ex:cospectrees}). We end this subsection with a discussion of the magneto-spectral height for particular pairs of cospectral graphs constructed by Sunada's method and via Seidel switching. In Subsection \ref{subsec:counterex}, we discuss some counterexamples for potential conjectures about the magneto-spectral height which may arise naturally.

Section \ref{sec:genresults} is concerned with various upper and lower bounds for the magneto-spectral height. While the derivation of lower bounds is often very challenging, upper bounds can be derived more easily. We present combinatorial upper bounds (Proposition \ref{prop:combupbd}), 
the sharp upper bound $d-\sqrt{d}$ for $d$-regular bipartite graphs (Proposition \ref{prop:regbipartuppbd}) and show that the class of graphs attaining this upper bound is closed under Cartesian products (Proposition \ref{prop:maxregbipartite}),
upper bounds in terms of topological/spectral data of induced subgraphs (Theorem \ref{thm:outdegree_upper_bound} and Corollary \ref{cor:nudmin}), upper and lower Ramanujan type bounds (Propositions \ref{prop:alonbopp} and \ref{prop:MSS}), and lower bounds in terms of diameter and curvature (Propositions \ref{prop:nudiamvol} and \ref{prop:nudiam}). 
In the final Subsection \ref{sec:regbipuppbd}, we investigate the class of $d$-regular bipartite graphs $G$ satisfying $\nu(G)=d-\sqrt{d}$ further and discover a direct connection with the class of unit weighing matrices (Theorem \ref{sec:regbipuppbd}). Unit weighing matrices are a generalization of complex Hadamard matrices. Via this connection we exhibit various other graphs in this class and conclude closedness of this graph class under a new graph product associated to the tensor product of the corresponding unit weighing matrices (Remark \ref{rem:UWnd}). We also show that the magnetic Bakry-\'Emery curvature for maximal potentials of all graphs in this class is non-negative (Theorem \ref{thm:Chunyang}).

Various graph constructions are discussed in Section \ref{sec:graphconstructions}, namely, the addition of internal edges (Lemma \ref{lem:addedge}), combining graphs by bridges (Lemma \ref{lem:GHbridge}), splitting vertices and edges (Lemmas \ref{lem:splitvertex} and \ref{lem:splitedge}) and taking Cartesian products (Lemma \ref{lem:cartprod}). The longer final Subsection \ref{sec:suspensions} is concerned with graph suspensions. Theorem \ref{thm:suspgraph} in this subsection states that the magneto-spectral height does not increase by more than one under a graph suspension and that we have $\nu(\tilde T) = 1$ for the suspension $\tilde T$ of any tree $T$ with at least two vertices.
Suspensions of star graphs are considered in Corollary \ref{cor:suspstargraph}, which leads to the somewhat surprising fact that there are finite graphs with uncountably many non gauge equivalent maximal potentials. Examples \ref{ex:susp1} and \ref{ex:susp2} confirm the observation that the relation $\nu(\tilde G) = \nu(G)+1$ holds for suspensions $\tilde G$ of many finite graphs. However, there are also counterexamples for this relation (see Remark \ref{remark:suspension}). Finally, Proposition \ref{prop:nususpsm1} classifies all graphs $G$ whose suspension $\tilde G$ satisfies $\nu(\tilde G)<1$.

In Section \ref{sec:specgap} we briefly discuss potential relations between our magneto-spectral height and the spectral gap, which is another widely investigated and fundamental spectral invariant. The example of hypercubes and Example \ref{ex:cycalmcomplete} show that $\nu(G)$ cannot be bounded above by the spectral gap (up to any multiplicative scaling). On the positive side, the inequality \eqref{eq:nulambda2} in \cite{Sa23} provides evidence that $\nu(G)$ may be bounded below by the spectral gap (up to a multiplicative scaling).

Section \ref{sec:openprobs} presents some open problems about the magneto-spectral height which may inspire further work on this topic.

At the end of this paper, we discuss in Appendix A that the magneto-spectral height has an analogue for Riemannian manifolds.

\subsection{Context and outlook}

It is natural to ask what the magneto-spectral height is actually measuring. For a connected graph, this invariant vanishes if and only if the graph does not contain any cycles. Hence this invariant could be viewed as a measure of the complexity of the cycle structure of a graph, where larger cycles seem to contribute less to this complexity than smaller ones. 
On the other hand, we will see that graphs $G$ with leaves have very small magneto-spectral heights $\nu(G) < 1$.

Magnetic potentials on graphs are special cases of graph connections. A \emph{connection} on a graph $G=(V,E)$ is a map $\sigma: E^{or}(G) \to O(n)$ or $\sigma: E^{or}(G) \to U(n)$ satisfying $\sigma(y,x) = \sigma(x,y)^{-1}$ for all $(x,y) \in E^{or}(G)$. A connection can be viewed as a discrete analogue of a parallel transport in Riemannian Geometry. A connection $\sigma$ gives rise to a connection Laplacian $\Delta^\sigma$, defined on function $f: V \to \R^n$ or $f: V \to \C^n$, with the same definition as in \eqref{eq:magnLap}.
In the case $O(1) = \{-1,1\}$, the associated connection Laplacians agree with the signed Laplacians and in the case $U(1) = S^1$, the associated connection Laplacians agree with the magnetic Laplacians.
The eigenvalues of connection Laplacians are again all real-valued and non-negative, and we denote the smallest eigenvalue again by $\lambda_1^\sigma(G)$.
Let $\Gamma =O(n)$ or $\Gamma= U(n)$.
Our magneto-spectral height has a natural generalization to connection Laplacians via
$$ \nu_\Gamma(G) = \sup_\sigma \lambda_1^\sigma(G), $$
where $\sigma$ runs through all connections $\sigma: E^{or}(G) \to \Gamma$. It is interesting to ask whether the $\nu$-invariants of higher dimensional connection Laplacians carry further interesting information about the graph.

Another modification of the $\nu$-invariant can be introduced for cyclic subgroups of $S^1$. If we refer to the set of all $k$-th roots of unity as 
$$
S_k := \{ {\xi_k}^j: j \in \{0,1,\dots,k-1\} \} \subset S^1, \quad \text{with $\xi_k = e^{2 \pi i/k}$,} 
$$
we can consider the following invariant:
$$ \nu^k(G) = \sup_{\sigma: E^{or}(G) \to S_k} \min\left\{
\lambda_1^\sigma(G),
\lambda_1^{\sigma^2}(G),
\dots,
\lambda_1^{\sigma^{k-1}}(G) \right\},
$$
where $\sigma^j(x,y) = (\sigma(x,y))^j$ for all oriented edges $(x,y) \in E^{or}(G)$.

In the special case $k=2$, $\nu^2(G)$ agrees with the supremum in \eqref{eq:magspecheight}, where the supremum runs only over the signatures instead of all magnetic potentials. Any $\sigma: E^{or}(G) \to S_k$ gives rise to a cyclic $k$-lift $\pi: \widehat G \to G$ as follows: Every directed edge $(x,y) \in E^{or}(G)$ corresponds to a particular $k$-matching between the fibers $\pi^{-1}(x) = \{x_1,\dots,x_k\} \in \widehat V$ and $\pi^{-1}(y) = \{y_1,\dots,y_k\} \in \widehat V$ in the lift $\widehat G=(\widehat V,\widehat E)$, which is determined by $\sigma(x,y) = \xi_k^\ell \in S_k$ as follows $x_j \sim_{\widehat G} y_{j+\ell}$ for $j \in \{0,1,\dots,k-1\}$, where the indices are taken modulo $k$ (an example of this is illustrated in Figure \ref{fig:cyclic-lift}).

The spectrum of the standard Laplacian $\Delta_{\widehat G}$ on $\widehat G$ is the multiset union of the spectra of the magnetic Laplacians $\Delta^{\sigma^j}$ on $G$, for $j=0,1,\dots,k-1$. For further details about cyclic lifts and this spectral relation, we refer the readers to \cite{MSS15}, \cite{LPV20} and \cite{ChV17}. These papers are concerned with the question whether bipartite Ramanujan graphs admit cyclic $k$-lifts which are again Ramanujan. In fact, this is true for $k=2,3,4$, as shown in these papers. For $k \ge 5$, bipartite Ramanujan graphs admits also always general $k$-lifts which are again Ramanujan (see \cite{HPS18}), but it is still an open problem, whether this also holds for the smaller class of cyclic $k$-lifts.    

\begin{figure}[h]
\begin{center}
\begin{tikzpicture}[thick,scale=1]
\filldraw [black] (0,0) circle (2pt);
\filldraw [black] (3,0) circle (2pt);
\filldraw [black] (0,2) circle (2pt);
\filldraw [black] (0,3) circle (2pt);
\filldraw [black] (0,4) circle (2pt);
\filldraw [black] (0,5) circle (2pt);
\filldraw [black] (0,6) circle (2pt);
\filldraw [black] (3,2) circle (2pt);
\filldraw [black] (3,3) circle (2pt);
\filldraw [black] (3,4) circle (2pt);
\filldraw [black] (3,5) circle (2pt);
\filldraw [black] (3,6) circle (2pt);
\draw (0,0) node[label=below:$x$]{};
\draw (3,0) node[label=below:$y$]{};
\draw (0,2) node[label=left:$x_1$]{};
\draw (0,3) node[label=left:$x_2$]{};
\draw (0,4) node[label=left:$x_3$]{};
\draw (0,5) node[label=left:$x_4$]{};
\draw (0,6) node[label=left:$x_5$]{};
\draw (3,2) node[label=right:$y_1$]{};
\draw (3,3) node[label=right:$y_2$]{};
\draw (3,4) node[label=right:$y_3$]{};
\draw (3,5) node[label=right:$y_4$]{};
\draw (3,6) node[label=right:$y_5$]{};
\draw (1.5,-1) node[label=below:{\LARGE$\sigma(x,y)={\xi_5}^2$}]{};
\draw (0,0) -- (3,0);
\draw (0,2) -- (3,4);
\draw (0,3) -- (3,5);
\draw (0,4) -- (3,6);
\draw (0,5) -- (3,2);
\draw (0,6) -- (3,3);
\end{tikzpicture}
\end{center}
\caption{Relation between a directed edge $(x,y) \in E^{or}(G)$ with $\sigma(x,y) = {\xi_5}^2$ and a matching between the fibers $\pi^{-1}(x)$ and $\pi^{-1}(y)$ in the $5$-lift $\widehat G$. \label{fig:cyclic-lift}}
\end{figure}

\FloatBarrier

Our magneto-spectral invariant can also be considered in relation to maximal abelian covers. For more information and new results about maximal abelian covers, we refer to the recent articles \cite{ABG25,LMST25}. By Floquet-Bloch Theory, the $L^2$-spectrum of the standard Laplacian $\Delta_{\tilde G}$ on the maximal abelian cover $G^{ab}$ of a finite connected graph $G=(V,E)$ is given by the union
\begin{equation} \label{eq:floquet} 
{\rm{spec}}(\Delta_{\tilde G}) = \bigcup_{j=1}^{|V|} \left( \bigcup_{\sigma: E^{or}(G) \to S^1} \lambda_j^\sigma(G) \right), \end{equation}
where the inner union represents either an interval (due to continuity of $\sigma \mapsto \lambda_j^\sigma(G)$) or a single value (if $\sigma \to \lambda_j^\sigma(G)$ is constant). In the first case, the contribution to the spectrum is absolutely continuous, and in the second case, the single value represents a genuine eigenvalue, which is of infinite multiplicity if $G$ is not a tree. Our magneto-spectral height $\nu(G)$ represents the upper end of the first inner union ($j=1$) of \eqref{eq:floquet}, that is,
$$ \bigcup_{\sigma: E^{or}(G) \to S^1} \lambda_1^\sigma(G) = [0,\nu(G)]. $$
Note also that there is a natural generalization of the magneto-spectral height for Schr\"odinger operators $\Delta_G + q$ with electric potential $q: V \to \R$ on multi-graphs $G=(V,E)$. For example, we have for the Laplacian on the multi-graph $G$ with one single vertex and one loop
$$ \nu(G) =  \sup_{a \in S^1} 2-a-\bar a = 4. $$
In this article, however, we will restrict to simple graphs, unless stated otherwise.

\section{Examples of graphs} \label{sec:examples}

We start this section about examples with the class of graphs whose magneto-spectral height vanishes.

\begin{lemma} \label{lem:treenu}
Any magnetic potential can be gauged away if and only if $G=(V,E)$ is a forest. Therefore,
we have $\nu(G)=0$ if and only if at least one connected component of $G$ is a tree.
\end{lemma}

\subsection{Cycles}\label{subsec:cycles}
\begin{example}[The cycle $C_n$] \label{ex:cycle}
Let $n \ge 3$. We will show for the $n$-cycle $C_n$ that we have
\begin{equation} \label{eq:nucycle}
\nu(C_n) = 2-2\cos(\pi/n). 
\end{equation}
Moreover, $C_n$ has a unique maximal potential $\hat \sigma$, up to gauge equivalence, and the multiplicity of $\lambda_1^{\hat \sigma}$ is two. 
\end{example}

In the remaining of this subsection, we provide the proof of the above statements. We enumerate the vertices of $C_n$ by $x_0,x_1,\dots,x_{n-1}$. By gauge equivalence and the fact that any path of length $n-1$ in $C_n$ is a spanning tree, we can restrict our considerations to magnetic potentials $\sigma$, which are non-trivial only in the edge $\{x_{n-2},x_{n-1}\}$, that is 
$$ \sigma_t(x_{n-2},x_{n-1}) = e^{it}. $$
The Laplacian $\Delta^{\sigma_t}$ is then  represented by the following matrix:
$$ \Delta^{\sigma_t} \cong \begin{pmatrix} \begin{array}{ccccccc|c}
2 & -1 & 0 & \cdots & \cdots & 0 & 0 & -1 \\ 
-1 & 2 & -1 & & & 0 & 0 & 0 \\
0 & -1 & 2 & \ddots & & 0 & 0& 0 \\
\vdots & & \ddots & \ddots & \ddots & & & \vdots \\
\vdots & & & \ddots & \ddots & \ddots & & \vdots \\
0 & 0 & 0 & & \ddots & 2 & -1 & 0 \\
0 & 0 & 0 & & & -1 & 2 & -e^{it} \\
\hline
-1 & 0 & 0 & \cdots & \cdots & 0 & -e^{-it} & 2
\end{array} \end{pmatrix}, $$
and the eigenfunctions $f_j: C_n \to \C$ are illustrated in Figure \ref{fig:Cdeigfunc} with $\xi_n = e^{2 \pi i/n}$. Note that, in the case $t=0$, the red factors disappear and we end up with the eigenfunctions of the standard Laplacian on $C_n$. Moreover, we check for any $j_1\neq j_2$ that
\begin{align}\label{eq:orthogonality}
\langle f_{j_1},f_{j_2} \rangle
&=\sum_{\ell=0}^{n-2}\left(e^{\frac{it}{n}}\xi_{n}^{j_1}\right)^{\ell}\overline{\left(e^{\frac{it}{n}}\xi^{j_2}_{n}\right)^{\ell}}+\left(e^{\frac{it}{n}}\xi_{n}^{j_1}\right)^{-1}\overline{\left(e^{\frac{it}{n}}\xi_n^{j_2}\right)^{-1}}\notag\\
&=\sum_{\ell=-1}^{n-2}\xi_{n}^{(j_1-j_2)\ell}=\xi_{n}^{-(j_1-j_2)}\sum_{\ell=0}^{n-1}\xi_{n}^{(j_1-j_2)\ell}=0.
\end{align}
\begin{figure}
\begin{center}
\begin{tikzpicture}[thick,scale=1]
\filldraw [black] (0,3) circle (2pt);
\filldraw [black] (3,6) circle (2pt);
\filldraw [black] (3,0) circle (2pt);
\filldraw [black] (7.7,0) circle (2pt);
\filldraw [black] (7.7,6) circle (2pt);
\filldraw [black] (10.7,3) circle (2pt);
\draw (1,1.5) node[label=below:$e^{i t}$]{};
\draw (1,4.5) node[label=above:$+1$]{};
\draw (5.6,6) node[label=above:$+1$]{};
\draw (5.6,0) node[label=below:$+1$]{};
\draw (9.2,4.5) node[label=above:$+1$]{};
\draw (3,6) node[label=above left:$x_0$]{};
\draw (7.7,6) node[label= above right:$x_1$]{};
\draw (10.7,3) node[label=right:$x_2$]{};
\draw (7.7,0) node[label=below right:$x_{n-3}$]{};
\draw (3,0) node[label=below left:$x_{n-2}$]{};
\draw (0,3) node[label=left:$x_{n-1}$]{};
\draw (3,6) node[label=below right:${\color{red}1} \cdot 1$]{};
\draw (7.7,6) node[label=below left:${\color{red}e^{\frac{it}{n}}}\cdot \xi_n^j$]{};
\draw (10.7,3) node[label=left:${\color{red}e^{\frac{2it}{n}}}\cdot \xi_n^{2j}$]{};
\draw (8.4,0.1) node[label=above left: ${\color{red}e^{\frac{(n-3)it}{n}}}\cdot \xi_n^{(n-3)j}$]{};
\draw (2.6,0.1) node[label=above right:${\color{red}e^{\frac{(n-2)it}{n}}}\cdot \xi_n^{(n-2)j}$]{};
\draw (0,3) node[label=right:${\color{red}e^{\frac{-it}{n}}}\cdot \xi_d^{(n-1)j}$]{};
\draw (0,3) -- (3,6);
\draw (0,3) -- (3,0);
\draw (3,0) -- (7.7,0);
\draw (3,6) -- (7.7,6);
\draw (7.7,6) -- (10.7,3);
\draw[dashed] (7.7,0) -- (10.7,3);
\end{tikzpicture}
\end{center}
\caption{Eigenfunctions $f_j$ of the magnetic Laplacian $\Delta^\sigma$ on $C_n$ with $j=0,1,\dots,n-1$. The red factors describe the influence of the magnetic potential $\sigma_t$. \label{fig:Cdeigfunc}}
\end{figure}
\FloatBarrier
It is easy to check that the corresponding eigenvalue to $f_j$ is
\begin{equation} \label{eq:mujt} 
\mu_j(t) = 4 \sin^2\left( \frac{t+2\pi j}{2n}\right) = 2 - 2\cos\left(\frac{t+2\pi j}{n}\right) = 2 - \left( \xi_n^j e^{i \frac{t}{n}} + \xi_n^{-j} e^{-i \frac{t}{n}} \right). 
\end{equation}
Note that we have for $t \in [0,\pi]$,
$$ \cos\left(\frac{t}{n}\right) \ge \cos\left(\frac{\pi}{n}\right) \ge \cos\left( \frac{t+2j\pi}{n} \right), $$
since $\frac{t+2\pi j}{n} \in [\frac{2\pi}{n},2\pi-\frac{\pi}{n}]$ for $j=1,2,\dots,n-1$. This implies for $t \in [0,\pi]$ that
$$ \mu_0(t) \le \mu_j(t) \quad \text{for all $j=1,2,\dots,n-1$}, $$
with equality if and only if $t=\pi$ and $j=n-1$.
Moreover, we have for $t \in [0,\pi]$,
$$ \lambda_1(\Delta^{\sigma_t}) = \mu_0(t) = 2-2\cos(t/n).  $$
Since the eigenvalue graphs are symmetric with respect to $\pi$, we conclude that
$$
\nu(C_n) = \mu_0(\pi) = 2-2\cos(\pi/n). 
$$
Moreover, the multiplicity of the smallest eigenvalue of $\Delta^{\sigma_\pi}$ is precisely $2$, since for $t=\pi$,
$$ \mu_0(\pi) = \mu_{n-1}(\pi) < \mu_j(\pi) \quad \text{for all $j =1,2,\dots,n-2$.} $$

The above results directly imply the following asymptotics of $\nu(C_n)$ as $n \to \infty$.

\begin{corollary} We have
$$ \nu(C_n) = 4\sin^2(\pi/(2n)) \le \frac{\pi^2}{n^2}, $$
and therefore $\nu(C_n) \to 0$ as $n \to \infty$.
\end{corollary}

\subsection{Complete graphs}

\begin{example}[The complete graph $K_n$] \label{ex:completegraph} For the complete graph $K_n$, $n \ge 2$, we will show that
$$
\nu(K_n) = n-2.
$$
Moreover, the anti-balanced signature  $\overline{\sigma}\equiv -1$ is the unique maximal potential of $K_n$, up to gauge equivalence, and the multiplicity of the eigenvalue $\lambda_1^{\overline{\sigma}}$ is $n-1$. 
\end{example}

Let us now provide proofs of these statements. We start with the following lemma, which can be viewed as a special case of the later Theorem \ref{thm:outdegree_upper_bound}(a) in Section \ref{sec:genresults}. For the reader's convenience, we also present a proof in this simpler case.

\begin{lemma} \label{lem:upbdxdy} Let $G=(V,E)$ be a graph and $T \subset V$ be a subset whose induced subgraph is a forest.
Then we have
\begin{equation} \label{eq:nuisoperi} 
\nu(G) \le \frac{|\partial T|}{|T|}, 
\end{equation}
where $\partial T \subset E$ is the set of edges connecting vertices of $T$ with vertices of $V \setminus T$. In particular, we have
\begin{equation} 
\label{eq:nuGmindxdy}
\nu(G) \le \min_{x \sim y} \frac{d_x+d_y}{2}-1. 
\end{equation}
\end{lemma}

\begin{remark} Inequality \eqref{eq:nuisoperi} in  Lemma \ref{lem:upbdxdy} can also be viewed as the inequality
$$ \nu(G) \le h_{\rm{forest}}(G), $$
where $h_{\rm{forest}}$ is a Cheeger-type constant, which takes the infimum of isoperimetric ratios only for induced subgraphs which are forests. 
\end{remark}

\begin{proof}[Proof of Lemma \ref{lem:upbdxdy}] We first present the proof in the special case when $T$ is a tree. Let $\sigma: E^{or} \to S^1$ be an arbitrary magnetic potential and $x_0 \in T$. Choose
$f: V \to \C$ as follows: We set $f(x_0) =1$ For any vertex $x \in T \setminus \{x_0\}$, choose the unique path $x_0 \sim x_1 \sim \cdot \sim x_l = x$ in $T$ and define
$$ f(x) = \sigma(x_0,x_1)^{-1} \sigma(x_1,x_2)^{-1} \cdots \sigma(x_{l-1},x_l)^{-1}. $$
For any vertex $x \in V \setminus T$, define
$$ f(x) = 0. $$
Then we have 
$$
\lambda_1^\sigma(G) \le \RR^\sigma(f) = \frac{\displaystyle{\sum_{x \sim y \atop x\in T, y \in V\setminus T}} |f(x)-\sigma(x,y)f(y)|^2}{\sum_{x \in T} |f(x)|^2}
= \frac{|\partial T|}{|T|}.
$$
In the case that $T$ is a forest, we choose in every connected component of $T$ a base point $x_0$ to define the test function $f$ and argue accordingly.

For the second part of the proposition, choose $T=\{x,y\}$ for any pair $x \sim y$.
\end{proof}

Lemma \ref{lem:upbdxdy} yields the following upper bound for the complete graph $K_n$:
$$ \nu(K_n) \le \min_{x \sim y} \frac{d_x+d_y}{2} - 1 = (n-1)-1 = n-2. $$
On the other hand, we have
$$ \lambda^{\overline{\sigma}}(K_n) = n-2 $$
for the anti-balanced magnetic potential $\overline{\sigma}\equiv -1$, since 
$$ \Delta^{ \overline{\sigma}} \equiv (n-2) {\rm{Id}}_{n} + J_n, $$
with the all-$1$ square matrix $J_n$ of size $n$, and all non-zero vectors perpendicular to the all-$1$ vector of size $n$ are eigenvectors to the eigenvalue $n-2$, whereas the all-$1$ vector of size $n$ is an eigenvector to the eigenvalue $2(n-1) > n-2$. Therefore, we have equality $\nu(G) = n-2$, assumed by the anti-balanced magnetic potential $\overline{\sigma}$. For this potential, the multiplicity of $\lambda_1^{\overline{\sigma}}$ is $n-1$. 

It remains to show that any maximal magnetic potential of $K_n$ is gauge equivalent to the anti-balanced signature $\overline{\sigma}$. Let $\hat \sigma$ be a maximal magnetic potential on $K_n$. For $k,l \in \{1,\dots,n\}$, we introduce the function
$$ f_{k,l}(x_m) = \begin{cases} 1 & \text{if $m=k$,} \\
\hat \sigma(x_k,x_l)^{-1} & \text{if $m=l$,} \\ 0 & \text{otherwise.} \end{cases} $$
Then we have
$$ \lambda_1^{\hat \sigma}(K_n) = n-2 = \RR^{\hat \sigma}(f_{k,l}), $$
and $f_{k,l}$ is therefore an eigenfunction of $\Delta^{\hat \sigma}$ to the eigenvalue $n-2$. The eigenfunction equation yields for $m \not\in\{k,l\}$,
$$ \Delta^{\hat \sigma} f_{k,l}(x_m) = \underbrace{\sum_{j \neq m} f_{k,l}(x_m) - \hat \sigma(x_m,x_j) f_{k,l}(x_j)}_{\hat \sigma(x_m,x_k)- \hat \sigma(x_m,x_l)\hat \sigma(x_k,x_l)^{-1}} = (n-1)f_{k,l}(x_m) = 0, $$
which implies 
$$ \hat \sigma(x_m,x_k) \hat \sigma(x_k,x_l) \hat \sigma(x_l,x_m) = -1. $$
Choosing the gauge function $\tau: V \to S^1$,
$$ \tau(z) = \begin{cases} -1, & \text{if $z=x_1$,} \\
\hat \sigma(x_1,x_j)^{-1}, & \text{if $z=x_j$, $j \ge 2$,} \end{cases} $$
we obtain
$$ \hat \sigma^\tau(x_1,x_j) = \tau(x_1)^{-1} \hat \sigma(x_i,x_j)\tau(x_j) = -1 \quad \text{for all $j \ge 2$.} $$
Since, for $k,l \ge 2$, $k\neq l$,
$$\hat \sigma^\tau(x_k,x_l) = \hat \sigma^\tau(x_1,x_k) \hat \sigma^\tau(x_k,x_l) \hat \sigma^\tau(x_l,x_1) = \hat \sigma(x_1,x_k) \hat \sigma(x_k,x_l) \hat \sigma(x_l,x_1) = -1,$$
$\hat \sigma^\tau$ is the anti-balanced signature.
\qed

\medskip

To derive the rigidity statement in Corollary \ref{cor:complrigid} below, we start with the following general fact.

\begin{proposition} \label{prop:nudxdy}
Let $G=(V,E)$ be a graph and $\{x,y\} \in E$ be an edge. If
$$ \nu(G) = \frac{d_x+d_y}{2}-1, $$
then we have
$$ S_1(x) \setminus \{y\} = S_1(y) \setminus \{x\}. $$
In particular, we have $d_x=d_y$.
\end{proposition}

\begin{proof}
    Let $\hat \sigma: E^{or}(G) \to S^1$ be a maximal magnetic potential, that is,
    $$ \nu(G) = \lambda_1^{\hat \sigma}(G). $$
    Let $f: V \to \C$ be the following function
    $$ f(z) = \begin{cases} 1 & \text{if $z=x$,} \\
\hat \sigma(x,y)^{-1} & \text{if $z=y$,} \\ 0 & \text{otherwise.} \end{cases}$$
    Then we have
    $$ \frac{d_x+d_y}{2}-1 = \nu(G) = \lambda_1^{\hat \sigma}(G) \le \RR^{\hat \sigma}(f) = \frac{d_x+d_y}{2}-1. $$
    Therefore, $f$ is an eigenfunction of $\Delta^{\hat \sigma}$ to the eigenvalue $\frac{d_x+d_y}{2}-1$. Let us denote the neighbours of $x$ by $y=y_1,y_2,\dots,y_{d_x}$. Assume $y_j$, $j \in \{2,\dots,d_x\}$ is not a neighbour of $y=y_1$. Then we have
    $$ \Delta^{\hat \sigma}f(y_j) = \sum_{z \sim y_j} f(y_j)-\hat \sigma(y_j,z)f(z) = -\hat \sigma(y_j,x)f(x) = 0, $$
    which is a contradiction. This shows that all neighbours of $x$, different from $y$, are also neighbours of $y$. Interchanging the role of $x$ and $y$ implies that we have $d_x=d_y.$
\end{proof}

Now we can prove the following rigidity result.

\begin{corollary} \label{cor:complrigid}
Any connected graph $G=(V,E)$ satisfying
$$ \nu(G) = \max_{x \sim y} \frac{d_x+d_y}{2}-1 $$
is already a complete graph. In particular, any connected graph $G$ with
$$ \nu(G) = d_{\max}-1, $$
where $d_{\max}$ is the maximal vertex degree of $G$, 
is a complete graph $K_{d_{\max}+1}$.
\end{corollary}

\begin{proof}
    It follows from Lemma \ref{lem:upbdxdy} that we have
    $$ \min_{x \sim y} \frac{d_x+d_y}{2} - 1 \ge \nu(G) = \max_{x \sim y} \frac{d_x+d_y}{2}-1,$$
    which implies $d_x+d_y$ is constant for any edge $\{x,y\} \in E$. It follows from Proposition \ref{prop:nudxdy} that any connected component of $G$ is regular and complete. Since $G$ is assumed to be connected, the first statement of the corollary follows. 

    Since 
    $$\max_{x \sim y} \frac{d_x+d_y}{2} \le d_{\max}, $$
    the condition $\nu(G) = d_{\max}-1$ implies that we have $d_x=d_y=d_{\max}$ for all adjacent vertices. Since $G$ is connected, $G$ is regular and Proposition \ref{prop:nudxdy} implies again that $G$ is a complete graph.
\end{proof}

\subsection{Wheel graphs}
\begin{example}[Suspension of a tree] \label{ex:treesuspension}
Let $T$ be a tree and $\tilde T$ be its suspension, that is, we add a vertex $x$ to $T$ and connect it to all vertices in $T$. Then we have
$$ \nu(\tilde T) \le 1. $$
This follows directly from \eqref{eq:nuisoperi} in Lemma \ref{lem:upbdxdy}. In fact, we will see later (see Theorem \ref{thm:suspgraph}) that we have $\nu(\tilde T) = 1$ for all trees. 
\end{example}

Next we consider suspensions of cycles. The suspension of a cycle $C_n$ ($n \ge 3$) is the wheel graph $W_n$. In the next example we compute an explicit lower bound for $\nu(
W_n)$. In contrast to the suspension of trees, we will see that $\nu(W_n) > 1$.  

\begin{example}[The wheel graphs $W_n$] \label{ex:wheel} Let $n \ge 3$ and $x_0,\dots,x_{n-1}$ be the vertices of $W_n$ in the cycle $C_n \subset W_n$ and $x \in W_n$ be the suspension vertex. Let $\sigma_1: E^{or}(W_n) \to \{-1,1\}$ be given by  
$$ \sigma_1(x_j,x_{j+1}) = \begin{cases} 1 & \text{for $j=0,1,\dots,n-2$,} \\
-1 & \text{for $(j,j+1)=(n-1,0)$,}\end{cases}$$ 
and
$$ \sigma_1(x,x_j) = \begin{cases} -\left(e^{-i \frac{2k+1}{n}\pi}\right)^j & \text{for $j=0,1,\dots,n-2$,} \\
- e^{i \frac{2k+1}{n} \pi} & \text{for $j=n-1$.}\end{cases} $$ 
We will show that the eigenvalues of $\Delta^{\sigma_1}$ on $W_n$ (with multiplicities) are given by the values 
$$ 3 - 2 \cos\left( \frac{\pi+2\pi j}{n} \right) \quad \text{with $j =0,1,\dots,n-1$, $j \neq k$}, $$
all counted once,
and the roots of
$$ \lambda^2-\left(n+3-2\cos\left(\frac{2k+1}{n}\pi \right)\right)\lambda+\left( 2-2 \cos \left( \frac{2k+1}{n} \pi \right)\right)n = 0, $$
(with $n=2k$ for even $n$ and $n=2k+1$ for odd $n$) which are also counted once.
Note that this equation simplifies in the case $n=2k+1$ to
$$ \lambda^2-(n+5)\lambda + 4n = 0. $$
Using Lemma \ref{lem:destimates} below, this implies that
$$ \nu(W_n) \ge \lambda_1^{\sigma_1}(W_n) = 3-2 \cos\left( \frac{\pi}{n} \right) = 1 + \nu(C_n). $$

Note that the restriction of $\sigma_1$ to $C_n$, denoted again by $\sigma_1$, for simplicity, is a maximal magnetic potential of $C_n$.
The eigenfunctions of $\Delta^{\sigma_1}_{C_n}$ 
are given by (cf. Figure \ref{fig:Cdeigfunc} with $t=\pi$)
$$f_j=\begin{pmatrix}
    1 \\
    e^{\frac{i\pi}{n}}\xi_{n}^j \\
    (e^{\frac{i\pi}{n}}\xi_{n}^j)^2 \\
    \vdots \\
    (e^{\frac{i\pi}{n}}\xi_{n}^j)^{n-2} \\
    e^{-\frac{i\pi}{n}}(\xi_{n}^j)^{-1}
\end{pmatrix},\quad \text{ $j=0,1,\cdots, n-1$.}$$
Note that we have 
$$ \overline{f_k}=\begin{pmatrix}
    -\sigma_1(x,x_0) \\
    -\sigma_1(x,x_1) \\
    -\sigma_1(x,x_2) \\
    \vdots \\
    -\sigma_1(x,x_{n-2}) \\
    -\sigma_1(x,x_{n-1})
\end{pmatrix}, $$
and, therefore, we have
$$\Delta^{\sigma_1}_{W_n}\cong\begin{pmatrix}\begin{array}{c|ccc}
n &  & \overline{f_{k}}^{\top}  &  \\
\hline
   &    &     &    \\
f_k &   & \Delta^{\sigma_1}_{C_n}+\mathrm{Id} &  \\
    &   &    &
\end{array}    
\end{pmatrix}.$$
The orthogonality relation \eqref{eq:orthogonality} implies that $\Delta^{\sigma_1}_{W_n}$ has the following set of $(n-1)$ eigenfunctions: 
$$\left\{\begin{pmatrix}
    0 \\
    f_j
\end{pmatrix},\,\,\,j=0,1,\cdots,k-1,k+1,\cdots,n-1\right\},$$
corresponding to the eigenvalues $\mu_j(\pi)+1$, $j \neq k$, with the eigenvalues $\mu_j(\pi)$ of $\Delta_{C_n}^{\sigma_1}$ given in \eqref{eq:mujt}.
The remaining $2$-dimensional vector space orthogonal to $span\{f_0,f_1,\cdots,f_{k-1},f_{k+1},\cdots,f_{n-1}\}$ can be written as 
$$\left\{\begin{pmatrix}
    a \\
    bf_k
\end{pmatrix},\,\,\,a,b\in \mathbb{C}\right\}.$$
Solving the eigenfunction equation
$$\begin{pmatrix}
    n  & \overline{f_k}^{\top} \\
    f_k & \Delta^{\hat \sigma}_{C_n}+\mathrm{Id}
\end{pmatrix}
\begin{pmatrix}
    a \\
    bf_{k}
\end{pmatrix}
=\lambda
\begin{pmatrix}
   a \\
   bf_k
\end{pmatrix}$$
leads to
$$\begin{cases}
    (\lambda-n)a=nb,\quad &\text{(A)}\\
    a=(\lambda-\mu_k(\pi)-1)b, &\text{(B)}
\end{cases}$$
with
$$ \mu_k(\pi) = 2-2\cos\left( \frac{2k+1}{n} \pi\right). $$
Replacing $a$ with $(B)$ in $(A)$ yields
$$(\lambda-n)(\lambda-\mu_k(\pi)-1)b=nb. $$
Since $b=0$ implies $\begin{pmatrix}
    a \\
    bf_k
\end{pmatrix}= 0$, in contradiction to an eigenfunction,  we can divide by $b \neq 0$ and obtain the quadratic equation
$$\lambda^2-(n+\mu_k(\pi)+1)\lambda+\mu_k(\pi)n=0,$$
which gives the remaining two eigenvalues of $\Delta^{\sigma_1}_{W_n}$.
\end{example}

\begin{lemma}\label{lem:destimates}
\begin{itemize}
    \item[(a)] Let $n \ge 3$ be odd. Then we have
    $$ \frac{n+5 - \sqrt{(n+5)^2-16n}}{2} \ge 3-2\cos(\pi/n). $$
    \item[(b)] Let $n \ge 4$ be even. Then we have
    $$ \frac{n+3+2\cos(\pi/n) - \sqrt{(n+3+2\cos(\pi/n))^2-4n(2+2\cos(\pi/n))}}{2} \ge 3-2\cos(\pi/n). $$
\end{itemize}
\end{lemma}
\begin{proof}
The inequality in (a) is an equality in the case $n=3$.
It is also easy to verify (b) for $n=4$.  So we can restrict to the case $n \ge 5$.

Both (a) and (b) are direct consequences of the following two facts:
\begin{itemize}
\item[(i)] We have
$$ n+3+2c-\sqrt{(n+3+2c)^2-4n(2+2c)} \ge 6-4c $$
for $d \ge n$ and $c= \cos(\pi/n)$.
\item[(ii)] The map
$$ [0,1] \ni c \mapsto n+3+2c-\sqrt{(n+3+2c)^2-4n(2+2c)} $$
is monotone increasing for all $n \ge 5$. 
\end{itemize}
(i) is equivalent to
$$ (4c-1)n + 8c^2 - 12 c \ge 0 $$
for $d \ge n$ and $c=\cos(\pi/n)$. Since $n \ge 5$, we have
$$ c=\cos(\pi/n) \ge \cos(\pi/4) = \frac{\sqrt{2}}{2}. $$
So it suffices to show
$$ (4c-1)n+8c^2-12c \ge 0$$
for $n \ge 5$ and $c \in [\sqrt{2}/2,1]$. Since $4c-1 > 0$, the left hand side is increasing in $n$, 
and it suffices to show the inequality for $n=5$. This reduces the problem to show that $8c^2+8c-5$ is non-negative for $c \in [\sqrt{2}/2,1]$, which is obviously true. This finishes the proof of (i).

To prove (ii), we set $\alpha=3+2c$ and need to verify, for all $n \ge 5$, that
$$ \alpha \mapsto n+\alpha- \sqrt{(n+\alpha)^2-4n(\alpha-1)} =
n+\alpha - \sqrt{(n-\alpha)^2+4n} $$
is monotone increasing for  $\alpha \in [3,5]$. This is obviously the case, which completes the proof of the lemma.
\end{proof}

The following result states that our lower bound for $\nu(W_n)$ is sharp.

\begin{proposition} \label{prop:wheelnu}
For any wheel graph $W_n$, $n\geq 3$, we have
\begin{equation} \label{eq:nuWdnuCd}
    \nu(W_{n})=\nu(C_{n})+1 = 3-2\cos(\pi/n).
\end{equation}
\end{proposition}
\begin{proof}
We apply Theorem \ref{thm:outdegree_upper_bound} in Section \ref{sec:genresults} below to wheel graphs, by choosing $G=(V,E)$ as $W_n$ and the subgraph $G_0=(V_0,E_0)$ as the cycle $C_n$. Then we obtain the inequality:
$$ \nu(W_n) \le \nu(C_n)+1. $$
Combining this with the results in Example \ref{ex:wheel} and Example \ref{ex:cycle} completes the proof of the proposition.
\end{proof}

\subsection{Hypercubes and complete bipartite graphs}

\begin{example}[Hypercubes] We will show that $d$-dimensional hypercubes $Q^d= (K_2)^d$ have magneto-spectral heights
$$ \nu(Q^d) = d-\sqrt{d}. $$
Moreover, $Q^d$ has a unique maximal potential, up to gauge equivalence, which is given by a signature $\hat \sigma: E^{or}(Q^d) \to \{\pm 1\}$ assuming the value $-1$ on all $4$-cycles (see Proposition \ref{prop:nuhypercube} below).  
\end{example}

The vertices of the $d$-dimensional hypercube $Q^d$ can be identified with $d$-tuples $(a_1,\dots,a_d) \in (\mathbb{Z}_2)^d$, $\mathbb{Z}_2 = \{0,1\}$, and every $4$-cycle in $Q^d$ consists of the vertices $(a_1,\dots,a_d) + \mathbb{Z}_2 e_i + \mathbb{Z}_2 e_j\, ({\rm{mod}}\, 2)$ for $1 \le i < j \le d$. 
Note that, for $d \ge 2$, $Q^d$ has precisely $2^{d-2} {d \choose 2}$ different $4$-cycles and that every edge of $Q^d$ is contained in precisely $d-1$ $4$-cycles. In the arguments below, we use the following parity function $\rho: V(Q^d) \to \{ \pm 1 \}$ on the vertices of $Q^d$:
$$ \rho(a_1,\dots,a_d) = (-1)^{\sum_{j=1}^d a_j}. $$

Our first result states the existence of a particular family of potentials of $Q^d$.

\begin{lemma} \label{lem:specsignhypcube} Every hypercube $Q^d$, $d \ge 1$, admits a magnetic potential which assumes the value $-1$ on each $4$-cycle. Moreover, any two magnetic potentials of this type are gauge equivalent.  
\end{lemma}

\begin{proof} The first statement follows by Induction over $d$. Any magnetic potential on $Q^1=K_2$ satisfies trivially the required property, since there are no $4$-cycles. For general $Q^d$, $d \ge 2$, we use the fact that $Q^d=K_2 \times Q^{d-1}$, and we apply the later Proposition \ref{prop:maxregbipartite} in Subsection \ref{subsec:uppcombbounds}.

For the second statement, let $\sigma_1, \sigma_2: E^{or}(Q^d) \to S^1$ be two magnetic potentials on $Q^d$ with the required property. Since the cycle space of $Q^d$ is generated by its $4$-cycles (see \cite[Corollary 2]{P-V22}), $\sigma_1$ and $\sigma_2$ agree on all cycles. Since the cycle space is also generated by the fundamental cycles of a spanning tree, the uniquely determined potentials $\sigma_j': E^{or}(Q^d) \to S^1$, $j=1,2$, which are gauge equivalent to $\sigma_j$, respectively, and trivial on $T$, must coincide.  
Therefore, $\sigma_1$ and $\sigma_2$ are gauge equivalent. 
\end{proof}

The next proposition states that, similarly as in the case of complete graphs, all maximal potentials of hypercubes are gauge equivalent, and we have the following result regarding their magneto-spectral heights.

\begin{proposition} \label{prop:nuhypercube} The magneto-spectral height of the hypercube $Q^d$, $d \ge 2$, is given by
$$ \nu(Q^d) = d - \sqrt{d}. $$
Moreover, the maximal potentials are precisely those described in Lemma \ref{lem:specsignhypcube}. In particular, any two maximal potentials of $Q^d$ are gauge equivalent.
\end{proposition}

The proof of this statement is a consequence of Proposition \ref{prop:regbipartuppbd} given in Subsection \ref{subsec:uppcombbounds}.

\begin{proof} Note that $Q^d$ is a $d$-regular bipartite graph. Therefore, we have
$$ \nu(G) \le d - \sqrt{d}, $$
by Proposition \ref{prop:regbipartuppbd} below.
Moreover, $Q^d$ admits a magnetic potential $\sigma: E^{or}(Q^d) \to S^1$ which assumes the value $-1$ on all $4$-cycles, by Lemma \ref{lem:specsignhypcube}. An equivalent description of magnetic potentials of this type on $Q^d$ is condition \eqref{eq:sigmasigma} in Proposition \ref{prop:regbipartuppbd} which, in turn, is equivalent to $\lambda_1^\sigma(G) = d-\sqrt{d}$. Therefore, we have 
$$ d - \sqrt{d} = \lambda_1^\sigma(G) \le \nu(G) \le d-\sqrt{d} $$
for precisely this type of signatures. Since only potentials of this type are maximal, any two maximal potentials must be gauge equivalent, by Lemma \ref{lem:specsignhypcube}.
\end{proof}

\begin{example}[Complete bipartite graphs] \label{ex:complbipart} The magneto-spectral height of the $d$-regular complete bipartite graph $K_{d,d}$, $d\ge 1$, is given by
(see Corollary \ref{cor:Kddmagnetospecheight} below and its proof)
$$ \nu(K_{d,d}) = d -\sqrt{d}. $$
\end{example}

\subsection{Cospectral graphs}
\label{subsec:cospectral}

\begin{example}[Wheel graphs and cospectrality] \label{ex:wheelcospec}
It is known (see \cite{ZLY09}) that wheel graphs $W_n$, $n \neq 6$, are uniquely determined by their Laplace spectrum. However, there exists a cospectral graph $\widehat G$ for the wheel graph $W_6$. Both graphs are illustrated in Figure \ref{fig:cospectral}. We will show below that these two graphs have different magneto-spectral heights.

The matrix representations of the standard Laplacian on $W_6$ is of the form
$$ \Delta_{W_6} \cong 
\begin{pmatrix} 
6 & -1 & -1 & -1 & -1 & -1 & -1 \\
-1 & 3 & -1 & 0 & 0 & 0 & -1 \\ 
-1 & -1 & 3 & -1 & 0 & 0 & 0 \\ 
-1 & 0 & -1 & 3 & -1 & 0 & 0 \\ 
-1 & 0 & 0 & -1 & 3 & -1 & 0 \\ 
-1 & 0 & 0 & 0 & -1 & 3 & -1 \\ 
-1 & -1 & 0 & 0 & 0 & -1 & 3\end{pmatrix}, $$
and its spectrum is given by
\begin{center}
\begin{tabular}{ccccccc}
$\lambda_1$ & $\lambda_2$ & $\lambda_3$ & $\lambda_4$ & $\lambda_5$ & $\lambda_6$ & $\lambda_7$ \\
\hline
$0$ & $2$ & $2$ & $4$ & $4$ & $5$ & $7$
 \end{tabular}
\end{center}
\begin{figure}
\begin{center}
\begin{tikzpicture}[thick,scale=1]
\filldraw [black] (0,1) circle (2pt);
\filldraw [black] (1,0) circle (2pt);
\filldraw [black] (1,2) circle (2pt);
\filldraw [black] (2,1) circle (2pt);
\filldraw [black] (3,0) circle (2pt);
\filldraw [black] (3,2) circle (2pt);
\filldraw [black] (4,1) circle (2pt);
\filldraw [black] (6,0) circle (2pt);
\filldraw [black] (7,0) circle (2pt);
\filldraw [black] (7,2) circle (2pt);
\filldraw [black] (8,0) circle (2pt);
\filldraw [black] (9,0) circle (2pt);
\filldraw [black] (9,2) circle (2pt);
\filldraw [black] (10,0) circle (2pt);
\draw (2,-1) node[label=below:{\LARGE$W_6$}]{};
\draw (8,-1) node[label=below: {\LARGE$\widehat G$}]{};
\draw (0,1) node[label=left:$x_7$]{};
\draw (1,0) node[label=below:$x_6$]{};
\draw (1,2) node[label=above:$x_2$]{};
\draw (2,1) node[label=below:$x_1$]{};
\draw (3,0) node[label=below:$x_5$]{};
\draw (3,2) node[label=above:$x_3$]{};
\draw (4,1) node[label=right:$x_4$]{};
\draw (0,1) -- (1,0);
\draw (0,1) -- (2,1);
\draw (0,1) -- (1,2);
\draw (0,1) -- (1,2);
\draw (1,0) -- (3,0);
\draw (1,0) -- (2,1);
\draw (1,2) -- (3,2);
\draw (1,2) -- (2,1);
\draw (2,1) -- (3,0);
\draw (2,1) -- (4,1);
\draw (2,1) -- (3,2);
\draw (3,0) -- (4,1);
\draw (3,2) -- (4,1);
\draw (6,0) node[label=below:$x_3$]{};
\draw (7,0) node[label=below:$x_4$]{};
\draw (7,2) node[label=above:$x_1$]{};
\draw (8,0) node[label=below:$x_5$]{};
\draw (9,0) node[label=below:$x_6$]{};
\draw (9,2) node[label=above:$x_2$]{};
\draw (10,0) node[label=below:$x_7$]{};
\draw [red] (7,2) -- (6,0);
\draw [red] (7,2) -- (7,0);
\draw [red] (7,2) -- (8,0);
\draw [red] (9,2) -- (8,0);
\draw [red] (9,2) -- (9,0);
\draw [red] (9,2) -- (10,0);
\draw [blue] (6,0) -- (7,0);
\draw [blue] (6,0) -- (9,2);
\draw [blue] (7,0) -- (9,2);
\draw [blue] (10,0) -- (9,0);
\draw [blue] (10,0) -- (7,2);
\draw [blue] (9,0) -- (7,2);
\end{tikzpicture}
\end{center}
\caption{The two isospectral graphs $W_6$ and $\widehat G$\label{fig:cospectral}}
\end{figure}
\FloatBarrier
The standard Laplacian $\Delta_{\widehat G}$ on $\widehat G$ has the same spectrum, and the matrix representation of the signed Laplacian $\Delta_{\widehat G}^\sigma$ with signature $\sigma: E^{or}(\widehat G) \to \{-1,1\}$ assuming the value $+1$ on the red edges and the value $-1$ on the blue edges in Figure \ref{fig:cospectral} is of the form
$$ \Delta_{\widehat G}^\sigma \cong \begin{pmatrix} 
5 & 0 & -1 & -1 & -1 & 1 & 1 \\
0 & 5 & 1 & 1 & -1 & -1 & -1 \\ 
-1 & 1 & 3 & 1 & 0 & 0 & 0 \\ 
-1 & 1 & 1 & 3 & 0 & 0 & 0 \\ 
-1 & -1 & 0 & 0 & 2 & 0 & 0 \\ 
1 & -1 & 0 & 0 & 0 & 3 & 1 \\ 
1 & -1 & 0 & 0 & 0 & 1 & 3\end{pmatrix}. $$
The spectrum of $\Delta_{\widehat G}^\sigma$ is given by
\begin{center}
\renewcommand*{\arraystretch}{1.4}
\begin{tabular}{ccccccc}
$\lambda_1^\sigma$ & $\lambda_2^\sigma$ & $\lambda_3^\sigma$ & $\lambda_4^\sigma$ & $\lambda_5^\sigma$ & $\lambda_6^\sigma$ & $\lambda_7^\sigma$ \\
\hline
$\frac{7}{2} - \frac{\sqrt{17}}{2}$ & $\frac{9}{2}- \frac{\sqrt{33}}{2}$ & $2$ & $2$ & $4$ & $\frac{7}{2} + \frac{\sqrt{17}}{2}$ & $\frac{9}{2} + \frac{\sqrt{33}}{2}$
 \end{tabular}
\renewcommand*{\arraystretch}{1} 
\end{center}
Note that we have
$$ \nu(W_6) = 3-2\cos(\pi/6) \approx 1.2679 < 1.4384 \approx \frac{7}{2} - \frac{\sqrt{17}}{2} \le \nu(\widehat G). $$
This shows that the cospectral graphs $W_6$ and $\widehat G$ can be distinguished by the magneto-spectral height.
\end{example}

In the above example, two non-isomorphic cospectral graphs are distinguished by their magneto-spectral heights. A natural question is whether there are also examples of non-isomorphic graphs which are cospectral and have the same magneto-spectral height? This question is similar in spirit as the topic in \cite{CP23} for pairs of Riemannian manifolds, which deals with the question whether higher-dimensional twisted isospectrality of pairs of manifolds over the same base manifold implies already that they are isometric. We provide such a pair of graphs in the following example.

\begin{example}[Cospectral trees] \label{ex:cospectrees}
    There exist two non-isomorphic and cospectral trees, denoted as $T_1$ and $T_2$, which both contain $17$ vertices, as shown in Figure \ref{fig:two_non-isomorphci_and_cospectral_trees}. The 
    characteristic polynomial of both Laplacians $\Delta_{T_1}$ and $\Delta_{T_2}$ is the following:
    \begin{multline*} 
    t(t - 1)(t^2 - 3t + 1)(t^3 - 6t^2 + 8t - 1)\cdot \\ (t^10 - 22t^9 + 205t^8 - 1056t^7 + 3293t^6 - 6402t^5 + 7708t^4 - 5522t^3 + 2156t^2 - 380t + 17), 
    \end{multline*}
    and the eigenvalues of both Laplacians are given by
\begin{multline*}
 \{ 0, 0.065, 0.139, 0.382, 0.391, 0.665, 1, 1.216, 1.522, 1.746, 2.441, 2.618, 3.075, \\ 3.396, 4.115, 4.422, 4.807 \}.
\end{multline*}
    Since both graphs are trees, their magneto-spectral heights are 
    $$\nu(T_1) = \nu(T_2) = 0,$$
    respectively. This example is given in \cite[Figure 2]{HOS01} and was found by McKay \cite{McK77}.
\end{example}

\begin{figure}[h]
\begin{center}
\includegraphics[width=\textwidth]{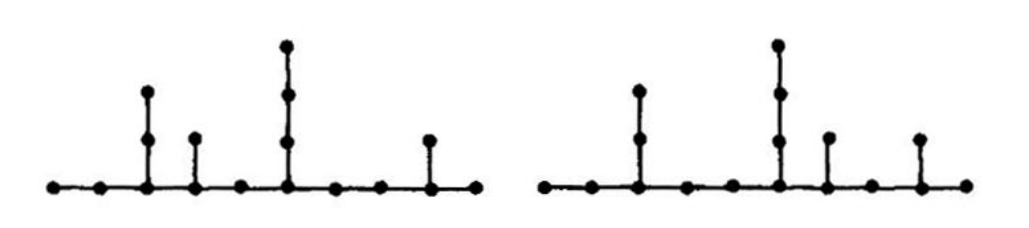}
\end{center}
\caption{Two non-isomorphic and cospectral trees $T_1$ and $T_2$
\label{fig:two_non-isomorphci_and_cospectral_trees}}
\end{figure}
\FloatBarrier

It remains to find examples of cospectral graphs whose magneto-spectral heights agree and are also \emph{non-zero}. Two different methods for the construction of cospectral graphs are due to Sunada and via Seidel switching. With regards to Sunada's method, we considered the cospectral pair $G_1,G_2$ of quotients of a Cayley graph $G$ of the semidirect product $\mathbb{Z}_8^* \ltimes \mathbb{Z}_8$, as described in \cite[Example 11.4.5]{Bu10} and \cite[Example 1]{HH99}. Both graphs have $8$ vertices and vertex degrees between $2$ and $4$, and one of them (graph $G_2$, see \cite[Figure 1]{HH99}) is a multi-graph with a double edge. We computed lower bounds for $\nu(G_j)$, $j=1,2$ via a refined grid search over the magnetic potentials, and our numerics suggest that their magneto-spectral heights are different:
$$ \nu(G_1) \ge 1.14 \quad \text{and} \quad \nu(G_2) \ge 1.25. $$
We considered also the pair $G_1,G_2$ of cospectral $4$-regular graphs with $8$ vertices, obtained via Seidel switching and illustrated in Figures 3 and 4 of \cite{BGG98}, respectively. They are multi-graphs, since they both have
two loops. Our numerics indicate again that they have different magneto-spectral heights:
$$ \nu(G_1) \ge 2.24 \quad \text{and} \quad \nu(G_2) \ge 2.39. $$
Another challenging example of cospectral graphs to investigate could be pairs strongly regular graphs with the same parameters. For example, the Shrikhande graph and the $4 \times 4$ rooks graphs (see, e.g., \cite[Example 11.4]{CLP17}) fall into this category with parameters $(16,6,2,2)$. However, since they are both $6$-regular graphs with $16$ vertices, they are likely to be too large to tackle their magneto-spectral heights numerically. 

\subsection{Some Counterexamples}
\label{subsec:counterex}

\begin{question}\label{quest:nu_multiplicity2}
The $\lambda_1$-muliplicity considerations in Examples \ref{ex:cycle} and \ref{ex:completegraph} motivates the question whether the $\lambda_1^{\hat \sigma}(G)$-eigenspace of a maximal magnetic potential $\hat \sigma$ has always dimension $\ge 2$ if $\nu(G) > 0$. Example \ref{quest:nu_multiplicity2} shows that this is wrong.
\end{question}

\begin{figure}[h]
\begin{center}
\begin{tikzpicture}[thick,scale=1]
\filldraw [black] (0,0) circle (2pt);
\filldraw [black] (2,2) circle (2pt);
\filldraw [black] (4,0) circle (2pt);
\filldraw [black] (5,2.2) circle (2pt);
\draw (5,2.2) node[label=above:$x_1$]{};
\draw (2,2) node[label=above:$x_2$]{};
\draw (0,0) node[label=below:$x_3$]{};
\draw (4,0) node[label=below:$x_4$]{};
\draw (1,1) node[label=above:$+1$]{};
\draw (2,0) node[label=below:{$\sigma_t(x_3,x_4)=e^{it}$}]{};
\draw (3,1) node[label=right:$+1$]{};
\draw (3.7,2.1) node[label=above:$+1$]{};
\draw (0,0) -- (2,2);
\draw (0,0) -- (4,0);
\draw (2,2) -- (4,0);
\draw (2,2) -- (5,2.2);
\end{tikzpicture}
\end{center}
\caption{A counterexample to Question \ref{quest:nu_multiplicity2} \label{fig:triangle_with_dangling_edge} with potential $\sigma_t$}
\end{figure}
\FloatBarrier

\begin{example}[Counterexample to Question \ref{quest:nu_multiplicity2}]\label{example:suspension}
    We consider a triangle with a dangling edge added, as shown in Figure \ref{fig:triangle_with_dangling_edge}. The magnetic potential is determined by one parameter $t\in[0,2\pi]$ and denoted as $\sigma_t$, and the Laplacian matrix is of the form
$$\Delta^{\sigma_t}\cong
\begin{pmatrix}
    1 & -1 & 0 & 0 \\
    -1 & 3 & -1 & -1 \\
    0 & -1 & 2 & -e^{it} \\
    0 & -1 & -e^{-it} & 2
\end{pmatrix}.$$
Hence, the characteristic polynomial is given by
$$(\lambda-1)(\lambda^3-7\lambda^2+12\lambda-2(1-\cos(t))).$$
The maximal smallest root appears at the choice $t=\pi$, see the eigenvalues picture in Figure \ref{fig:plot_eigenvalues}, in which case we have the characteristic polynomial as follows:
$$(\lambda-1)(\lambda^3-7\lambda^2+12\lambda-4)=(\lambda-1)(\lambda-2)(\lambda^2-5\lambda+2),$$
and $\nu(G) =\frac{5-\sqrt{17}}{2}\approx 0.438447187$. In addition, this figure shows that a maximal magnetic potential of a triangle with a dangling edge added is the anti-balanced signature $\overline{\sigma}\equiv -1$ and the smallest eigenvalue $\lambda^{\overline{\sigma}}_{1}(G)$ has multiplicity $1$.

\begin{figure}
\begin{center}
\includegraphics[width=0.49\textwidth]{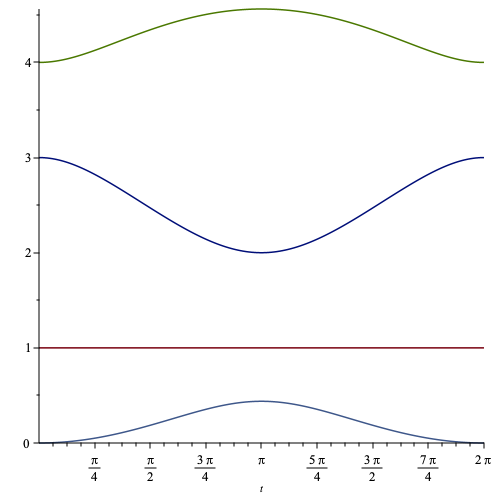}
\end{center}
\caption{The functions $t \mapsto \lambda_1^{\sigma_t}(G),\lambda_2^{\sigma_t}(G),\lambda_{3}^{\sigma_t}(G), \lambda^{\sigma_t}_4(G)$
\label{fig:plot_eigenvalues}}
\end{figure}
\FloatBarrier
\end{example}

\begin{question} \label{quest:specpot}
    Note that, in the examples $G=C_n$ and $G=K_n$, there exists a magnetic potential $\hat \sigma$ only assuming values $\pm 1$, which satisfies
    $$ \nu(G) = \lambda_1^{\hat \sigma}(G). $$
    Therefore, it is natural to ask whether, for any graph $G=(V,E)$, there exists a magnetic potential $\hat \sigma: E^{or}(G) \to \{-1,1\}$ satisfying
    $$ \lambda_1^{\hat \sigma}(G) = \nu(G). $$
    Example \ref{ex:countexquest} shows that this is not always the case. 
\end{question}    
    
\begin{example}[Counterexample to Question \ref{quest:specpot}] \label{ex:countexquest} 
    We consider a graph and a corresponding magnetic potential as illustrated in Figure \ref{fig:twotriangles}. Note that this graph is the suspension of a path of length $2$.  The magnetic potential is determined by two parameters $\alpha,\beta \in [0,2\pi]$ and denoted by $\sigma_{\alpha,\beta}$, and the Laplacian is given by the matrix
    $$ \Delta^{\sigma_{\alpha,\beta}} \cong \begin{pmatrix} 2 & -1 & 0 & -e^{i \alpha} \\ -1 & 3 & -1 & -1 \\ 0 & -1 & 2 & -e^{i \beta} \\ -e^{-i\alpha} & -1 & - e^{-i\beta} & 3 \end{pmatrix}.
    $$
    Its characteristic polynomial is given by
    $$ t^4 - 10 t^3 + 32 t^2 - (36-2(\cos\alpha+\cos\beta))t + 10 - 4(\cos\alpha+\cos\beta)-2\cos(\alpha-\beta). $$

\begin{figure}[h]
\begin{center}
\begin{tikzpicture}[thick,scale=1]
\filldraw [black] (0,3) circle (2pt);
\filldraw [black] (3,6) circle (2pt);
\filldraw [black] (3,0) circle (2pt);
\filldraw [black] (6,3) circle (2pt);
\draw (1,4.5) node[label=above:$+1$]{};
\draw (5,4.5) node[label=above:$+1$]{};
\draw (3,3) node[label=right:$+1$]{};
\draw (1,1.5) node[label=below:$e^{i \alpha}$]{};
\draw (5,1.5) node[label=below:$e^{i \beta}$]{};
\draw (3,6) node[label=above:$x_2$]{};
\draw (0,3) node[label=left:$x_1$]{};
\draw (3,0) node[label=below:$x_4$]{};
\draw (6,3) node[label=right:$x_3$]{};
\draw (3,0) -- (3,6);
\draw (0,3) -- (3,0);
\draw (3,0) -- (6,3);
\draw (0,3) -- (3,6);
\draw (3,6) -- (6,3);
\end{tikzpicture}
\end{center}
\caption{A counterexample to Question \ref{quest:specpot} \label{fig:twotriangles}}
\end{figure}

\FloatBarrier
    
    The function $(\alpha,\beta) \mapsto \lambda_1^{\sigma_{\alpha,\beta}}(G)$ is illustrated in Figure \ref{fig:plot-twotriangles}. For example, we have
    the following eigenvalues for all magnetic potentials only assuming values $\{-1,1\}$:
    \begin{align*}
    \lambda_1^{\sigma_{0,0}}(G) &= 0, \\
    \lambda_1^{\sigma_{0,\pi}}(G) &= \lambda_1^{\sigma_{\pi,0}}(G) = 0.5857864376, \\
    \lambda_1^{\sigma_{\pi,\pi}} (G)&= 0.7639320225.
    \end{align*}
      However, we have
    $$ \lambda_1^{\sigma_{2\pi/3,4\pi/3}}(G) = \lambda_1^{\sigma_{4\pi/3,2\pi/3}}(G) = 1 \stackrel{(*)}{=} \nu(G). $$
\begin{figure}
\begin{center}
\includegraphics[width=0.4\textwidth]{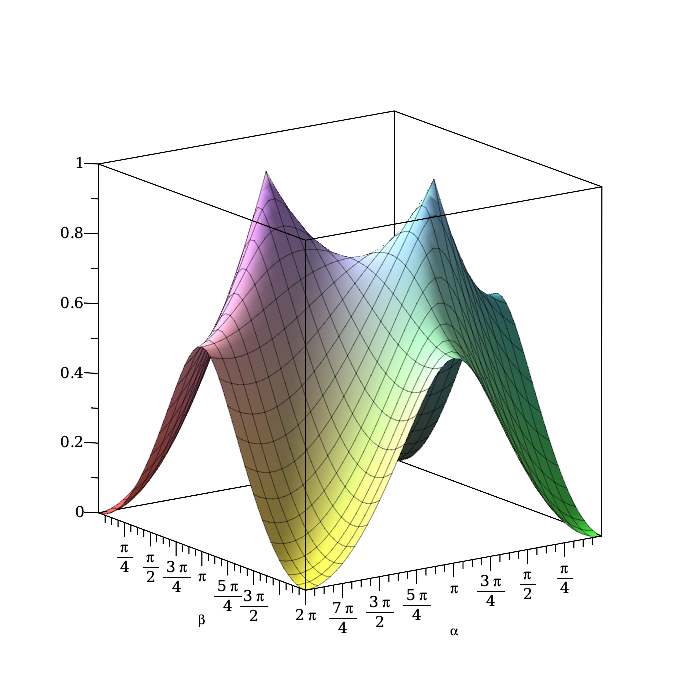}
\includegraphics[width=0.4\textwidth]{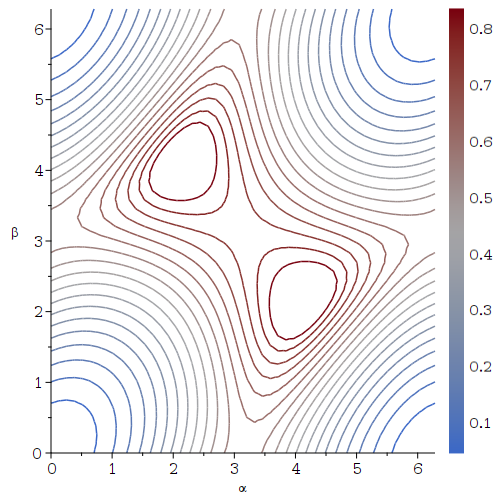}
\end{center}
\caption{The function $(\alpha,\beta) \mapsto \lambda_1^{\sigma_{\alpha,\beta}}(G)$ \label{fig:plot-twotriangles}}
\end{figure}
\FloatBarrier
  
Note that $(*)$ follows from Theorem \ref{thm:suspgraph} below. Therefore, the following choices of potentials are maximal:
    \begin{equation} \label{eq:twotrianglespotmax}
    \sigma(x_1,x_2) = \sigma(x_2,x_3) = \sigma(x_2,x_4) = 1, \,\,\,
    \sigma(x_1,x_4) = \sigma(x_4,x_3) = {\xi_3}^j
    \end{equation}
    with $j=1,2$ and $\xi_3 = e^{i \frac{2\pi}{3}}$, and none of the magnetic potentials $\sigma$ with values $\{-1,1\}$ is maximal.

\end{example}

\begin{question}\label{quest:nu_degree}
Proposition \ref{prop:combupbd} below states that $\nu(G)$ is bounded above by the average vertex degree 
$$ \frac{1}{|V|} \sum_{x \in V} d_x = \frac{2|E|}{|V|}. $$
It is therefore natural to ask, whether $\nu(G)$ is also bounded below  by the average degree of a graph $G$, that is, whether there exist a universal constant $c >0$ such that 
\begin{equation} \label{eq:nuaverdeg}
\frac{c}{|V|}\sum_{x\in V}d_x\leq \nu(G).
\end{equation}
Example \ref{ex:avverdeg} answers this question negatively.
\end{question}

\begin{example}[Counterexample to Question \ref{quest:nu_degree}] 
\label{ex:avverdeg}
    A counterexample for the left hand side of the inequality \eqref{eq:nuaverdeg} are increasing circles $C_n$, since $\nu(C_n) \to 0$ as $n \to \infty$, while the average degree of these graphs is constant and equal to $2$.
\end{example}

\begin{question}\label{quest:antibalanced_max}
    Restricting to signatures, the following question arises: Does the smallest eigenvalue $\lambda^{\overline{\sigma}}_1(G)$ with anti-balanced signature $\overline{\sigma} \equiv -1$ attain the maximum among all the smallest eigenvalues $\{\lambda^{\sigma}_1\}$ corresponding to all signatures $\sigma:E^{or}(G) \to\{\pm1\}$? A counterexample is given in Example \ref{ex:antibalanced}.
\end{question}

\begin{example}[Counterexample to Question \ref{quest:antibalanced_max}] \label{ex:antibalanced}
Let $x_0,x_1,x_2,x_3$ be the vertices of the wheel graph $W_4$ in the cycle $C_4 \subset W_4$ and $x \in W_d$ be the suspension vertex. 
We choose the following signatures/magnetic potentials:
\begin{center}
\begin{tabular}{c|c|c|c|c|c|c|c|c}
$(u,v)$ & $(x,x_0)$ & $(x,x_1)$ & $(x,x_2)$ & $(x,x_3)$ & $(x_0,x_1)$ & $(x_1,x_2)$ & $(x_2,x_3)$ & $(x_3,x_0)$ \\ \hline 
$\overline{ \sigma}(u,v)$ & $-1$ & $-1$ & $-1$ & $-1$ & $-1$ & $-1$ & $-1$ & $-1$ \\ 
$\sigma_1(u,v)$ & $1$ & $1$ & $1$ & $1$ & $-1$ & $-1$ & $1$ & $-1$ \\
$\hat\sigma(u,v)$ & $-1$ & $e^{-i \pi/4}$ & $i$ & $e^{i \pi/4}$ & $1$ & $1$ & $-1$ & $1$
\end{tabular}
\end{center}
Note that $\overline{ \sigma}$ is the anti-balanced signature, while $\sigma_1$ is a non-equivalent unbalanced signature. Maple yields the following results:
\begin{align*}
\lambda_1^{\overline{\sigma}}(W_4) &= 1, \\
\lambda_1^{\sigma_1}(W_4) &= \frac{10}{3} - \frac{\sqrt{19}}{3}\left( \sqrt{3} \cos\left(\frac{\arctan\left(\frac{330\sqrt{426}}{809}\right)}{6}\right) - \sin\left(\frac{\arctan\left(\frac{330\sqrt{426}}{809}\right)}{6}\right)
\right) \\
&\approx 1.238442816, \\
\lambda_1^{\hat \sigma}(W_4) &= 3-2\cos\left(\frac{\pi}{4}\right) = 1 + \nu(C_4) \approx 1.585786438.
\end{align*}
Since $\lambda_1^{\hat \sigma}(W_4) \ge \lambda_1^{\overline{\sigma}}(W_4)$, the maximum for $\lambda_1^\sigma(W_4)$ amongst all magnetic potentials is not assumed by the anti-balanced signature. Moreover, the unbalanced signature $\sigma_1$ provides a strictly larger first magnetic eigenvalue than the anti-balanced signature $\overline{\sigma}$.
\end{example}

\section{General results}
\label{sec:genresults}

\subsection{Upper combinatorial bounds}
\label{subsec:uppcombbounds}

We start this section with several combinatorial upper bounds of the magneto-spectral height. One of these bounds involves the \emph{first Betti number} $b_1(G)$ of a graph $G$, which is given by $|E|-|V|+k$, where $k$ is the number of connected components of $G$. Note that $b_1(G)$ is also called the \emph{dimension of the cycle space}. We also provide a sharp upper bound for regular bipartite graphs and show that graphs attaining this upper bound are closed under Cartesian products.

\begin{proposition} \label{prop:combupbd}
    Let $G=(V,E)$ be a graph. Then we have
    $$ \nu(G) \le \min\left\{ \frac{2|E|}{|V|}, \frac{4b_1(G)}{|V|}, 
    d_{\max}-1 \right\},$$
    where $d_{\max}$ is the maximal vertex of degree of $G$.
\end{proposition}

\begin{remark}
All three upper bounds in this proposition are relevant. Note that the upper bound $\frac{2|E|}{|V|}$ is the average vertex degree of the graph $G=(V,E)$. Figure \ref{fig:firstandsecondgraph} illustrates examples of graphs $G$ for which we have
$$ \frac{2|E|}{|V|} = \frac{2(2n-1)}{n} < \min\left\{\frac{4b_1(G)}{|V|}, d_{\max}-1\right\}, $$
(since $n \ge 6$, $\frac{4b_1(G)}{|V|}=4$ and $d_{\max}=n-2$)
and
$$ \frac{4b_1(G)}{|V|} = \frac{4(e-n+1)}{n} < \min\left\{\frac{2|E|}{|V|},d_{\max}-1\right\}, $$
(since $n \ge 5$, $e=|E|=2n-3$, $\frac{2|E|}{|V|}=4-\frac{6}{n}$ and
$d_{\max}=n-2$), respectively. Examples of graphs for which we have
$$ d_{\max}-1 < \min
\left\{\frac{2|E|}{|V|},\frac{4b_1(G)}{|V|}\right\}, $$
are all $d$-regular connected graphs with $d \ge 3$. 

\begin{figure}
\begin{center}
\begin{tikzpicture}[thick,scale=1]
\filldraw [black] (0,0) circle (2pt);
\filldraw [black] (1,0) circle (2pt);
\filldraw [black] (2,0) circle (2pt);
\filldraw [black] (3,0) circle (2pt);
\filldraw [black] (4,0) circle (2pt);
\filldraw [black] (4.7,0) circle (1pt);
\filldraw [black] (5,0) circle (1pt);
\filldraw [black] (5.3,0) circle (1pt);
\filldraw [black] (6,0) circle (2pt);
\filldraw [black] (7,0) circle (2pt);
\filldraw [black] (8,0) circle (2pt);
\filldraw [black] (9,0) circle (2pt);
\draw (0,0) node[label=below:$x_1$]{};
\draw (1,0) node[label=below:$x_2$]{};
\draw (2,0) node[label=below:$x_3$]{};
\draw (3,0) node[label=below:$x_4$]{};
\draw (4,0) node[label=below:$x_5$]{};
\draw (6,0) node[label=below:$x_{n-3}$]{};
\draw (7,0) node[label=below:$x_{n-2}$]{};
\draw (8,0) node[label=below:$x_{n-1}$]{};
\draw (9,0) node[label=below:$x_n$]{};
\draw (0,0) -- (1,0);
\draw (1,0) -- (2,0);
\draw (2,0) -- (3,0);
\draw (3,0) -- (4,0);
\draw (6,0) -- (7,0);
\draw (7,0) -- (8,0);
\draw (8,0) -- (9,0);
\draw (2,0) arc (0:180:1);
\draw (3,0) arc (0:180:1.5);
\draw (4,0) arc (0:180:2);
\draw (6,0) arc (0:180:3);
\draw (7,0) arc (0:180:3.5);
\draw (8,0) arc (0:180:4);
\draw (9,0) arc (0:180:4.5);
\draw [dashed] (9,0) arc (0:180:1);
\draw [dashed] (8,0) arc (0:180:1);
\end{tikzpicture}
\end{center}
\caption{ If the dashed circles are included and $n \ge 6$, the resulting graph $G$ satisfies $\frac{2|E|}{|V|} < \frac{4b_1(G)}{|V|}, d_{\max}-1$. If the dashed circles are dropped and $n \ge 5$, the resulting graph $G$ satisfies $\frac{4b_1(G)}{|V|} < \frac{2|E|}{|V|},d_{\max}-1$.
\label{fig:firstandsecondgraph}}
\end{figure}
\FloatBarrier
\end{remark}

\begin{proof}[Proof of Proposition \ref{prop:combupbd}]
The proof of $\nu(G) \le \frac{2|E|}{|V|}$ is a simple trace argument. For $x \in V$, let $d_x$ be the vertex degree of $x$. Then we have for every $\sigma: E^{or}(G) \to S^1$, 
$$ \lambda_1^\sigma(G) \le \frac{1}{|V|} \sum_{j=1}^{|V|} \lambda_j^\sigma(G) = \frac{1}{|V|} {\rm{Trace}}(\Delta^\sigma) = \frac{1}{|V|} \sum_{x \in V} d_x = \frac{2|E|}{|V|}. $$

The upper bound 
$\nu(G) \le \frac{4b_1(G)}{|V|}$ is a consequence of the more general Theorem \ref{thm:outdegree_upper_bound}(a) in the next subsection, where we choose $G_0=G$, which implies $d_x^{out}=0$ for any $x\in G$.

Finally, the upper bound $d_{\max}-1$ is a direct consequence of \eqref{eq:nuGmindxdy} in Lemma \ref{lem:upbdxdy}.

\end{proof}

For regular bipartite graphs, we have the following upper bound.

\begin{proposition} \label{prop:regbipartuppbd} Let $G=(V,E)$ be a regular bipartite graph of vertex degree $d$. Then we have
\begin{equation} \label{eq:dsqrtd}
\nu(G) \le d-\sqrt{d}. 
\end{equation}
Moreover, we have the following equivalences for potentials $\sigma: E^{or}(G) \to S^1§$ in $d$-regular bipartite graphs:
\begin{itemize}
\item[(i)] $\lambda_1^\sigma(G) = d-\sqrt{d}$, that is, $\sigma$ is a maximal potential;
\item[(ii)] $(A^\sigma)^2 = d \cdot {\rm{Id}}$;
\item[(iii)] $\Delta^\sigma$ has precisely two eigenvalues $d \pm \sqrt{d}$, each with multiplicity $|V|/2$;
\item[(iv)] for any distance-$2$ vertices $x,y \in V$:
\begin{equation}
\label{eq:sigmasigma} 
\sum_{z: x \sim z \sim y} \sigma(x,z)\sigma(z,y) = 0. 
\end{equation}
\end{itemize}
In the above equivalences, the Hermitian \emph{magnetic adjacency matrix} $A^\sigma$ is defined to be $A^\sigma = d \cdot {\rm{Id}} - \Delta^\sigma$ with respect to any enumeration of the vertices.  
\end{proposition}

\begin{proof} Let $G=(V,E)$ be regular and bipartite and $\sigma: E^{or}(G) \to S^1$ be a potential. 
We first prove the upper bound \eqref{eq:dsqrtd}. 
The following arguments uses the matrix representation of the Laplacian $\Delta^\sigma$ (with respect to a fixed enumeration of the vertices). Let $A^\sigma$ be the Hermitian magnetic adjacency matrix
and 
$$ \lambda_1(A^\sigma) \le \lambda_2(A^\sigma) \le \cdots \le \lambda_n(A^\sigma) $$
be its eigenvalues with $n=|V|$. By the bipartiteness of $G$, we have 
\begin{equation} \label{eq:Aomegasymm} 
\lambda_{n-j+1}(A^\sigma) = - \lambda_j(A^\sigma) 
\end{equation}
for all $1 \le j \le n$. In particular, we have $\lambda_n(A^\sigma) \ge 0$.
Since the diagonal entries of $(A^\sigma)^2$ are the degrees of the respective vertices, we have
\begin{equation} \label{eq:ndnlambdan} 
n \cdot d = {\rm{trace}}\left[(A^\sigma)^2\right] = \sum_{j=1}^n \lambda_j((A^\sigma)^2), 
\end{equation}
and therefore, using the symmetry \eqref{eq:Aomegasymm},
$$ n \cdot d \le n \lambda_n((A^\sigma)^2). $$
This implies
$\lambda_n(A^\sigma) \ge \sqrt{d}$, which leads to 
$$ \lambda_1^\sigma(G) = d - \lambda_n(A^\sigma) \le d - \sqrt{d}. $$
Since $\sigma$ was arbitrary, this completes the proof of \eqref{eq:dsqrtd}.

Now we prove the equivalences in the proposition:
\begin{description}
\item[\bf(i) $\Rightarrow$ (ii)] It follows from $\lambda_1^\sigma(G) = d-\sqrt{d}$ that
$\lambda_{\max}(A^\sigma) = \sqrt{d}$. Since $G$ is bipartite, we have $\lambda_{\min}(A^\sigma)= -\sqrt{d}$, and therefore $\lambda_j((A^\sigma)^2) \le d$ for all $j \in \{1,\dots,n\}$. By \eqref{eq:ndnlambdan}, this implies $\lambda_j((A^\sigma)^2) = d$ for all $j$, which means (ii).
\item[\bf(ii) $\Rightarrow$ (iii)] $(A^\sigma)^2=d \cdot {\rm{Id}}$ implies, together with bipartiteness of $G$, that $A^\sigma$ has precisely the two eigenvalues $\pm \sqrt{d}$ with the same multiplicity $|V|/2$. (iii) follows then directly from $\Delta^\sigma = D - A^\sigma = d \cdot {\rm{id}} - A^\sigma$.
\item[\bf(iii) $\Rightarrow$ (i)] This implication is trivial.
\item[\bf(ii) $\Leftrightarrow$ (iv)] Since any off-diagonal entry of $(A^\sigma)^k$ coincides with the sum of the potentials of all walks of length $k$ between the corresponding vertices, the only non-vanishing off-diagonal entries of $(A^\sigma)^2$ of a bipartite graph correspond to pairs of distance-$2$ vertices. Therefore, condition \eqref{eq:sigmasigma} for all distance-$2$ vertices in a $d$-regular bipartite graph is equivalent to $(A^\sigma)^2 = d \cdot {\rm{Id}}$.
\end{description}
\end{proof}

\begin{proposition} \label{prop:maxregbipartite}
  Let $G_j = (V_j,E_j)$, $j=1,2$, be two regular bipartite graphs of vertex degrees $d_j$, respectively, satisfying
  $$ \nu(G_j) = d_j - \sqrt{d_j}. $$
  Then the Cartesian product $G = G_1 \times G_2$ is also a $d$-regular bipartite graph with $d=d_1+d_2$ and
  $$ \nu(G) = d - \sqrt{d}. $$
\end{proposition}

\begin{proof}
Let $G_j$, $j=1,2$ be as in the proposition and $\sigma_j: E^{or}(G_j) \to \{\pm 1\}$ be two maximal potentials satisfying \eqref{eq:sigmasigma} in Proposition \ref{prop:regbipartuppbd}, that is, we have for any distance-2 vertices in $V_j$:
$$ \sum_{z: x \sim z \sim y} \sigma_j(x,z)\sigma_j(z,y) = 0. $$
Let $V_1 = V_1^{+1} \cup V_1^{-1}$ be a vertex partition of the bipartite graph $G_1$ such that adjacent vertices of $G_1$ lie in different subsets.  

Now, we construct the following potential $\hat \sigma: E^{or}(G) \to \{ \pm 1 \}$ on the Cartesian product $G=G_1 \times G_2$: 
\begin{align}
\hat \sigma((x_1,y),(x_2,y)) &= \sigma_1(x_1,x_2) \quad \text{for $x_1 \sim_{G_1} x_2$ and $y \in V_2$ arbitrary,} \\
\hat \sigma((x,y_1),(x,y_2)) &= \epsilon \cdot \sigma_2(y_1,y_2) \quad \text{for $x \in V_1^\epsilon$, $\epsilon \in \{\pm 1\}$ and $y_1 \sim_{G_2} y_2$.}
\end{align}
To verify that $\hat \sigma$ on $G$ satisfies also \eqref{eq:sigmasigma}, we distinguish the following  three cases for any pair $(x_1,y_1), (x_2,y_2)$ of distance-$2$ vertices in $G$:
\begin{itemize}
\item[(i)] $d_{G_1}(x_1,x_2)=2$ and $y=y_1=y_2$; here all $2$-paths stay in $V_1 \times \{y\}$;
\item[(ii)] $x=x_1=x_2$ and $d_{G_2}(y_1,y_2)=2$; here all $2$-paths stay in $\{x\} \times V_2$;
\item[(iii)] $x_1 \sim_{G_1} x_2$ and $y_1 \sim_{G_2} y_2$ and the only $2$-paths between $(x_1,y_1)$ and $(x_2,y_2)$ are $(x_1,y_1) \sim_G (x_1,y_2) \sim_G (x_2,y_2)$ and $(x_1,y_1) \sim_G (x_1,y_1) \sim_G (x_2,y_2)$.
\end{itemize}
It is easy to see that $\hat \sigma$ satisfies \eqref{eq:sigmasigma} in all three cases (i)-(iii). 
It follows from Proposition \ref{prop:regbipartuppbd} that $\hat \sigma$ is a maximal potential of $G$ satisfying
$$ \nu(G) = \lambda^{\hat \sigma}(G) = d-\sqrt{d}. $$
\end{proof}

\subsection{Upper bounds in terms of induced subgraphs}

In this section, we present the following relations between $\nu(G)$ and properties of  subgraphs $G_0$. 

\begin{theorem}\label{thm:outdegree_upper_bound}
Let $G=(V,E)$ be a graph and $G_0=(V_0,E_0)$ be a non-empty induced subgraph of $G$. Let $d_x^{out}:=|\{y\in V \setminus V_0:y\sim x \}|$ be the \emph{out-degree} of $x\in V_0$. Then $\nu(G)$ has the following upper bounds.
\begin{itemize}
\item[(a)] We have    \begin{align}\label{eq:outdegree_upper_bound}
    \nu(G)\leq\frac{4b_1(G_0)}{|V_0|}+\frac{1}{|V_0|}\sum_{x\in V_0}d_{x}^{out}.
    \end{align}
\item[(b)] We have 
\begin{equation} \label{eq:nuGnuG0weak} 
\nu(G) \leq \nu(G_0) + \max_{x \in V_0} d_x^{out}.
\end{equation}
\end{itemize}
If there exists an eigenfunction $f \in C(V_0,\C)$ of the smallest eigenvalue of a maximal magnetic potential of $G_0$ with constant modulus $|f|$, then we have
\begin{equation} \label{eq:nuGnuG0} 
\nu(G) \leq \nu(G_0) + \frac{1}{|V_0|}\sum_{x\in V_0}d_{x}^{out}. \end{equation}
\end{theorem}

Note that the upper bound \eqref{eq:nuGnuG0} is generally stronger than both \eqref{eq:outdegree_upper_bound} and \eqref{eq:nuGnuG0weak}, since $\frac{4b_1(G_0)}{|V_0|} \ge \nu(G_0)$ and $\max_{x \in V_0} d_x^{out} \ge \frac{1}{|V_0|}\sum_{x \in V_0}d_x^{out}$. Moreover, the upper bound \eqref{eq:nuGnuG0} holds whenever we choose $G_0$ to be a tree or a cycle in $G$. In the special case that $G_0$ runs through all edges, we recover the upper bound \eqref{eq:nuGmindxdy} in Lemma \ref{lem:upbdxdy}.

\begin{proof}
    In the two statements (a) and (b) of the theorem, we choose different functions and calculate the corresponding Reighley quotients, respectively.
    \begin{itemize}
        \item [(a)] We first consider the case that $G_0$ is connected. Assume $T\subset G_0$ is a spanning tree of $G_0=(V_0,E_0)$. We choose $x_0$ as a base point of $T$. Let $\hat \sigma:E^{or}(G) \to S^1$ be a maximal magnetic potential of $G$. Then we choose the test function $f$ as follows: 
    $$f(z)=\begin{cases}
       1, & \text{if $z=x_0$,} \\\hat \sigma(x_0,x_1)^{-1}\hat \sigma(x_1,x_2)^{-1}\cdots\hat \sigma(x_{n-1},z)^{-1}, & \text{if $z\in V_0$,}\\
       0, & \text{if $z\notin V_0$.}
    \end{cases},$$
    where, in the second case, $x_0\sim x_1\sim \cdots x_{n-1}\sim z$ is the unique path from $x_0$ to $z$ on the spanning tree $T$. Then we have
    \begin{align*}\nu(G)=\lambda^{\hat \sigma}_1(G)\leq\RR^{\hat \sigma}(f)&=\frac{\sum_{\{x,y\}\in E}|f(x)-\hat\sigma(x,y)f(y)|^2}{\sum_{x\in V}|f(x)|^2}\\
    &=\frac{4(|E_0|-|V_0|+1)+\sum_{x\in v_0}d_x^{out}}{|V_0|}=\frac{4b_1(G_0)}{|G_0|}+\frac{1}{|V_0|}\sum_{x\in V_0}d_x^{out}.
    \end{align*}
    In the case that $G_0$ has more than one connected component, we choose base points in each of its components and argue accordingly.
    \item [(b)] Assume $\hat \sigma$ is a maximal magnetic potential on $G$. Denote the restriction of $\hat \sigma$ to the subgraph $G_0$ by $\hat \sigma_0$. Let $f$ be an eigenfunction on $G_0$ corresponding to the smallest eigenvalue $\lambda^{\hat \sigma_0}_1(G_0)$ and denote its trivial extension to $G$ by $\tilde{f}$. Then the Reighley quotient of $\tilde{f}$ is given as follows:
    \begin{align}
        \nu(G)&=\lambda_1^{\hat \sigma}(G) \leq \RR^{\hat \sigma}(\tilde{f}) \label{eq:corcase}\\ &=\frac{\sum_{\{x,y\}\in E}|\tilde{f}(x)-\hat \sigma(x,y)\tilde{f}(y)|^2}{\sum_{x\in V}|\tilde{f}(x)|^2}\notag\\
        &=\frac{\sum_{\{x,y\}\in E_0}|f(x)-\hat \sigma_0(x,y)f(y)|^2}{\sum_{x\in V_0}|f(x)|^2}+\frac{\sum_{x\in V_0}\sum_{y\notin V_0,x\sim y}|f(x)|^2}{\sum_{x\in V_0}|f(x)|^2}\label{eq:reighley_quotient_eigenfunction}\\
        &\leq\lambda^{\hat \sigma_0}_1(G_0)+\max_{x\in V_0}d^{out}_x\leq \nu(G_0)+\max_{x\in V_0}d_x^{out}.\notag
    \end{align}
    Moreover, in the special case where the norm $|f|$ is constant on $G_0$, the second term in \eqref{eq:reighley_quotient_eigenfunction} becomes the average out-degree, and we obtain the upper bound given in \eqref{eq:nuGnuG0}.
    \end{itemize}
\end{proof}

\begin{corollary} \label{cor:nudmin} For any finite graph $G=(V,E)$ without isolated vertices, we have
\begin{equation} \label{eq:nudmin} 
\nu(G) < d_{\rm{min}}. 
\end{equation}
In particular, any graph $G$ with a leaf satisfies
\begin{equation} \label{eq:nuleaf} 
\nu(G) < 1. 
\end{equation}
\end{corollary}

\begin{proof}
  Let $x_0 \in V$ be a vertex of minimal degree. By choosing $G_0$ to be this vertex, it follows from \eqref{eq:nuGnuG0} that  
  $$ \nu(G) \le d_{x_0}^{out} = d_{\rm{min}}. $$
  Moreover, we have strict inequality in \eqref{eq:corcase} of the proof of \eqref{eq:nuGnuG0}, since the trivial extension $\tilde f$ cannot satisfy the eigenvalue equation at any neighbour of $x_0$ and, therefore, we obtain the strict inequality in \eqref{eq:nudmin}. 
\end{proof}

\subsection{Upper and lower Ramanujan type bounds}
\label{subsec:ramanujan}

A connected $d$-regular graph $G=(V,E)$ is called Ramanujan, if all of its eigenvalues $\lambda \not\in \{0,2d\}$ lie within the (Ramanujan) interval $[d-2\sqrt{d-1},d+2\sqrt{d-1}]$. This interval is the (continuous) $L^2$-Laplace spectrum of the $d$-regular tree. The following upper bound is derived in the spirit of the proof of the Alon-Boppana Theorem (see \cite{N91}). Recall that the girth of a graph is the length of a shortest cycle in the graph.
    
\begin{proposition} \label{prop:alonbopp} Let $G=(V,E)$ be a connected graph with $d_{\max} \ge 2$ and a vertex $x_0 \in V$ which is not contained in a cycle of length $< 2k+1$, $k \in \mathbb{N}$. Then we have
$$ \nu(G) \le d_{\max} - 2\sqrt{d_{\max}-1} + \frac{2 \sqrt{d_{\max}-1}-1}{k}. 
$$
\end{proposition}

\begin{proof}
The condition on the vertex $x_0 \in V$ implies that the subgraph induced by the vertices of the ball $B_k(x_0)$ of radius $k$ about $x_0$ is a tree. For simplicity, we denote the sphere $S_i(x_0)$ of radius $i$ about $x_0$ henceforth by $S_i$. Let $f: V \to \R$ be defined as follows: $f(x) = F_i$ for all $x \in S_i(x_0)$ for $i \le k$, where
$$ F_0 = F_1 = 1, \quad F_i = \frac{1}{(\sqrt{d-1})^{i-1}}, $$
and $f(x) = 0$, otherwise. Let $\sigma: E^{or}(G) \to S^1$ be an arbitrary magnetic potential. Since this potential can be gauged to become trivial on all edges of $B_k(x_0)$, we can assume, without loss of generality, that $\sigma$ is trivial on these edges, and the Rayleigh quotient of $f$ is given as follows
$$ \RR(f) = \frac{\sum_{i=1}^{k-1} |E(S_i,S_{i+1})| \cdot (F_{i+1}-F_i)^2 + |E(S_k,S_{k+1})| \cdot F_k^2}{\sum_{i=0}^k |S_i| \cdot F_i^2}. $$
The denominator of $\RR(f)$ is bounded below by
\begin{equation} \label{eq:crucial0} 
\sum_{i=1}^k |S_i| \cdot F_i^2 = \sum_{i=1}^k \frac{|S_i|}{(d_{\max}-1)^{i-1}}. 
\end{equation}
The first term of the numerator can be estimated from above as follows:
\begin{align}
   \sum_{i=1}^{k-1} |E(S_i,S_{i+1})| \cdot (F_{i+1}-F_i)^2 &\le \sum_{i=1}^{k-1} |S_i| (d_{\max}-1) \left( \frac{1}{(\sqrt{d_{\max}-1})^i} - \frac{1}{(\sqrt{d_{\max}-1})^{i-1}}\right)^2 \nonumber \\
   &= (d_{\max}-2\sqrt{d_{\max}-1}) \sum_{i=1}^{k-1} \frac{|S_i|}{(d_{\max}-1)^{i-1}} \label{eq:crucial1}
\end{align}
The second term of the numerator can be estimated from above as follows:
\begin{align}
|E(S_k,S_{k+1})| \cdot F_k^2 &\le  |S_k| (d_{\max}-1) \cdot \frac{1}{(d_{\max}-1)^{k-1}} \nonumber\\
&= \frac{(d_{\max}-2\sqrt{d_{\max}-1})\cdot|S_k|}{(d_{\max}-1)^{k-1}}
+ \frac{(2\sqrt{d_{\max}-1}-1)\cdot|S_k|}{(d_{\max}-1)^{k-1}} \nonumber\\
&\le \frac{(d_{\max}-2\sqrt{d_{\max}-1})\cdot|S_k|}{(d_{\max}-1)^{k-1}}
+ (2\sqrt{d_{\max}-1}-1) \cdot \frac{1}{k} \sum_{i=1}^k \frac{|S_i|}{(d_{\max}-1)^{i-1}} \label{eq:crucial2}.
\end{align}
In the derivation of \eqref{eq:crucial2}, we applied the inequality
$$ |S_{i+1}| \le (d_{\max}-1)|S_i|$$ 
recursively and applied averaging.

The total numerator, that is, the sum of \eqref{eq:crucial1} and \eqref{eq:crucial2}, amounts to
$$ \left( d_{\max} - 2 \sqrt{d_{\max}-1} + \frac{2\sqrt{d_{\max}-1}-1}{k}\right) \sum_{i=1}^k \frac{|S_i|}{(d_{\max}-1)^{i-1}}. $$
Combining this with the denominator estimate \eqref{eq:crucial0} above, we end up with the estimate
$$ \RR^\sigma(f) \le d_{\max} - 2 \sqrt{d_{\max}-1} + \frac{2\sqrt{d_{\max}-1}-1}{k}. $$
\end{proof}

\begin{remark} In the case of $d$-regular bipartite graphs $G$, the upper bound $d-\sqrt{d}$ in Proposition \ref{prop:regbipartuppbd} is weaker than the upper Alon-Boppana type bound 
$$ d-2\sqrt{d-1}+\frac{2\sqrt{d-1}-1}{k} $$
in Proposition \ref{prop:alonbopp}, whenever $k \ge 2$, which means that the girth of $G$ is $\ge 6$. 
\end{remark}

Next we present a lower bound of the magneto-spectral height which, in the case of $d$-regular graphs, reduces to a lower Ramanujan bound.

\begin{proposition} \label{prop:MSS}
Let $G=(V,E)$ be a finite connected graph with $d_{\rm{min}} \ge 2$. Then we have
\begin{equation} \label{eq:ramadmindmax} 
d_{\min} - 2 \sqrt{d_{\max}-1} \le \nu(G) \le \min_{x \sim y} \frac{d_x+d_y}{2} - 1, 
\end{equation}
and for $d$-regular graphs
$$ d-2\sqrt{d-1} \le \nu(G) \le d-1, $$
where the upper bound is assumed for the complete graph $K_{d+1}$.
\end{proposition}

\begin{proof} 
The upper bound is just \eqref{eq:nuGmindxdy} in Lemma \ref{lem:upbdxdy}. For the lower bound, we apply Theorem 5.3 in \cite{MSS15}, which states that 
there exists a signature $\sigma^*: E^{or}(G) \to \{-1,1\}$ with
$$ \lambda_{\rm{max}}(A_G^{\sigma^*}) \le \rho(A_{\tilde G}), $$
where $A_G^\sigma = D - \Delta_G^\sigma$ and $D: V \to \mathbb{N}$ is the vertex degree operator, $A_{\tilde G}$ is the adjacency operator of the universal covering $\tilde G$ of $G$, and $\rho$ denotes the spectral radius. Note that the tree $\tilde G$ can be injectively embedded into the infinite $d_{\max}$-regular tree $T$ with $\rho(T) = 2\sqrt{d_{\max}-1}$, and we have
$$ \rho(A_{\tilde G}) = \sup_{f\in C(\tilde G,\C) \atop f \neq 0} | \mathcal{R}_{\tilde G}(f)| \le \sup_{\tilde f \in C(T,\C) \atop \tilde f \neq 0} | \RR_T(\tilde f) | = \rho(T), $$
since $\RR_{\tilde G}(f) = \RR_T(\tilde f)$ for the trivial extension $\tilde f \in C(T,\C)$ of any non-zero $f \in C(\tilde G,\C)$. This implies
\begin{align*}
\nu(G) \ge \lambda_1(\Delta_G^{\sigma^*}) = \lambda_1(D-A_G^{\sigma^*}) &\ge d_{\min} - \lambda_{\max}(A_G^{\sigma^*}) \\ &\ge d_{\min} - \rho(A_{\tilde G}) \ge d_{\min} - 2 \sqrt{d_{max}-1}.
\end{align*}
\end{proof}

\begin{remark} The proof of the previous proposition requires an upper bound of the spectral radius of the adjacency operator of the universal covering of a finite graph. For an explicit formula of this spectral radius, we refer readers to \cite[Theorem 1.1]{GVK22} and for a lower bound, to \cite[Corollary 2]{Ji19}.
\end{remark}

\begin{remark}
The previous two propositions imply, in particular, that we have for any infinite family $G_n$ of connected, increasing $d$-graphs, $d \ge 3$, with ${\rm{girth}}(G_n) \to \infty$, that the sequence $\nu(G_n)$ is convergent with 
$$ \lim_{n \to \infty} \nu(G_n) = d-2\sqrt{d-1}, $$
which is the lower end of the Ramanujan interval.
\end{remark}

\subsection{Lower bounds in terms of diameter}

The lower diameter bounds in this subsection and their proofs are inspired by \cite[Theorem 4.2]{ChL24} and \cite[Theorem 1.1]{ChL24}, respectively. 

\begin{proposition} \label{prop:nudiamvol}
  Let $G=(V,E)$ be a connected graph. Then we have
  $$ 
  \nu(G) \ge \inf_{\sigma: E^{or}(G) \to \{\pm 1\} \atop \text{unbalanced}} \lambda_1^\sigma(G) \ge \frac{1}{({\rm{diam}(G)+1})|V|}.
  $$
\end{proposition}

Note that the order of this inequality is sharp, since we have for cycles $\nu(C_n) \sim \frac{1}{n^2}$ and ${\rm{diam}}(C_n) \sim n$ and $|C_n| = n$.

\begin{proof}
  Let $\sigma: E^{or}(G) \to \{\pm 1\}$ be an unbalanced magnetic potential and $f: V \to \R$ be an eigenfunction to the eigenvalue $\lambda_1^\sigma(G)$. Let $x_0 \in V$ be a vertex satisfying
  $$ |f(x_0)| \ge f(x) \quad \text{for all $x\in V$. } $$
  Without loss of generality, we can assume that $f(x_0)=1$, by rescaling. This implies that
  $$ \sum_{x \in V} |f(x)|^2 \le |V|. $$
  For any path $x_0 \sim x_1 \sim \cdots \sim x_t$, we have, by Cauchy-Schwarz,
  \begin{multline*}
  \sum_{j=0}^{t-1} |f(x_j) - \sigma(x_j,x_{j+1}) f(x_{j+1})|^2 = |f(x_0)-\sigma(x_0,x_1)f(x_1)|^2 \\ + |\sigma(x_0,x_1)f(x_1)-\sigma(x_0,x_1)\sigma(x_1,x_2)f(x_2)|^2 + \cdots \\ + |\sigma(x_0,x_1)\cdots\sigma(x_{t-2},x_{t-1})f(x_{t-1}) - \sigma(x_0,x_1)\cdots\sigma(x_{t-1},x_t)f(x_t)|^2 \ge \\
  \frac{1}{t} | f(x_0) - \sigma(x_0,x_1)\cdots\sigma(x_{t-1},x_t)f(x_t)|^2. 
  \end{multline*}

  {\bf{Claim:}} There exists a path $x_0 \sim x_1 \sim \cdots \sim x_s$ of length $s \le {\rm{diam}}(G)+1$, such that 
  $$ \sum_{j=0}^{s-1} |f(x_j) - \sigma(x_j,x_{j+1}) f(x_{j+1})|^2 \ge \frac{1}{{\rm{diam}}(G)+1}. $$

  The claim is proved as follows: In the case that there exists $x \in V$ with $f(x)=0$, then we choose a path of length $t \le {\rm{diam}}(G)$ from $x_0$ to $x=x_t$, and this inequality yields
$$ \sum_{j=0}^{t-1} |f(x_j)-\sigma(x_j,x_{j+1})f(x_{j+1})|^2 \ge \frac{1}{t} |f(x_0)|^2 \ge  \frac{1}{{\rm{diam}}(G)}. $$
Now we consider the case when $f(x) \neq 0$ for all $x \in V$. Since $\sigma$ is unbalanced, there must exist an edge $x \sim y$ with
$$ f(x)\sigma(x,y)f(y) < 0, $$
for, otherwise, $\tau(x):=f(x)/|f(x)|$ would be a switching function such that $\sigma^\tau$ is the trivial magnetic potential. Now we choose a path of length $t \le {\rm{diam}}(G)$ from $x_0$ to one of $x,y$, which does not contain the other vertex. Without loss of generality the end point of this path, denoted by $x_t$, is $x$. If 
$$ \sigma(x_0,x_1)\cdots\sigma(x_{t-1},x_t)f(x_t) \le 0, $$
then we obtain 
$$ \sum_{j=0}^{t-1} |f(x_j)-\sigma(x_j,x_{j+1})f(x_{j+1})|^2 \ge \frac{1}{t} |f(x_0)|^2 \ge \frac{1}{{\rm{diam}}(G)}. $$
Otherwise, we conclude that
$$ \sigma(x_0,x_1)\cdots\sigma(x_{t-1},x_t)\sigma(x_{t},y)f(y) \le 0, $$
and we obtain, by applying the above inequality again for the extended path from $x_0$ to $y=x_{t+1}$ via $x=x_t$,
$$ \sum_{j=0}^t |f(x_j)-\sigma(x_j,x_{j+1})f(x_{j+1})|^2 \ge \frac{1}{t+1} |f(x_0)|^2 \ge \frac{1}{{\rm{diam}}(G)+1}.$$
This finishes the proof of the claim.

The claim implies that 
$$ \lambda_1^\sigma(G) = \RR(f) \ge \frac{1}{{\rm{diam}}(G)+1} \cdot \frac{1}{\sum_{x \in V} |f(x)|^2} \ge \frac{1}{({\rm{diam}(G)+1})|V|}. $$
\end{proof}

The next result holds for graphs with non-negative Bakry-\'Emery curvature. Before we state it, let us briefly introduce this curvature notion. For $f,g: V \to \C$ and $\sigma: E^{or}(G) \to S^1$, we consider the following bilinear forms:
\begin{align*}
    2 \Gamma^\sigma(f,g)(x) &= - \Delta(f\bar g)(x) + f(x) \cdot \overline{\Delta^\sigma g(x)} + \Delta^\sigma f(x) \cdot \overline{ g(x)} \\
    &= \sum_{y: y \sim x} \left( f(x)-\sigma(x,y)f(y) \right) \overline{\left(g(x)-\sigma(x,y)g(y) \right)}, \\
    2 \Gamma_2^\sigma(f,g)(x) &= \Delta \Gamma^\sigma(f,g)(x) - \Gamma^\sigma(f,\Delta^\sigma g)(x) - \Gamma^\sigma(\Delta^\sigma f,g)(x).
\end{align*}
We say that a graph $G$ satisfies the \emph{curvature-dimension inequality $CD^\sigma(K,n)$} at a vertex $x \in V$, if we have for all functions $f: V \to \C$,
\begin{equation} \label{eq:CD-ineq} 
\Gamma_2^\sigma(f,f)(x) \ge \frac{1}{n} (\Delta^\sigma f(x))^2 + K \Gamma^\sigma(f,f)(x) \quad \text{for all functions $f \in C(V,\mathbb{C})$.} 
\end{equation}
For $n=\infty$, the first right hand term in \eqref{eq:CD-ineq} disappears, and we define the \emph{$\sigma$-curvature} $K^\sigma_\infty(x)$ of the vertex $x$ to be the supremum of all $K \in \mathbb{R}$ satisfying \eqref{eq:CD-ineq} for $n=\infty$. 

We say that $G$ satisfies $CD^\sigma(K,n)$ globally, if it satisfies $CD^\sigma(K,n)$ at all vertices. In the special case of the trivial signature $\sigma_0$, we use the simplified notation that $G$ satisfies $CD(K,n)$ instead of $CD^{\sigma_0}(K,n)$. This curvature notion was originally introduced in \cite{B85,BE85}. The modification for magnetic potentials was given in \cite{LMP19}. 

The $CD^\sigma$-condition at a vertex $x \in V$ is a local condition, and it is completely determined by the combinatorial structure and the magnetic potential on the incomplete $2$-ball of $x$. The incomplete $2$-ball $\mathring{B}_2(x)$ is the induced subgraph of the $2$-ball $B_2(x)$ with all spherical edges within the $2$-sphere $S_2(x)$ removed.

\begin{proposition} \label{prop:nudiam} Let $G=(V,E)$ be a connected, non-bipartite, triangle-free graph satisfying $CD(0,\infty)$. Then we have
$$ \nu(G) \ge \frac{C}{{\rm{diam}}^2(G)} $$
with a universal constant $C$.
\end{proposition}

\begin{proof}
{\bf{Step 1:}} In this step, we prove the following: If $G=(V,E)$ is a connected graph satisfying $CD^\sigma(0,\infty)$ for an unbalanced signature $\sigma: E^{or}(G) \to \{-1,1\}$, then
$$ \lambda_1^\sigma(G) \ge \frac{C}{{\rm{diam}}^2(G)} $$
for a universal constant $C$.
Since $\sigma$ is unbalanced, we know that $\lambda_1^\sigma=\lambda_1^\sigma(G) > 0$.  Let $f: V \to \R$ be an eigenfunction of $\Delta^\sigma$ to the eigenvalue $\lambda_1^\sigma$. Let
$$ g(x) = 2 \Gamma^\sigma f(x) + \alpha \lambda_1^\sigma f^2(x) $$
and $v \in V$ be such that
$$ g(v) = \max_{x \in V} g(x). $$
Then we have 
\begin{align*} 
0 \le \Delta g(v) &= 2\Delta \Gamma^\sigma f(v) + \alpha \lambda_1^\sigma \Delta f^2(v) \\
&= 2\Delta \Gamma^\sigma f(v) + 2 \alpha \lambda_1^\sigma ( f(v) \Delta^\sigma f(v)-\Gamma^\sigma f(v)).
\end{align*}
On the other hand, we conclude from the $CD^\sigma(0,\infty)$-condition
$$ 4 \Gamma_2^\sigma(f)(v) = - (2\Delta \Gamma^\sigma f)(v) + 4 \Gamma^\sigma(f,\Delta^\sigma f)(v) \ge 0. $$
Combining both results yields
\begin{align*} 
0 &\le 4 \Gamma^\sigma(f,\Delta^\sigma f)(v) + 2 \alpha \lambda_1^\sigma ( f(v) \Delta^\sigma f(v)-\Gamma^\sigma f(v))  \\
&= 2 \left( 2\lambda_1^\sigma \Gamma^\sigma(f)(v) + \alpha (\lambda_1^\sigma)^2f^2(v) - \alpha \lambda_1^\sigma \Gamma^\sigma f(v) \right) \\
&= 2 \left( \alpha (\lambda_1^\sigma)^2 f^2(v) + (2-\alpha)\lambda_1^\sigma \Gamma^\sigma(f)(v) \right). 
\end{align*}
Assuming $\alpha >2$, this implies
$$ \Gamma^\sigma f(v) \le \frac{\alpha}{\alpha-2} \lambda_1^\sigma f^2(v). $$
Using this, we obtain for any $x \in V$:
\begin{align*}
g(x) \le g(v) &= 2 \Gamma^\sigma f(v) + \alpha \lambda_1^\sigma f^2(v) \\
& \le \frac{2\alpha}{\alpha-2} \lambda_1^\sigma f^2(v) + \alpha \lambda_1^\sigma f^2(v) \\
&\le \frac{\alpha^2}{\alpha-2} \lambda_1^\sigma \max_{z \in V}f^2(z).
\end{align*}
This implies that we have for all $x \in V$,
$$ \sum_{y \sim x} |f(x) - \sigma(x,y)f(y)|^2 = 2 \Gamma^2f(x) \le g(x) \le \frac{\alpha^2}{\alpha-2} \lambda_1^\sigma \max_{z \in V } f^2(z).  $$
Without loss of generality, we can scale $f$ in such a way that
$$ f(x_0) = \max_{z \in V} f(z) = 1 $$
and
$$ f(x) \ge -1 \quad \text{for all $x \in V$.} $$
The claim in the proof of Proposition \ref{prop:nudiamvol} guarantees the existence of a path $x_0 \sim x_1 \sim \cdots \sim x_s$ of length $s \le {\rm{diam}}(G)+1$, satisfying
$$ \sum_{j=0}^{s-1} |f(x_j) - \sigma(x_j,x_{j+1}) f(x_{j+1})|^2 \ge \frac{1}{{\rm{diam}}(G)+1}. $$
Using the above estimate, we conclude that
$$ \frac{1}{{\rm{diam}}(G)+1} \le \sum_{j=0}^{s-1} \frac{\alpha^2}{\alpha-2} \lambda_1^\sigma \max_{z \in V} f^2(z) \le \frac{\alpha^2}{\alpha-2}({\rm{diam}}+1)\lambda_1^\sigma. $$
This finishes the proof of step 1.

\medskip

{\bf Step 2:} Assume that $G=(V,E)$ is a connected, non-bipartite triangle-free graph satisfying $CD(0,\infty)$. Let $\overline{\sigma}: E^{or}(G) \to S^1$ be the anti-balanced magnetic potential. Since $G$ is not bipartite, $\overline{\sigma}$ is an unbalanced magnetic potential. 

Triangle-freeness guarantees that, at every vertex $x \in V$, the magnetic potential on the incomplete $2$-ball $\mathring{B}_2(x)$ can be gauged to be trivial. 
Since the $CD^\sigma(K,n)$ condition at a vertex $x \in V$ is determined by the incomplete $2$-ball $\mathring{B}_2(x)$ and does not change under gauge transforms, we conclude that $G$ satisfies also the $CD^\sigma(0,\infty)$-condition at all vertices. Then the statement of the proposition follows from Step 1.
\end{proof}

\subsection{Regular bipartite graphs assuming the upper bound}
\label{sec:regbipuppbd}

For $d$-regular bipartite graphs $G$, we have the upper bound $\nu(G) \le d-\sqrt{d}$. This upper bound is assumed for all hypercubes and all complete bipartite graphs. In this subsection, we present further investigations about this particular class of graphs. We start with another example assuming this upper bound.

\begin{example}
Let $G$ be the $4$-regular bipartite graph $G$ given in \cite[Figure 1]{LMP24}. A magnetic potential $\hat \sigma: E^{or}(G) \to \{-1,1\}$ of this graph is illustrated in Figure \ref{fig:LMP24}, assuming the value $+1$ on the red edges and the value $-1$ on the blue edges. (Note that the blue edges form a perfect matching of $G$.)
\begin{figure}[h]
\begin{center}
\begin{tikzpicture}[thick,scale=1]
\filldraw [black] (5,0) circle (2pt);
\filldraw [black] (1,2) circle (2pt);
\filldraw [black] (3,2) circle (2pt);
\filldraw [black] (7,2) circle (2pt);
\filldraw [black] (9,2) circle (2pt);
\filldraw [black] (0,4) circle (2pt);
\filldraw [black] (2,4) circle (2pt);
\filldraw [black] (4,4) circle (2pt);
\filldraw [black] (6,4) circle (2pt);
\filldraw [black] (8,4) circle (2pt);
\filldraw [black] (10,4) circle (2pt);
\filldraw [black] (2,6) circle (2pt);
\filldraw [black] (5,6) circle (2pt);
\filldraw [black] (8,6) circle (2pt);
\draw (5,0) node[label=below:$x_1$]{};
\draw (1,2) node[label=below:$x_2$]{};
\draw (3,2) node[label=below:$x_3$]{};
\draw (7,2) node[label=below:$x_4$]{};
\draw (9,2) node[label=below:$x_5$]{};
\draw (0,4) node[label=below:$x_6$]{};
\draw (2,4) node[label=below:$x_7$]{};
\draw (4,4) node[label=below:$x_8$]{};
\draw (6,4) node[label=below:$x_9$]{};
\draw (8,4) node[label=below:$x_{10}$]{};
\draw (10,4) node[label=below:$x_{11}$]{};
\draw (2,6) node[label=above:$x_{12}$]{};
\draw (5,6) node[label=above:$x_{13}$]{};
\draw (8,6) node[label=above:$x_{14}$]{};
\draw [blue] (5,0) -- (1,2);
\draw [red] (5,0) -- (3,2);
\draw [red] (5,0) -- (7,2);
\draw [red] (5,0) -- (9,2);
\draw [red] (1,2) -- (0,4);
\draw [red] (1,2) -- (2,4);
\draw [red] (1,2) -- (4,4);
\draw [red] (3,2) -- (0,4);
\draw [blue] (3,2) -- (6,4);
\draw [red] (3,2) -- (8,4);
\draw [red] (7,2) -- (2,4);
\draw [red] (7,2) -- (6,4);
\draw [blue] (7,2) -- (10,4);
\draw [red] (9,2) -- (4,4);
\draw [blue] (9,2) -- (8,4);
\draw [red] (9,2) -- (10,4);
\draw [blue] (0,4) -- (2,6);
\draw [red] (0,4) -- (5,6);
\draw [red] (2,4) -- (2,6);
\draw [blue] (2,4) -- (8,6);
\draw [blue] (4,4) -- (5,6);
\draw [red] (4,4) -- (8,6);
\draw [red] (6,4) -- (5,6);
\draw [red] (6,4) -- (8,6);
\draw [red] (8,4) -- (2,6);
\draw [red] (8,4) -- (8,6);
\draw [red] (10,4) -- (2,6);
\draw [red] (10,4) -- (5,6);
\end{tikzpicture}
\end{center}
\caption{A $d$-regular bipartite graph with $\nu(H) = d-\sqrt{d}$ for $d=4$. \label{fig:LMP24}}
\end{figure}
\FloatBarrier
The matrix representation of the corresponding magnetic Laplacian is given by
$$ \Delta^{\hat \sigma} \cong {\small{\begin{pmatrix} \begin{array}{c|cccc|cccccc|ccc}
4& 1& -1& -1& -1& 0& 0& 0& 0& 0& 0& 0& 0& 0\\ 
\hline
1& 4& 0& 0& 0& -1& -1& -1& 0& 0& 0& 0& 0& 0\\ 
-1& 0& 4& 0& 0& -1& 0& 0& 1& -1& 0& 0& 0& 0\\ 
-1& 0& 0& 4& 0& 0& -1& 0& -1& 0& 1& 0& 0& 0\\ 
-1& 0& 0& 0& 4& 0& 0& -1& 0& 1& -1& 0& 0& 0\\ 
\hline
0& -1& -1& 0& 0& 4& 0& 0& 0& 0& 0& 1& -1& 0\\ 
0& -1& 0& -1& 0& 0& 4& 0& 0& 0& 0& -1& 0& 1\\ 
0& -1& 0& 0& -1& 0& 0& 4& 0& 0& 0& 0& 1& -1\\ 
0& 0& 1& -1& 0& 0& 0& 0& 4& 0& 0& 0& -1& -1\\ 
0& 0& -1& 0& 1& 0& 0& 0& 0& 4& 0& -1& 0& -1\\ 
0& 0& 0& 1& -1& 0& 0& 0& 0& 0& 4& -1& -1& 0\\ 
\hline
0& 0& 0& 0& 0& 1& -1& 0& 0& -1& -1& 4& 0& 0\\ 
0& 0& 0& 0& 0& -1& 0& 1& -1& 0& -1& 0& 4& 0\\ 
0& 0& 0& 0& 0& 0& 1& -1& -1& -1& 0& 0& 0& 4
\end{array} \end{pmatrix}}} $$
It can be verified with Maple that $\Delta^{\hat \sigma}$ has precisely the two eigenvalues $2,6$, both with multiplicity $7$. Applying Proposition \ref{prop:regbipartuppbd}(iii) yields $\nu(G) = 4-\sqrt{4}=2$.
\end{example}

The following result presents a relation between $d$-regular bipartite graphs $G$ satisfying $\nu(G) = d-\sqrt{d}$ and a special family of unitary matrices.

\begin{theorem} \label{thm:BBtop} Let $G=(V,E)$ be a $d$-regular bipartite graph with $2n$ vertices and $\nu(G) = d-\sqrt{d}$ and $\sigma: E^{or}(G) \to S^1$ be a maximal potential.  Enumerating the vertices of $G$ such that the first $n$ vertices agree with $V_1$ and the last $n$ vertices with $V_2$ of the bipartite vertex partition $V = V_1 \cup V_2$, the magnetic adjacency matrix $A^\sigma$ has the following block structure:
\begin{equation*} \label{eq:Asigmablock} 
A^\sigma = \begin{pmatrix} {\bf{0}}_n & B \\ \overline{B}^\top & {\bf{0}}_n \end{pmatrix}, 
\end{equation*}
with a matrix $B$ of size $n \times n$ satisfying 
$$ \left( \frac{1}{\sqrt{d}} B \right) \left( \frac{1}{\sqrt{d}} \overline{B}^\top \right) = {\rm{Id}}_n. $$
Moreover, all non-zero entries of $B$ have modulus $1$. 

Conversely, any complex-valued matrix $B$ of size $n \times n$ with $n \ge d$ having these properties provides a $d$-regular bipartite graph $G=(V,E)$ with $2n$-vertices and adjacency matrix
$$ A_G = \begin{pmatrix} {\bf{0}}_n & |B| \\ |B|^\top & {\bf{0}}_n \end{pmatrix}, $$
where $|B|$ is the matrix whose entries are the moduli of the entries of $B$, that is $|B|_{ik}=|B_{jk}|$ for $j,k \in \{1,\dots,n\}$.
\end{theorem}

\begin{proof} The statements in the theorem are an immediate consequence of Proposition \ref{prop:regbipartuppbd}(ii), since
any $d$-regular bipartite graph with $2n$ vertices satisfies $\nu(G)=d-\sqrt{d}$ if and only if there exists one (every) maximal potential $\sigma$ with $(A^\sigma)^2= d \cdot {\rm{Id}}_{2n}$. The chosen enumeration implies that $A^\sigma$ has the block structure given in \eqref{eq:Asigmablock}, which implies
$$ (A^\sigma)^2 = \begin{pmatrix} B \overline{B}^\top & {\bf{0}}_n \\ {\bf{0}}_n & \overline{B}^\top B \end{pmatrix}. $$
The identity $(A^\sigma)^2=d\cdot {\rm{Id}}_{2n}$ leads then to the stated results.
\end{proof}

\begin{corollary} \label{cor:Kddmagnetospecheight} The magneto-spectral height of any $d$-regular complete bipartite graph $K_{d,d}$ with $d \ge 2$ satisfies $\nu(K_{d,d})=d-\sqrt{d}$.
\end{corollary}

\begin{proof} Applying Theorem \ref{thm:BBtop}, this follows immediately from the fact that the DFT matrix
\begin{equation} \label{eq:BDFT} 
B = \begin{pmatrix} 1 & 1 & \cdots & 1 \\ 1 & \xi_d & \cdots & \xi_d^{d-1} \\ 1 & \xi_d^2 & \cdots & {\xi_d}^{2(d-1)} \\
\vdots & \vdots & \vdots & \vdots \\ 1 & \xi_d^{d-1} & \cdots & \xi_d^{(d-1)^2} \end{pmatrix} 
\end{equation}
with $\xi_d = e^{2 \pi i/d}$ satisfies $B \overline{B}^\top = d \cdot {\rm{Id}}$.
\end{proof}

\begin{remark} \label{rem:UWnd} Square matrices $B$ of size $n \times n$ with all non-zero entries of modulus $1$ satisfying $B \overline{B}^\top = d \cdot {\rm{Id}}_n$ are called \emph{unit weighing matrices} forming the set $UW(n,d)$. This family is a generalization of the set of Hadamard matrices. A discussion of these matrices for special values of $n,d$ is presented in \cite{BKR13}. Their results imply that there are only two different connected $3$-regular bipartite graphs $G$ satisfying $\nu(G) = 3-\sqrt{3}$, namely, $K_{3,3}$ and the cube $Q^3$. They also classify all matrices in $UW(n,4)$. This means that there are only four different connected $4$-regular bipartite graphs 
$G=G(W_j)$, $j=5,6,7,8$, corresponding to the $B$ matrices $W_5,W_6,W_7,W_8$ in \cite[Section 3.4]{BKR13} satisfying $\nu(G) = 4-\sqrt{4}=2$ and one infinite family $G=G(E_{2m})$ of such connected graphs with $4m$ vertices, corresponding to the $B=E_{2m}$ matrices in \cite[Section 3.4]{BKR13}.
For example, the matrix 
$$ B = W_5= \begin{pmatrix} 1 & 1 & 1 & 1 & 0 \\ 1 & \xi_3 & \xi_3^2 & 0 & 1 \\ 1 & \xi_3^2 & 0 & \xi_3 & \xi_3^2 \\ 1 & 0 & \xi_3 & \xi_3^2 & \xi_3 \\ 0 & 1 & \xi_3^2 & \xi_3 & \xi_3 \end{pmatrix} \in UW(5,4), $$
corresponds to crown graph with $10$ vertices ($K_{5,5}$ with a perfect matching removed). The graphs $G(E_{2m})$ can be constructed from even cycles $C_{2m}$, by doubling each vertex and replacing each edge by four edges forming a $K_{2,2}$, as shown in Figure \ref{fig:G2}.
\begin{figure}[h]
\begin{center}
\begin{tikzpicture}[thick,scale=1]
\draw (3,6.5) node[label=above:{\Huge{$C_4$}}]{};
\draw (11,6.5) node[label=above:{\Huge{$G(E_4)$}}]{};
\filldraw [red] (1,3) circle (2pt);
\filldraw [blue] (3,1) circle (2pt);
\filldraw [blue] (3,5) circle (2pt);
\filldraw [red] (5,3) circle (2pt);
\filldraw [red] (8,3) circle (2pt);
\filldraw [red] (9,3) circle (2pt);
\filldraw [blue] (11,0) circle (2pt);
\filldraw [blue] (11,1) circle (2pt);
\filldraw [blue] (11,5) circle (2pt);
\filldraw [blue] (11,6) circle (2pt);
\filldraw [red] (13,3) circle (2pt);
\filldraw [red] (14,3) circle (2pt);
\draw [black] (1,3) -- (3,1);
\draw [black] (3,1) -- (5,3);
\draw [black] (5,3) -- (3,5);
\draw [black] (3,5) -- (1,3);
\draw [black] (11,0) -- (8,3);
\draw [black] (11,0) -- (9,3);
\draw [black] (11,0) -- (14,3);
\draw [black] (11,0) -- (13,3);
\draw [black] (11,1) -- (8,3);
\draw [black] (11,1) -- (9,3);
\draw [black] (11,1) -- (13,3);
\draw [black] (11,1) -- (14,3);
\draw [black] (11,5) -- (8,3);
\draw [black] (11,5) -- (9,3);
\draw [black] (11,5) -- (13,3);
\draw [black] (11,5) -- (14,3);
\draw [black] (11,6) -- (8,3);
\draw [black] (11,6) -- (9,3);
\draw [black] (11,6) -- (13,3);
\draw [black] (11,6) -- (14,3);
\end{tikzpicture}
\end{center}
\caption{Construction of the $4$-regular bipartite graph $G(E_4)$ from the cycle $C_{2m}$ with $m=2$.\label{fig:G2}}
\end{figure}
\FloatBarrier
It also follows from the results that there does not exist a $5$-regular bipartite graph with $14$ vertices (see \cite[Setion 3.7]{BKR13}).

Moreover, the (Kronecker) tensor product of two matrices $H_1 \in UW(n_1,d_1)$ and $H_2 \in UW(n_2,d_2)$ yields a new matrix $H_1 \otimes H_2 \in UW(n_1n_2,d_1d_2)$. Note that this operation on the level of bipartite graphs $G_1$, $G_2$, denoted by $G_1 \odot G_2$, is different from taking Cartesian products or the standard tensor product of full adjacency metrices. However, the upper bound $\nu(G) = d-\sqrt{d}$ is also preserved under this particular graph operation.
\end{remark}

We also have the following interesting Bakry-\'Emery curvature results for $d$-regular bipartite graphs $G$ assuming the upper bound $\nu(G) = d-\sqrt{d}$. 

\begin{theorem} \label{thm:Chunyang} For any $d$-regular bipartite graph $G=(V,E)$ with $\nu(G) = d-\sqrt{d}$ and any maximal magnetic potential $\sigma$, the graph $G$ satisfies $CD^\sigma(0,\infty)$. Moreover, we have the following:
\begin{itemize}
    \item[(a)] Let $x \in V$. If all distance-$2$ vertices $z \in V$ from $x$ have precisely $2$ common neighbours with $x$, the $\sigma$-curvature of $G$ vanishes, that is, $K^\sigma_\infty(x) = 0$. 
    \item[(b)] Let $x \in V$. If all distance-$2$ vertices $z\ \in V$ from $x$ have at least $b \ge 3$ common neighbours with $x$, the graph $G$ has strictly positive $\sigma$-curvature satisfying 
    \begin{equation} \label{eq:Kinfsigmaregbip}
    K_\infty^\sigma(x) \ge \min\left\{2,d\left( 1-\frac{2}{b} \right)\right\} \quad \text{for all $x \in V$.} 
    \end{equation}
    Moreover, \eqref{eq:Kinfsigmaregbip} holds with equality if all distance-$2$ vertices $z\ \in V$ from $x$ have precisely $b \ge 3$ common neighbours with $x$.
\end{itemize}
In particular, all hypercubes are $\sigma$-flat, and all complete bipartite graphs $K_{d,d}$, $d \ge 3$ have constant positive $\sigma$-curvature $K_\infty^\sigma \equiv 1$ for $d=3$ and $K_\infty^\sigma \equiv 2$ for $d\ge 4$.
\end{theorem}
\begin{proof} Let $G=(V,E)$ be a $d$-regular bipartite graph with $\nu(G)=d-\sqrt{d}$ and $x \in V$. Assume that $\sigma$ is a  maximal potential of $G$. We denote the local structure of $x\in V$ as follows:
$$S_1(x):=\{y_1,y_2,\cdots,y_d\},\,\,\text{and}\,\,S_2(x):=\{z_1,z_2,\cdots,z_n\}.$$
Without loss of generality, we can assume $\sigma_{xy_j}=1$ for all $j=1,\dots,d$, by choosing a switching function $\tau$ which locally satisfies
$$\tau(x):=1,\,\,\tau(y_j):=\sigma_{xy_j}^{-1}.$$
We estimate the magnetic Bakry-\'Emery curvature at $x \in V$ from below as follows:

At first, the local matrix $\Gamma^\sigma_2(x)$ corresponding to the sesquilinear form $\Gamma_2^\sigma(f,f)(x)$ for any function $f:V \to \mathbb{C}$ is
\begin{equation}\label{eq:gamma2}
4\Gamma^\sigma_2(x)=
\begin{pmatrix}
\begin{array}{c|ccc|ccc}
   3d+d^2  & -2-2d & \cdots & -2-2d & 0 & \cdots & 0 \\
\hline
   -2-2d  & 2d+2 & \cdots & 2 & -2\sigma_{y_1z_1} & \cdots & -2\sigma_{y_1z_n} \\
  \vdots & \vdots &  \ddots &  \vdots & \vdots &   &  \vdots \\
  -2-2d & 2 & \cdots & 2d+2 & -2\sigma_{y_dz_1} & \cdots & -2\sigma_{y_dz_n}\\
\hline
0 & -2\sigma_{z_1y_1} & \cdots & -2\sigma_{z_1y_d} & d_{z_1}^{-}  &    &    \\
\vdots & \vdots &     &  \vdots &  &  \ddots &  \\
0 & -2\sigma_{z_ny_1} & \cdots & -2\sigma_{z_ny_d} &  &  & d_{z_n}^{-}
\end{array}
\end{pmatrix}.
\end{equation}
Here we set $\sigma_{yz}=0$ if two vertices $y\in V $ and $z\in V$ are not adjacent and $d_z^-$ denotes the in-degree of $z\in S_2(x)$, which is the number of $2$-paths connecting $x\in V$ and $z\in S_2(x)$. By Proposition \ref{prop:regbipartuppbd}, we know that, for any $z\in S_2(x)$, the in-degree satisfies $d_z^-\ge 2$. Let us denote 
$$b:=\min\{d_{z_1}^-,d_{z_2}^-,\cdots,d_{z_n}^-\}\ge 2,$$
and write the matrix (\ref{eq:gamma2}) as the sum of the following two matrices:
\begin{align}
4\Gamma^\sigma_2(x)=&
\begin{pmatrix}
\begin{array}{c|ccc|ccc}
   3d+d^2  & -2-2d & \cdots & -2-2d & 0 & \cdots & 0 \\
\hline
   -2-2d  & 2d+2 & \cdots & 2 & -2\sigma_{y_1z_1} & \cdots & -2\sigma_{y_1z_n} \\
  \vdots & \vdots &  \ddots &  \vdots & \vdots &   &  \vdots \\
  -2-2d & 2 & \cdots & 2d+2 & -2\sigma_{y_dz_1} & \cdots & -2\sigma_{y_dz_n}\\
\hline
0 & -2\sigma_{z_1y_1} & \cdots & -2\sigma_{z_1y_d} & b  &    &    \\
\vdots & \vdots &     &  \vdots &  &  \ddots &  \\
0 & -2\sigma_{z_ny_1} & \cdots & -2\sigma_{z_ny_d} &  &  & b
\end{array}
\end{pmatrix}\label{matrix:gamma2_constant_diagonal}\\
&+\begin{pmatrix}
\begin{array}{c|ccc|ccc}
    0 &  & \cdots & 0 & 0 & \cdots & 0 \\
\hline
   0  & 0 & \cdots & 0 & 0 & \cdots & 0 \\
  \vdots & \vdots &  \ddots &  \vdots & \vdots &   &  \vdots \\
 0 & 0 & \cdots & 0 & 0 & \cdots & 0\\
\hline
0 & 0 & \cdots & 0 & d_{z_1}^{-}-b  &    &    \\
\vdots & \vdots &     &  \vdots &  &  \ddots &  \\
0 & 0 & \cdots & 0 &  &  & d_{z_n}^{-}-b
\end{array}
\end{pmatrix}\label{matrix:positve_semidefnite}.
\end{align}
It is obvious that the second matrix (\ref{matrix:positve_semidefnite}) is positive semidefinite. To obtain a lower magnetic curvature bound, it is sufficient to consider the case when the matrix \eqref{matrix:positve_semidefnite} vanishes and to compute the corresponding curvature matrix $A_\infty$  in this particular case, following the steps given in \cite[Section 3]{HL22}. 

We take the Schur complement of the submatrix $\text{diag}\{b,b,\cdots,b\}$ in the first matrix (\ref{matrix:gamma2_constant_diagonal}) and obtain the following matrix, denoted by $4Q(x)$ to be in agreement with the notation in \cite{HL22}:
\begin{multline*}
4Q(x)=\\
(4\Gamma^\sigma_2(x))_{B_1(x),B_1(x)}-(4\Gamma^\sigma_2(x))_{B_1(x),S_2(x)}(4\Gamma^\sigma_2(x)_{S_2(x),S_2(x)})^{-1}(4\Gamma^\sigma_2(x))_{S_2(x),B_1(x)} =\\
\begin{pmatrix}
\begin{array}{c|ccc}
   3d+d^2 & -2-2d & \cdots & -2-2d  \\
\hline
 -2-2d & 2d+2 & \cdots & 2 \\
 \vdots &  & \ddots &  \\
 -2-2d & 2 & \cdots & 2d+2 
\end{array}
\end{pmatrix}
-\frac{4}{b}
\begin{pmatrix}
   0 & \cdots  & 0 \\
\hline
   \sigma_{y_1z_1}  & \cdots & \sigma_{y_1z_n}\\
   \vdots &   & \vdots \\
   \sigma_{y_dz_n} & \cdots & \sigma_{y_dz_n}
\end{pmatrix}
\begin{pmatrix}
\begin{array}{c|ccc}
   0  & \sigma_{z_1y_1} & \cdots & \sigma_{z_1y_d} \\
    \vdots &  \vdots &   &  \vdots \\
    0 & \sigma_{z_ny_1} & \cdots & \sigma_{z_ny_d}
\end{array}
\end{pmatrix}.
\end{multline*}
The diagonal entries of the matrix product in the second term are $d-1$, except for the first entry, which vanishes. Moreover, the off-diagonal $(j,k)$-entry is $\sum_{l=1}^n\sigma_{y_jz_l}\sigma_{z_ly_k}$. Any two different vertices $y_j$ and $y_k$ in $S_1(x)$ are not adjacent, since $G$ is bipartite, and the distance between them is $2$, given by $y_j \sim x\sim y_k$. According to Proposition \ref{prop:regbipartuppbd}, we have
$$\sum_{u\in V}\sigma_{y_ju}\sigma_{uy_k}=\sigma_{y_jx}\sigma_{xy_k}+\sum_{z_l\in S_2(x)}\sigma_{y_jz_l}\sigma_{z_ly_k}=0.$$
Using $\sigma_{xy_j}=1$, this leads to 
$$\sum_{z_l \in S_2(x)}\sigma_{y_jz_l}\sigma_{z_jy_k}=-\sigma_{y_jx}\sigma_{xy_k}=-1,$$
and we obtain the following Schur complement $4Q(x)$:
\begin{equation}
4Q(x)=\begin{pmatrix}\label{matrix:Q(x)}
\begin{array}{c|ccc}
    3d+d^2 & -2-2d & \cdots & -2-2d \\
    \hline
    -2-2d & (2d+2)-\frac{4}{b}(d-1) & \cdots & 2+\frac{4}{b} \\
    \vdots &  &  \ddots &  \\
    -2-2d & 2+\frac{4}{b} & \cdots & (2d+2)-\frac{4}{b}(d-1)
\end{array}    
\end{pmatrix}.
\end{equation}
As in \cite[Section]{HL22}, we now introduce the non-singular matrix $B$, given via the formulas (3.4), (3.6) and (3.7) in \cite{HL22}:
$$B:=\begin{pmatrix}
\begin{array}{c|ccc}
    1 & 1 & \cdots & 1 \\
    \hline
     & 1 &  & \\
     &   & \ddots &  \\
     &   &     &  1
\end{array}
\end{pmatrix}.
$$
Computing the matrix $B(4Q(x))\overline{B}^\top$ yields:
\begin{equation} \label{eq:BQB}
B(4Q(x))\overline{B}^\top=\begin{pmatrix}
\begin{array}{c|ccc}
   d^2-d  & 2d-2 & \cdots & 2d-2 \\
   \hline
   2d-2  & (2d+2)-\frac{4}{b}(d-1) & \cdots & 2+\frac{4}{b}\\
   \vdots &  & \ddots &   \\
   2d-2 & 2+\frac{4}{b} & \cdots & (2d+2)-\frac{4}{b}(d-1)
\end{array}
\end{pmatrix}.
\end{equation}
This leads to the following expression of the curvature matrix $A_\infty$ in \cite[(3.17) of Definition 3.1]{HL22}, by using the fact that $(a,\omega^\top)$ in (3.17) agrees with the first row of \eqref{eq:BQB}:
$$ A_{\infty}(G,x,\sigma,B) = 
\begin{pmatrix}
   (d+1)-\frac{2}{b}(d-1) &  \cdots & 1+\frac{2}{b} \\
    \vdots & \ddots & \vdots \\
    1+\frac{2}{b} & \cdots & (d+1)-\frac{2}{b}(d-1)
\end{pmatrix}
-\frac{2(d-1)}{d}\mathbf{J}_d.
$$
Here $\mathbf{J}_d$ is the all-one matrix of size $d\times d$. We rewrite this matrix as follows:
\begin{align*}
A_{\infty}(G,x,\sigma,B)=&\mathbf{J}_+d+d\cdot\mathbf{I}_d+\frac{2}{b}
\begin{pmatrix}
    -(d-1) & \cdots & 1 \\
    \vdots & \ddots & \vdots \\
    1 & \cdots & -(d-1)
\end{pmatrix}-\frac{2(d-1)}{d}\mathbf{J}_d\\
=&\mathbf{J}_d+d\cdot\mathbf{I}_d+\frac{2}{b}(-\Delta_{K_d})-\frac{2(d-1)}{d}\mathbf{J}_d.
\end{align*}
Here $\Delta_{K_d}$ is the standard Laplacian matrix of the complete graph $K_d$.

For the constant vector $\mathbf{1}_d$, we have
$$A_{\infty}(G,x,\sigma,B)\mathbf{1}_d=2 \cdot \mathbf{1}_d,$$
and for any vector $v\perp \mathbf{1}_d$, we have
$$A_{\infty}(G,x,\sigma,B)v=d(1-\frac{2}{b})v.$$
Therefore, the curvature matrix $A_{\infty}(G,x,\sigma,B)$ has only the distinct eigenvalues $2$ and $d(1-\frac{2}{b})$. Consequently, by \cite[Theorem 3.2]{HL22} and the fact that the matrix \eqref{matrix:positve_semidefnite} is positive semidefinite, the Bakry-\'Emery curvature at $x\in V$ has the following lower bound:
$$K^\sigma_{\infty}(x)\ge\min\left\{2,d\left(1-\frac{2}{b}\right)\right\}.$$

In particular, if all the distance-$2$ vertices $z\in S_2(x)$ of a vertex $x \in v$ have precisely $2$ common neighbors with $x$, i.e. $d_z^-=2$ for any $z\in S_2(x)$, the matrix in  (\ref{matrix:positve_semidefnite}) vanishes, leading to the equality $K^\sigma_{\infty}(x)=0$, which completes $(a)$. An example for this class is the hypercubes.

Moreover, if $b=3$, we have $K^\sigma_{\infty}(x)\ge 1$ and, if $b\ge 4$, we have $K^\sigma_{\infty}(x) \ge 2$. These lower bounds are again attained if all in-degrees of distance-$2$ vertices of $x$ are equal.
\end{proof}

\begin{remark} In comparison, the standard Bakry-\'Emery curvature (for $\sigma = \sigma_0$) of all complete bipartite graphs $K_{d,d}$, $d \ge 1$, and all hypercubes $Q^d$, $d \ge 1$, is also constant and given by $K_\infty \equiv 2$. This holds also for the graphs $G(W_j)$, $j=5,6,7,8$ in Remark \ref{rem:UWnd}. However, the infinite family $G(E_{2m})$, $m \ge 2$, are flat with respect to the standard Bakry-\'Emery curvature and also bone-idle with respect to the Ollivier Ricci curvature (see \cite[Definition 7.1]{BCLMP18} for the definition).
\end{remark}

\section{Graph constructions} \label{sec:graphconstructions}

In this section, we discuss the behaviour of our magneto-spectral invariant under various graph constructions.

\subsection{Connecting two non-adjacent vertices}

\begin{lemma} \label{lem:addedge}
    Let $G=(V,E)$ be a graph and $\tilde G=(V,\tilde E)$ be obtained from $G$ by adding an edge between two distinct vertices of $G$. Then we have
    $$ \nu(\tilde G) \ge \nu(G). $$
\end{lemma}

This is a straightforward consequence of the Rayleigh quotient, since this graph manipulation adds an additional non-negative term in the numerator of the Rayleight quotient. An immediate consequence of the lemma and the results in Example \ref{ex:cycle} is the following

\begin{corollary} Assume that a graph $G=(V,E)$ with $n$ vertices admits a Hamiltonian cycle. Then we have
$$ \nu(G) \ge \nu(C_n) = 2-2\cos(\pi/n). $$
\end{corollary}

\subsection{Adding a bridge between two graphs}

It is obvious that we have for two disjoint graphs $G$ and $H$:
\begin{equation} \label{eq:GHnu} 
\nu(G \cup H) = \min(\nu(G),\nu(H)). 
\end{equation}

\begin{lemma} \label{lem:GHbridge}
  Let $G=(V,E)$ and $H=(V',E')$ be two disjoint graphs and $\tilde G = (\tilde V, \tilde E)$ be obtained by connecting $G$ and $H$ by an edge (bridge). Then we have
  \begin{equation} \label{eq:minmaxsqueeze} 
  \min(\nu(G),\nu(H)) \le \nu(\tilde G) \le \max(\nu(G),\nu(H)). 
  \end{equation}
\end{lemma}

\begin{proof} The lower bound of \eqref{eq:minmaxsqueeze} follows directly from \eqref{eq:GHnu} and Lemma \ref{lem:addedge}. 

For the upper bound, we choose $\hat \sigma: E^{or}(\tilde G) \to S^1$ to satisfy
$$ \nu(\tilde G) = \lambda_1^{\tilde \sigma}(\tilde G). $$
Let $x_0 \in V$ and $z_0 \in V'$ be the bridge in $\tilde G$. Then we have for any nonzero function $f: \tilde V \to \C$,
\begin{multline*} 
\RR^{\hat \sigma}(f) = \\ \frac{{\displaystyle{\sum_{\{x,y\} \in E}}} |f(x)-\hat \sigma(x,y)f(y)|^2 + {\displaystyle{\sum_{\{z,w\} \in E'}}} |f(z)-\hat \sigma(z,w)f(w)|^2 + |f(x_0) - \hat \sigma(x_0,z_0) f(z_0)|^2}{\sum_{x \in V} |f(x)|^2 + \sum_{z \in V'} |f(z)|^2}. 
\end{multline*}
Let $\sigma_1$ and $\sigma_2$ be the restrictions of $\hat \sigma$ to the graphs $G$ and $H$, respectively. Let $f_1: V \to \C$ and $f_2: V' \to \C$ be non-zero functions satisfying 
$$ \RR^{\sigma_1}(f_1) = \lambda_1^{\sigma_1}(G) \quad \text{and} \quad \RR^{\sigma_2}(f_2) = \lambda_1^{\sigma_2}(H). $$
In the case that $f_1(x_0)$ and $f_2(z_0)$ are either both zero or both non-zero, we can assume, after rescaling, that $|f_1(x_0)-\hat \sigma(x_0,z_0)f_2(z_0)| = 0$, and that
$$ f(v) = \begin{cases} f_1(v) & \text{if $v \in V$,} \\ f_2(v) & \text{if $v \in V'$,} \end{cases} $$
is non-zero, too. Then we have
$$ \RR^{\hat \sigma}(f) = \frac{\sum_{\{x,y\}\in E} |f(x)-\hat \sigma(x,y)f(y)|^2 + \sum_{\{z,w\} \in E'} |f(z)-\hat \sigma(z,w)f(w)|^2}{\sum_{x \in V} |f(x)|^2 + \sum_{z \in V'} |f(z)|^2}, $$
and applying
$$ \frac{a+b}{c+d} \le \max\left\{\frac{a}{c},\frac{b}{d}\right\} $$
for $a,b,c,d > 0$, we conclude that
\begin{align*} 
\nu(\tilde G) &= \lambda_1^{\hat \sigma}(\tilde G) \le \RR^{\hat \sigma}(f) \\ &\le \max\left\{\frac{\sum_{\{x,y\} \in E} |f_1(x)-\sigma_1(x,y)f_1(y)|^2}{\sum_{x \in V} |f_1(x)|^2}, \frac{\sum_{\{z,w\} \in E'} |f_2(z)-\sigma_2(z,w)f_2(w)|^2}{\sum_{z \in V'} |f_2(z)|^2} \right\}\\
&= \max\left\{\RR^{\sigma_1}(f_1),\RR^{\sigma_2}(f_2)\right\} = \max\{\lambda_1^{\sigma_1}(G)\,\lambda_1^{\sigma_2}(H)\} \le \max\{\nu(G),\nu(H)\}.
\end{align*}
It remains to consider the case when precisely one of $f_1(x_0), f_2(z_0)$ is zero. Assume, without loss of generality, that $f_1(x_0) = 0$ and $f_2(z_0) \neq 0$.
We set
$$ f(v) = \begin{cases} f_1(v) & \text{if $v \in V$,} \\ 0 & \text{if $v \in V'$.} \end{cases}$$
Note that $f$ is non-zero, and we obtain
\begin{align*} 
\nu(\tilde G) &= \lambda_1^{\hat \sigma}(\tilde G) \le \RR^{\hat \sigma}(f) \\ &\le \frac{\sum_{\{x,y\}\in E} |f_1(x)-\sigma_1(x,y)f_1(y)|^2}{\sum_{x \in V} |f_1(x)|^2}\\
&= \RR^{\sigma_1}(f_1) = \lambda_1^{\sigma_1}(G) \le \nu(G).
\end{align*}
\end{proof}

The next result describes the behaviour of the magneto-spectral height if a dangling edge is added.

\begin{proposition} Let $G=(V,E)$ be a graph and $\tilde G=(\tilde V, \tilde E)$ be obtained by adding a dangling edge to $G$. Then we have
\begin{equation} \label{eq:nutildeGnuG0} 
\nu(\tilde G) \le \nu(G). 
\end{equation}
as well as
\begin{equation} \label{eq:nutilde1} 
\nu(\tilde G) < 1. 
\end{equation}
\end{proposition}

\begin{proof}
Inequality \eqref{eq:nutildeGnuG0} follows from Lemma \ref{lem:GHbridge} by choosing $H$ to be a single vertex, and inequality \eqref{eq:nutilde1} is just a reformulation of \eqref{eq:nuleaf} in Corollary \ref{cor:nudmin}.
\end{proof}

\subsection{Splitting vertices and edges}

We first discuss a graph manipulation involving the splitting of a vertex.

\begin{lemma}\label{lem:splitvertex}
    Let $G=(V,E)$ be a graph and $x \in V$ be a vertex of degree $d=d_1+d_2$. Let $\tilde G=(\tilde V,\tilde E)$ be the graph obtained by splitting $x$ into two non-adjacent vertices $x_1,x_2$, where $x_1$ is adjacent to $d_1$ neighbours of $x$ and $x_2$ adjacent to the other $d_2$ neighbours of $x$. Then we have
    $$ \nu(G) \ge \nu(\tilde G). $$
\end{lemma}

\begin{proof}
Note that under this operation, the set of edges is unchanged (with the minor modification that an edge incident to $x$ is now incident to either $x_1$ or $x_2$). Let $\sigma: E^{or}(G) \to S^1$ be a magnetic potential and $f: V \to \C$ be an arbitrary non-zero function. Let $\tilde \sigma: E^{or}(\tilde G) \to S^1$ be the corresponding magnetic potential inheriting the directed edge values of $\sigma$ and $\tilde f: \tilde V \to \C$ be the following function:
$$ \tilde f(z) = \begin{cases} f(z) & \text{if $z \in V \setminus \{x\}$,} \\ f(x) & \text{if $z\in \{x_1,x_2\}$.} \end{cases} $$
If numerator and denominator of $\RR^\sigma_G(f)$ is denoted by $A$ and $B$, then we have
$$ \RR^{\tilde \sigma}_{\tilde G}(\tilde f) = \frac{A}{B - |f(x)|^2 + |\tilde f(x_1)|^2 + |\tilde f(x_2)|^2} = \frac{A}{B + |f(x)|^2} \le \frac{A}{B} = \RR^\sigma_G(f). $$
This finishes the proof.
\end{proof}

This result implies the following interesting lower bounds for graphs with certain vertex degree restrictions. 

\begin{corollary} \label{cor:evendegree}
Let $G=(V,E)$ be a graph whose vertices have even degrees. Then we have
$$ \nu(G) \ge 2 - 2 \cos\left( \frac{\pi}{|E|} \right). $$
Moreover, if this inequality holds with equality, then $G$ is the cycle $C_{|E|}$.
\end{corollary}

\begin{proof} Applying the above graph manipulation iteratively, we end up with a new graph $\tilde G$ with the same number of edges and which is $2$-regular. Any $2$-regular graph is a disjoint union of cycles. Assume $C_n$ ($n \le |E|$) is the largest cycle in this disjoint union. Using \eqref{eq:GHnu} and \eqref{eq:nucycle}, we conclude that
$$ \nu(G) \ge \nu(\tilde G) \ge \nu(C_n) = 2-2\cos\left(\frac{\pi}{n}\right) \ge 2-2\cos\left(\frac{\pi}{|E|}\right). $$
This shows
$$ \nu(G) \ge 2 - 2 \cos\left( \frac{\pi}{|E|} \right). $$
If $G$ is not the cycle $C_{|E|}$, then we apply these graph manipulations only until we end up with a graph $\tilde G$ with all but one vertices of degree $2$ and the remaining vertex $x$ of degree $4$. The connected component of $\tilde G$ containing $x$ must be the union of two disjoint cycles glued at the vertex $x$. In the last step, we split $x$ and disconnect this connected component into two disjoint cycles $C_j$ and $C_k$ with $j,k < |E|$. Using \eqref{eq:GHnu}, we conclude that
$$ \nu(G) \ge \nu(\tilde G) \ge \nu(C_k) = 2-2\cos\left(\frac{\pi}{k}\right) > 2-2\cos\left(\frac{\pi}{|E|}\right), $$
proving the rigidity statement.
\end{proof}

\begin{corollary} \label{cor:Cddef}
Introducing
$$
\nu_d := \min_{\text{G with $d_{\min}(G)\ge d$}} \nu(G), 
$$
we have $\nu_2=0$ and, for $d \ge 7$,
$$ \nu_d \ge d-2\sqrt{2}\sqrt{d-1} \to \infty \quad \text{as $d \to \infty$.} $$
In particular, we have 
for every finite graph $G$ with $d_{\min}(G) \ge 7$,
$$ \nu(G) \ge d_{\min} - 2\sqrt{2} \sqrt{d_{\min}-1} \ge \nu_7 \ge 7-4\sqrt{3} \approx 0.0718. $$
\end{corollary}

The proof is, similarly as in the proof of Corollary \ref{cor:evendegree}, a direct consequence of iterative vertex splitting whenever a vertex has degree $\ge 2 d_{\min}$ and the lower Ramanujan type bound \eqref{eq:ramadmindmax} in Proposition \ref{prop:MSS}. Moreover, $\nu_2=0$ follows from the fact that $\nu(C_n) \to 0$ for increasing cycles $C_n$. A challenging question is whether we have $\nu_3 > 0$.

\begin{remark} In view of the above vertex splitting arguments, we observe the following: In order to prove $\nu_3>0$, it would be sufficient to show that there exists a positive lower bound for the magneto-spectral height on all graphs $G$ with minimal vertex degree $\ge 3$ and maximal vertex degree $\le 5$.
\end{remark}

The next result is concerned with the splitting of an edge, that is, replacing an edge by a path of length $2$.

\begin{lemma} \label{lem:splitedge}
Let $G=(V,E)$ be a graph and $\tilde G=(\tilde V,\tilde E)$ be obtained by splitting an edge in $G$. Then we have
$$ \nu(\tilde G) \le \nu(G). $$
More generally, if $\tilde G$ is a subdivision of $G$, then we have
$$ \nu(\tilde G) \le \nu(G). $$
\end{lemma}

\begin{proof}
    Let $\{x,y\} \in E$ be the edge which is split with an additional vertex $z$, that is, $\tilde V = V \cup \{z\}$ and 
    $$ \tilde E = \left( E \setminus \{ \{x,y\} \} \right) \cup \{ \{x,z\},\{z,y\}\}. 
    $$
    Let $\hat \sigma$ be a maximal magnetic potential of $\tilde G$, that is,
    $$ \nu(\tilde G) = \lambda_1^{\hat \sigma}(\tilde G). $$
    Let $\sigma_1$ be the magnetic potential on $G$ which agrees with $\hat \sigma$ on the edges $E\setminus\{\{x,y\}\}$ and which satisfies
    $$ \sigma_1(x,y) = \hat \sigma(x,z) \hat \sigma(z,y). $$
    Let $f: V \to \C$ be a non-zero function satisfying
    $$ \lambda_1^{\sigma_1}(G) = \RR^{\sigma_1}(f). $$
    Let $\tilde f: \tilde V \to \C$ be the function which agrees with $f$ on the vertices in $V$ and which satisfies 
    $$ \tilde f(z) = \hat \sigma(x,z)^{-1}f(x). $$
    Then we have
    \begin{align*} 
    \nu(\tilde G) &\le \RR^{\hat\sigma}(\tilde f) \\
    &= \frac{|\tilde f(x)-\hat \sigma(x,z)\tilde f(z)|^2 + |\tilde f(z)-\hat \sigma(z,y)\tilde f(y)|^2 + \displaystyle{\sum_{\{v,w\}\in E\setminus \{\{x,y\}\}}} |\tilde f(v)-\hat \sigma(v,w)\tilde f(w)|^2}{|\tilde f(z)|^2 + \sum_{w \in V} |\tilde f(w)|^2} \\
    &\le \frac{ |f(x)-\hat \sigma(x,z) \hat \sigma(z,y) f(y)|^2 + \displaystyle{\sum_{\{v,w\}\in E\setminus \{\{x,y\}\}}} |f(v)-\sigma_1(v,w)f(w)|^2}{\sum_{w \in V} |f(w)|^2} \\
    &= \RR^{\sigma_1}(f) = \lambda_1^{\sigma_1}(G) \le \nu(G).
\end{align*}
\end{proof}

\subsection{Cartesian products}

\begin{lemma} \label{lem:cartprod} Let $G$ and $H$ be two graphs. Then we have
$$ \nu(G \times H) \ge \nu(G) + \nu(H). $$
\end{lemma}

\begin{proof}
    We use the following fact for magnetic potentials $\sigma_1$ and $\sigma_2$ on $G$ and $H$, respectively, and their Cartesian product magnetic potential $ \sigma$:
    $$ {\rm{spec}}^{\sigma}(G \times H) = \left\{ \lambda + \mu: \lambda \in {\rm{spec}}^{\sigma_1}(G), \mu \in {\rm{spec}}^{\sigma_2}(H) \right\}, $$
    where ${\rm{spec}}^\sigma$ is the multiset of eigenvalues of the magnetic Laplacian $\Delta^\sigma$. Here $\sigma: E^{or}(G \times H) \to S^1$ is defined as follows:
    $$ \sigma(\{(x,z),(y,z)\}) = \sigma_1(x,y) \quad\text{for all $x \sim_G y$ and $z \in H$} $$
    and
    $$ \sigma(\{(x,z),(x,w)\}) = \sigma_2(z,w) \quad \text{for all $x \in G$ and $z \sim_H w$.} $$
    This implies that
    $$ \lambda_1^{\sigma}(G \times H) = \lambda_1^{\sigma_1}(G) + \lambda_1^{\sigma_2}(H). $$
    Therefore, we have
    \begin{align*} 
    \nu(G \times H) &= \sup_\sigma \lambda_1^\sigma(G \times H) \\ &\ge \sup_{\sigma_1,\sigma_2} \left(\lambda_1^{\sigma_1}(G) + \lambda_1^{\sigma_2}(H) \right) \\ &= \sup_{\sigma_1} \lambda_1^{\sigma_1}(G) + \sup_{\sigma_2 }\lambda_1^{\sigma_2}(H) = \nu(G) + \nu(H). 
    \end{align*}
\end{proof}

\subsection{Suspensions}
\label{sec:suspensions}

Regarding suspensions, our first result reads as follows, by employing the result from Theorem \ref{thm:outdegree_upper_bound} about subgraphs.

\begin{theorem} \label{thm:suspgraph}
Let $G$ be an arbitrary graph without isolated vertices and $\tilde G$ be its suspension. Then we have
\begin{equation} \label{eq:nutildeGnuG} 
1 \le \nu(\tilde G) \le \nu(G)+1. 
\end{equation}
In particular, we have 
\begin{equation} \label{eq:nutildeT} 
\nu(\tilde T) = 1 
\end{equation}
for the suspension $\tilde T$ of every tree $T$ with at least two vertices.
\end{theorem}

\begin{proof} 

For the upper bound $\nu(\tilde G) \le  \nu(G)+1$, we consider $G$ as the subgraph of $\tilde{G}$. For each vertex $x\in G$, we have $d_x^{out}=1$, and the upper bound follows from Theorem \ref{thm:outdegree_upper_bound}(b).

Regarding the lower bound for $\nu(\tilde G)$, we can assume without loss of generality that $G$ is connected: This follows from Lemma \ref{lem:splitvertex}  and the fact that, via splitting of vertices, the suspension of a union of disjoint graphs can be transformed into the disjoint union of the suspension of these disjoint graphs. So we are left to consider a connected graph $G$ with at least $2$ vertices. Moreover, it is sufficient to prove the lower bound $1$ for $\nu(\tilde G)$ for trees with at least $2$ vertices, due to Lemma \ref{lem:addedge}. The result \eqref{eq:nutildeT} for the suspension of trees $T$ with at least $2$ vertices follows then immediately from \eqref{eq:nutildeGnuG} and $\nu(T)=0$.

Let us now provide the proof of the lower bound in \eqref{eq:nutildeGnuG} for a tree $T=(V,E)$ with $V = \{x_1,\dots,x_t\}$, $t \ge 2$. Let $x$ be the suspension vertex of $\tilde T$. Any $\sigma: E^{or}(\tilde T) \to S^1$ can be switched to a magnetic potential which is trivial on the edges of the tree and assumes the values $a_j = -e^{i \alpha_j}$ on the directed edge 
$(x,x_j) \in E^{or}(\tilde T)$. The corresponding Laplacian is represented by the following matrix:
$$ \Delta^\sigma \cong \begin{pmatrix} \begin{array}{c|ccccc}t & a_1 & a_2 & \cdots & a_{t-1} & a_t \\
\hline
\bar a_1 & & & & & \\
\bar a_2 & & & & & \\
\vdots & & & \Delta_T + {\rm{Id}} & & \\
\bar a_{t-1} & & & & & \\
\bar a_t & & & & &
\end{array}
\end{pmatrix},
$$
where $\Delta_T$ is the non-normalized standard Laplacian
of the tree $T$. Under the assumption that the values $a_1,\dots,a_t \in S^1$ satisfy
\begin{equation} \label{eq:sumaj}
\sum_{j=1}^t a_j = 0,
\end{equation}
the vector $(0,1,\dots,1)^\top$ is an eigenvector of $\Delta^\sigma$ to the eigenvalue $1$ since any constant vector is an eigenvector of $\Delta_T$ to the eigenvalue $0$. It remains to prove that $1$ is the smallest eigenvalue of $\Delta^\sigma$ for a suitable choice of $a_1,\dots,a_t \in S^1$ satisfying \eqref{eq:sumaj}, since we already know from Example \ref{ex:treesuspension} that $\nu(\tilde T) \le 1$. We prove this by showing that we can choose $a_1,\dots,a_t \in S^1$ such that 
$$ \Delta^\sigma - {\rm{Id}} \cong \begin{pmatrix} \begin{array}{c|ccccc}t-1 & a_1 & a_2 & \cdots & a_{t-1} & a_t \\
\hline
\bar a_1 & & & & & \\
\bar a_2 & & & & & \\
\vdots & & & \Delta_T & & \\
\bar a_{t-1} & & & & & \\
\bar a_t & & & & &
\end{array}
\end{pmatrix} $$
is positive semidefinite. The corresponding quadratic form for the complex variables $z,z_1,\dots,z_t \in \C$ is given by
\begin{equation} \label{eq:Q}
Q=(t-1) |z|^2 + \sum_{j=1}^t (a_j z \bar z_j + \bar a_j \bar z z_j) + \sum_{j=1}^t d_{x_j} |z_j|^2 - \sum_{\{x_j,x_k\} \in E} (z_j \bar z_k + \bar z_j z_k).
\end{equation}
We consider $T$ to be a rooted tree with a leaf root $x_1$ and assume the other vertices $x_2,\dots,x_t$ to be ordered with respect to the distance to the root. Note that every vertex in $T$ has at most one parent and potentially several children. Let $I$ be the set of index pairs $(k,l)$ corresponding to oriented edges $x_k \sim x_l$ in $T$ with $x_k$ the parent of $x_l$. Let 
$$ S_6 = \{\pm 1, \pm \omega, \pm \omega^2\} $$
be the set of all $6$-th roots of unity with $\omega = e^{2\pi i/3}$. The following claim allows us to find a suitable choice of $a_1,\dots,a_t \in S^1$ with the required properties.

\medskip

{\bf{Claim:}} There exist a family of values $\{ b_{jk} \}_{(j,k) \in I}$ in $S_6$ which satisfies the following: For any vertex $x_j \in T$ we have  
\begin{equation} \label{eq:tildeaj} 
a_j := \left( \sum_{k: (j,k) \in I} b_{jk}\right) - b_{j_0j} \in S_6, 
\end{equation}
where $x_{j_0} \in T$ is the parent of $x_j$.
Here we use the convention that, for vertices without parents or without children, the corresponding term is chosen to be zero.

\medskip

\begin{figure}
\begin{center}
\begin{tikzpicture}[thick,scale=1]
\filldraw [black] (0,5) circle (4pt);
\filldraw [black] (2,5) circle (2pt);
\filldraw [black] (4,3) circle (2pt);
\filldraw [black] (4,7) circle (2pt);
\filldraw [black] (6,1) circle (2pt);
\filldraw [black] (6,3) circle (2pt);
\filldraw [black] (6,5) circle (2pt);
\filldraw [black] (6,7) circle (2pt);
\draw (0,5) node[label=above:$a_1$]{};
\draw (2,5) node[label=above:$a_2$]{};
\draw (4,7) node[label=above:$a_3$]{};
\draw (4,3) node[label=above:$a_4$]{};
\draw (6,7) node[label=above:$a_5$]{};
\draw (6,5) node[label=above:$a_6$]{};
\draw (6,3) node[label=above:$a_7$]{};
\draw (6,1) node[label=below:$a_8$]{};
\draw (1,5) node[label=above:$b_{12}$]{};
\draw (3,6.1) node[label=above:$b_{23}$]{};
\draw (3,3) node[label=above:$b_{24}$]{};
\draw (5,7) node[label=above:$b_{35}$]{};
\draw (5,4) node[label=above:$b_{46}$]{};
\draw (5,2.8) node[label=above:$b_{47}$]{};
\draw (5,1.2) node[label=above:$b_{48}$]{};
\draw (0,5) -- (2,5);
\draw (2,5) -- (4,7);
\draw (2,5) -- (4,3);
\draw (4,7) -- (6,7);
\draw (4,3) -- (6,5);
\draw (4,3) -- (6,3);
\draw (4,3) -- (6,1);
\end{tikzpicture}
\end{center}
\caption{The relation between the values $a_j$ and $b_{jk}$. For example, we have $a_2 = b_{23}+b_{24}-b_{12}$. \label{fig:ajbjk}}
\end{figure}

Before proving the claim, we now explain how to derive positive semidefiniteness of the matrix $\Delta^\sigma - {\rm{Id}}_{t+1}$ with a magnetic potential corresponding to the $a_j$'s, derived from the $b_{jk}$'s via \eqref{eq:tildeaj} (see Figure \ref{fig:ajbjk} for a visual illustration of these values). We first need to verify that we have $\sum a_j = 0$. This can be verified as follows:
\begin{align*} 
\sum_j a_j &= \left( \sum_{j \in V} \sum_{k: (j,k) \in I} b_{jk} \right) - \left( \sum_{j \in V} b_{j_0(j)j} \right) \\ 
&= \left( \sum_{(j,k) \in I} b_{jk} \right) - \left( \sum_{(j,k) \in I} b_{jk}\right) = 0
\end{align*}
where $x_{j_0(j)}$ is the parent of $x_j$.
Next we show positive semidefiniteness of the quadratic form $Q$ in \eqref{eq:Q}: We have
\begin{align*}
0 &\le \sum_{(j,k) \in I} |b_{jk} z + z_j - z_k|^2 \\
&= \left( \sum_{(j,k) \in I} |b_{jk}|^2 \right) |z|^2 + \sum_{j=1}^t d_{x_j} |z_j|^2 - \sum_{\{x_j,x_k\} \in E} (z_j \bar z_k + \bar z_j z_k) \\
& \,\,\,\,\,+\sum_{j =1}^t \underbrace{\left( \sum_{k: (j,k) \in I} b_{jk} - b_{j_0(j)j} \right)}_{= a_j} z \bar z_j + \sum_{j =1}^t \underbrace{\left( \sum_{k: (j,k) \in I} \bar b_{jk} - \bar b_{j_0(j)j} \right)}_{= \bar a_j} \bar z z_j \\
& = Q, 
\end{align*}
since $|I| = t-1$ and $|b_{jk}| = 1$.

\smallskip

It remains to show the claim, which we prove via Induction. The base case is $T=K_2$ (that is, $t=2$). In this case we choose $b_{12}=1$ and obtain $a_1=1$, $a_2=-1$. Assume the claim holds for any tree of size $t \ge 2$. Let $T$ be a rooted tree with root $x_1$ and $t+1$ vertices. Since $x_1$ is a leaf, it has precisely one child, namely $x_2$. Let $d=d_{x_2}$. Then $x_2$ has precisely $d-1$ children $x_3,\dots,x_{d+1}$, and we can consider the rooted subtrees $T_j$, $j=3,4,d+1$ of $T$, where $T_j$ has the leaf root $x_2$ and is connected to $x_j$. Hence, the trees $T_j$ have only one vertex in common, namely, $x_2$.
By induction hypothesis, each $T_j$ has a system of variables $\{b_{jk}\}$ corresponding to the edges of $T_j$, satisfying the claim. It is easy to see that there exist $b,s_3,\dots,s_{d+1} \in S_6$ such that
$$ \left( \sum_{j=3}^{d+1} s_j b_{2j} \right) -b \in S_6, $$
by choosing the factors $s_j$ such that two consecutive terms $s_{2l+1} b_{2l+1}$ and $s_{2l+2} b_{2l+2}$ in the above sum cancel each other out and using the fact that $\xi_6^l+\xi_6^{l+2}+\xi_6^{l+4}=0$ for the choice of $b$. A suitable family of $\{ b_{jk} \}$ for the tree $T$ is now given by $b_{12}=b$ and $b_{kl} = s_j b_{kl}$ for edges $\{x_k,x_l\}$ in the subtree $T_j$. This leads to $a_1 = b_{12} \in S_6$ and $a_2 = \sum s_j b_{2j} - b \in S_6$, completing the induction step.
\end{proof}

The next results are concerned with graphs $G$ whose suspensions $\tilde G$ satisfy \eqref{eq:nutildeGnuG} with equality, that is $\nu(\tilde G)= \nu(G)+1$. Recall that we have shown in Example \ref{ex:completegraph} and Example \ref{ex:wheel} that $\nu(\tilde G) = \nu(G)+1$ holds for complete graphs and cycles.

\begin{proposition} \label{prop:maxsusprel} Let $G=(V,E)$ be a graph and $\tilde G$ be its suspension. If 
$$ \nu(\tilde G) = \nu(G)+1, $$
then the restriction $\sigma$ of any maximal potential $\tilde \sigma$ of $\tilde G$ is maximal in $G$. Moreover, any $\lambda_1^\sigma(g)$-eigenfunction $f: V \to \mathbb{C}$ of $G$ can be trivially extended to a $\lambda_1^{\tilde \sigma}(\tilde G)$-eigenfunction $\tilde f: \tilde V \to \mathbb{C}$ of $\tilde G$.

Conversely, if $G$ has a unique maximal potential $\sigma$, up to gauge equivalence, then this potential can be extended to a maximal potential $\tilde \sigma$ on $\tilde G$.
\end{proposition}

\begin{proof}
  Let $\sigma, \tilde \sigma, f, \tilde f$ be as in the proposition. Then we have
  \begin{multline*}
  \nu(\tilde G) = \lambda_1^{\tilde \sigma}(\tilde G) \le \mathcal{R}^{\tilde \sigma}(\tilde f) = \frac{\sum_{\{u,v\} \in E} |f(u)-\sigma(u,v)f(v)|^2 + \sum_{u \in V} |f(u)|^2}{\sum_{u \in V} |f(u)|^2} = \\ \lambda_1^\sigma(G) + 1 \le \nu(G)+1.
  \end{multline*}
  For the converse direction and a given maximal potential $\sigma$ in $G$, we choose an arbitrary maximal potential $\tilde \sigma_0$ on $\tilde G$ and know that its restriction $\sigma_0$ to $G$ is gauge equivalent to $\sigma$: $\sigma_0 = \sigma^\tau$. Choose the extension $\tilde \tau: \tilde V \to S^1$ of $\tau$ with $\tilde \tau(x_0) = 1$ for the suspension vertex $x_0 \in \tilde V$. Then the potential of $\tilde \sigma:=\tilde \sigma_0^{\tau^{-1}}$ is maximal on $\tilde G$ and an extension of the maximal potential $\sigma$ of $G$.
\end{proof}

Let us present the following application of this result, which generalizes the observation in Example \ref{ex:countexquest} (where $G$ is the suspension of the star graph $S_2$). Moreover, this application shows the surprising fact that there exist finite graphs with uncountably many non gauge equivalent maximal potentials.

\begin{corollary} \label{cor:suspstargraph} Let $S_d$ be the star graph with center $x_1$ and $d$ leaves $x_2,\dots,x_{d+1}$, and $\tilde S_d$ be its suspension with suspension vertex $x_{d+2}$. Then a magnetic potential is maximal if and only if it is gauge equivalent to a potential $\tilde \sigma$ satisfying
\begin{equation} \label{eq:maxpotstar1}
\tilde \sigma(x_1,x_j) = 1 \quad \text{for $j=2,\dots,d+1$}, 
\end{equation}
and
\begin{equation} \label{eq:maxpotstar2}
\tilde \sigma(x_j,x_{d+2}) = a_j \in S^1 \quad \text{for $j=1,\dots,d+1$}, 
\end{equation}
with
\begin{equation} \label{eq:sumaj} 
\sum_{j=1}^{d+1} a_j = 0. 
\end{equation}
In particular, if $d \ge 3$, $\tilde S_d$ has uncountably many non gauge equivalent maximal potentials.
\end{corollary}

\begin{proof} Note first that any potential $\tilde \sigma$ on $\tilde S_d$ is gauge equivalent to one satisfying the conditions \eqref{eq:maxpotstar1} an \eqref{eq:maxpotstar2}. Therefore, we can restrict our consideration to such potentials. Since $S_d=(V,E)$ is a tree, we have $\nu(\tilde S_d) = 1$ by Theorem \ref{thm:suspgraph}. 

Let $\tilde \sigma$ be a maximal potential of $\tilde S_d$ satisfying \eqref{eq:maxpotstar1} and \eqref{eq:maxpotstar2}. The restriction of $\tilde \sigma$ to $S_d$ is the trivial potential $\sigma_0$, and the constant $1$-function $f$ is an eigenfunction of $\lambda_1^{\sigma_0}(S_d)=0$. Applying Proposition \ref{prop:maxsusprel}, the trivial extension $\tilde f: \tilde V \to \{0,1\}$ of $f$ is an eigenfunction of $\lambda_1^{\tilde \sigma}(\tilde S_d)$. The eigenvalue equation at the suspension vertex $x_{d_2}$ is 
$$ 0 = \Delta f(x_{d+2}) =  \sum_{j=1}^{d+1} (f(x_{d+2})-\bar a_j f(x_j)) = - \sum_{j=1}^{d+1} \bar a_j. $$
Hence, \eqref{eq:sumaj} is a necessary condition for $\tilde \sigma$ satisfying \eqref{eq:maxpotstar1} and \eqref{eq:maxpotstar2} to be a maximal potential. For the converse direction, we show that $1$ is the smallest $\Delta^{\tilde \sigma}$-eigenvalue of any potential $\tilde \sigma$ satisfying \eqref{eq:maxpotstar1}, \eqref{eq:maxpotstar2} and \eqref{eq:sumaj}. The matrix representation of $\Delta^\sigma - {\rm{Id}}$ with respect to the vertex enumeration $x_1,x_{d+2},x_2,\dots,x_{d+1}$ is of the form
$$ \Delta^\sigma - {\rm{Id}} \cong \begin{pmatrix} \begin{array}{cc|ccc}d & -a_1 & -1 & \cdots & -1 \\
-\bar a_1 & d & -\bar a_2 & \cdots & - \bar a_{d+1} \\
\hline
-1 & -a_2 & 1 & \cdots &0 \\
\vdots & \vdots & \vdots & \ddots & \vdots \\
-1 & - a_{d+1} & 0 & \cdots & 1
\end{array}
\end{pmatrix}, $$
and is easy to see that the vector $(1,0,1,\dots,1)^\top$ is an eigenvector to the eigenvalue $0$ of this matrix.
This matrix is positive semidefinite if and only if its Schur complement 
$$ \begin{pmatrix} d & -a_1 \\ -\bar a_1 & d \end{pmatrix} -
\begin{pmatrix} -1 & \cdots & -1 \\ - \bar a_2 & \cdots & - \bar a_{d+1} \end{pmatrix} \begin{pmatrix} -1 & - a_2 \\ \vdots & \vdots \\ -1 & - a_{d+1} \end{pmatrix} = \begin{pmatrix} 0 & - \sum_{j=1}^{d+1} a_j \\ -\sum_{j=1}^{d+1} \bar a_j & 0 \end{pmatrix}
$$
is positive semidefinite. Therefore, the conditions \eqref{eq:maxpotstar1}, \eqref{eq:maxpotstar2} and \eqref{eq:sumaj} imply that $\lambda_1^{\tilde \sigma}(\tilde S_d) = \nu(\tilde S_d) = 1$.
\end{proof}

Our next result is concerned with the suspension of disjoint unions of graphs.

\begin{proposition} Let $G = \{G_1,\dots,G_k\}$ be the disjoint union of $k$ finite graphs $G_1,\dots,G_k$. Let $\tilde G$ and $\tilde G_j$, $j=1,\dots,k$,  denote the suspension of the graphs $G$ and $G_j$, respectively. If each graph $G_j$ satisfies
\begin{equation} \label{eq:tGjplus1} 
\nu(\tilde G_j) = \nu(G_j) + 1, 
\end{equation}
then we also have
$$ \nu(\tilde G) = \min \{ \nu(\tilde G_1),\dots,\nu(\tilde G_k)) = \nu(G) + 1.
$$
\end{proposition}

\begin{proof} Similarly as in the proof of Theorem, this result is again a straightforward consequence of splitting of vertices and Lemma \ref{lem:splitvertex}, applied to suspensions.
\end{proof}

\begin{example} \label{ex:susp1} Let us consider the suspension $\tilde G$ of the graph $G$ in Example \ref{example:suspension} with $\nu(G) = \frac{5-\sqrt{17}}{2} \approx 0.438447187$. If we denote the suspension vertex by $x_5$, it can be checked that the magnetic potential $\hat \sigma: E^{or}(\tilde G) \to S^1$, which agrees with the maximal potential $\sigma_\pi: E^{or} \to S^1$ of $G$ in Figure \ref{fig:triangle_with_dangling_edge} and
$$ \hat \sigma(x_5,x_1) = -1, \quad \hat(\sigma(x_5,x_2) = \hat \sigma(x_5,x_3) = \hat \sigma(x_5,x_4) = +1 $$
satisfies
$$ \lambda_1^{\hat \sigma}(\tilde G) = \frac{7-\sqrt{17}}{2} = \nu(G) +1, $$
and is therefore a maximal potential, by \eqref{eq:nutildeGnuG} in Theorem \ref{thm:suspgraph}. Note that the eigenvalue $\lambda_1^{\hat \sigma}(\tilde G)$ of $\Delta^{\hat \sigma}$ has multiplicity $2$, while the eigenvalue $\lambda_1^{\sigma_\pi}(G) $ of $\Delta^{\sigma_\pi}$ had multiplicity $1$. Moreover, numerics suggest that $\hat \sigma$ is the unique maximal potential of $\tilde G$, up to gauge equivalence. 
\end{example}

\begin{example} 
\label{ex:susp2} Let us consider the suspension $\tilde G$ of the graph $G$ in Example \ref{ex:countexquest} with $\nu(G) = 1$ (see Figure \ref{fig:twotriangles}). If we denote the suspension vertex by $x_5$, it can be checked that the magnetic potential $\hat \sigma: E^{or}(\tilde G) \to S^1$, which agrees with the maximal potential $\sigma: E^{or}(G) \to S^1$ of $G$ in \eqref{eq:twotrianglespotmax} for $j=2$ and 
$$ \hat \sigma(x_5,x_1) = \hat \sigma(x_5,x_3)^{-1} = - \xi_6, \quad \hat \sigma(x_5,x_2) = \hat \sigma(x_5,x_4) = +1 $$
satisfies
$$ \lambda_1^{\hat \sigma}(\tilde G) = 2, $$
and is therefore a maximal potential, by \eqref{eq:nutildeGnuG} in Theorem \ref{thm:suspgraph}. Thus we have 
$$ \nu(\tilde G) = 2 = \nu(G) + 1. $$
Moreover, numerics suggest that this potential is the unique maximal potential of $\tilde G$, up to gauge equivalence.
\end{example}

All above examples indicate that the relation $\nu(\tilde G) = \nu(G)+1$ holds for a large class of graphs (see Problem \ref{prob:susp} in the final section). However, there are examples for which this relation does not hold. In the case of $k$ isolated vertices, the suspension $\tilde G$ is the star graph satisfying $\nu(\tilde G)=0$. An example with $\nu(\tilde G) < \nu(G) + 1$ and $\nu(\tilde G) > 0$ is the following.

\begin{remark}\label{remark:suspension}
    Example \ref{example:suspension} can be regarded as a suspension $\tilde G$ of a graph $G = \{K_2, x_1\}$. In this case, we know that the magneto-spectral heights are $\nu(\tilde G) \approx 0.438447187$ and $\nu(G) = 0$, respectively. Therefore, we have a strict inequality 
    $$\nu(\tilde G) < \nu(G) + 1. $$
\end{remark}

Generally, we have the following characterizations of graphs $G$ with $\nu(\tilde G) < 1$.

\begin{proposition}\label{prop:nususpsm1} We have $\nu(\tilde G) < 1$ if and only if $G$ contains an isolated vertex. Moreover, we have $0 < \nu(\tilde G) < 1$ if and only if $G$ contains both an isolated vertex and at least one edge.
\end{proposition}

\begin{proof} If $G$ contains an isolated vertex, we have $d_{\min}(\tilde G)=1$, and it follows from \eqref{eq:nuleaf} in Corollary \ref{cor:nudmin} that $\nu(\tilde G)<1$. Conversely, if $\nu(\tilde G) < 1$, the lower bound of \eqref{eq:nutildeGnuG} in Theorem \ref{thm:suspgraph} implies that $G$ must contain an isolated vertex. Since $\tilde G$ is always connected, $\nu(\tilde G)=0$ holds if and only if $\tilde G$ is a tree. The final statement follow now from the fact that $\tilde G$ can only be a tree if and only if $G$ does not contain an edge.
\end{proof}

\section{Relation to the spectral gap} \label{sec:specgap}

In this section, we discuss relations between the magneto-spectral height $\nu(G)$ and the spectrum of the standard Laplacian $\Delta_G$, in particular the extremal non-trivial eigenvalues $\lambda_2(G), \lambda_{\max}(\Delta_G)$. Regarding the largest eigenvalue, we have the following relation.

\begin{proposition} \label{prop:nudminlambdamax} For any finite graph $G$, the magneto-spectral height satisfies
$$ \nu(G) \ge 2 d_{\min}(G) - \lambda_{\max}(\Delta_G). $$
\end{proposition}

Note that for $d$-regular graphs, this proposition is only meaningful if no connected component of $G$ is bipartite since, otherwise, $\lambda_{\max}(\Delta_G)=2d$.

\begin{proof} Let $\bar \sigma: E^{or}(G) \to \{-1\}$ be the anti-balanced potential of $G$. Then we have
\begin{align*}
\nu(G) &\ge \lambda_1^{\bar \sigma}(G) = \lambda_1(D-A^{\bar \sigma}) = \lambda_1(2D-(D-A)) \\ &\ge 2 d_{\min}(G) - \lambda_{\max}(D-A) = 2 d_{\min}(G) - \lambda_{\max}(\Delta_G).
\end{align*}
\end{proof}

In the rest of this section, we discuss relations between $\nu(G)$ and the spectral gap $\lambda_2(G)$ of connected graphs $G$.
In the case of the normalized Laplacian $\widetilde \Delta$ with corresponding eigenvalues $\widetilde \lambda_j(G)$, Theorem 1.1 in \cite{Sa23} provides for connected, non-bipartite vertex transitive graphs $G$ of degree $d$ the relation
\begin{equation} \label{eq:nulambda2} 
2 - \widetilde \lambda_{\max}(\widetilde \Delta) \ge C \cdot \frac{\widetilde \lambda_2(G)}{d}, 
\end{equation}
with a universal constant $C > 0$. Since $\Delta = d \widetilde \Delta$ for these graphs, this, together with Proposition \ref{prop:nudminlambdamax}, leads to the estimate
$$ \nu(G) \ge \frac{C}{d} \cdot \lambda_2(G). $$
Limitations of this inequality are that we have no knowledge about the value of $C$, that the inequality is restricted to a special class of $d$-regular graphs. (Note that the whole spectrum of $\Delta$ is bounded above by $2d$.) A more desirable inequality would be an estimate of the form
$$ \nu(G) \ge c_d \cdot \lambda_2(G) $$
with concrete universal constants $c_d>0$ for $d \ge 2$, which holds for all connected graphs $G$ satisfying $d_{\min}(G) \ge d$. We require $d \ge 2$ to exclude trees $T$, for which we have $\nu(T) =0$. Proposition \ref{prop:nudiam} can be viewed as some kind of support of such an inequality, since there exists an equivalent lower diameter bound for the spectral gap:

\begin{theorem}[see {\cite[Theorem 3.5]{CLY14}}] Let $G=(V,E)$ be a connected graph satisfying $CD(0,\infty)$. Then we have
$$ \lambda_2(G) \ge \frac{C}{{\rm{diam}}^2(G)} $$
with a universal constant $C$.
\end{theorem}

The opposite inequality
\begin{equation} \label{eq:lambda2nu} 
\lambda_2(G) \ge C \cdot \nu(G)
\end{equation}
holds true for $2$-regular connected graphs, since we have  have
$\lambda_2(C_n) = 2 - 2\cos(2\pi/n)$, and therefore
$$ \frac{1}{4} \lambda_2(C_n) \le \nu(C_n) = 2 - 2\cos(\pi/n) \le \frac{1}{3} \lambda_2(C_n) $$
for all $n\ge 3$. However, 
the spectral gap and the magneto-spectral height have different behaviors for Cartesian products, namely
\begin{align*}
    \lambda_2(G \times H) &= \min\{ \lambda_2(G),\lambda_2(H) \}, \\
    \nu(G_1 \times G_2) &\ge \nu(G)+\nu(H),
\end{align*}
and \eqref{eq:lambda2nu} cannot hold with a universal constant $C>0$ for arbitrary connected graphs. Indeed, considering the hypercubes $Q^d=(K_2)^d$, we have
\begin{align*}
\lambda_2(Q^d) &= \lambda_2(K_2) = 2, \\
\nu(Q^d) &= d - \sqrt{d} \to \infty \quad \text{as $d \to \infty$.}
\end{align*}
Even for $d$-regular connected graphs $G$, there cannot be constants $c_d > 0$, only depending on $d$, such that
\begin{equation} \label{eq:lambda2cdnu} 
\lambda_2(G) \ge c_d \cdot \nu(G), 
\end{equation}
since we have the following counterexample.

\begin{example}[Cycles of almost complete graphs] \label{ex:cycalmcomplete} Let $n \ge 4$ be fixed and $K_n^-$ be the complete graph $K_n$ with one edge removed. This graph has precisely two vertices $x,y$ of degree $n-2$. For $\ell \ge 1$, we choose $2\ell$ copies $(K_n^-)^j$, $j=1,\dots,2\ell$, of these graphs with the corresponding special vertices $x_j,y_j$. We construct an $(n-1)$-regular graph $G_\ell$ by connecting these graphs cyclically along their special vertices, that is, we add edges between $y_j$ and $x_{j+1}$ for all $j = 1,\dots,2\ell$ (modulo $2\ell$). Since the additional edges $\{y_{\ell},x_{\ell+1}\}$ and $\{y_{2\ell},x_1\}$ form a cut-set of $G_{\ell}$, the Cheeger constant of these graphs satisfies
$$ h(G_\ell) \le \frac{2}{n \ell} \to 0 \quad \text{as $\ell \to \infty$,} $$
and we have, by Cheeger's Inequality,
$$ \lambda_2(G_{\ell}) \le 2h(G_{\ell}) \to 0 \quad \text{as $\ell \to \infty$.} $$
On the other hand, Proposition \ref{prop:MSS} yields
$$ \nu(G_\ell) \ge (n-1) - 2 \sqrt{n-2} > 0, $$
contradicting to \eqref{eq:lambda2cdnu}.
\end{example}

\begin{remark} Connected $3$-regular graphs of fixed vertex cardinality $n$ and minimal spectral gap were characterized in \cite[Theorem 1]{BGI07}, and their spectral gaps have the asymptotics $(1+o(1))\frac{2\pi^2}{n^2}$, as $n \to \infty$ (see \cite[Theorem 2.2]{AGI21}. 
\end{remark}

\section{Some open problems} \label{sec:openprobs}

We end this paper with a collection of questions about the magneto-spectral height.

\begin{problem}[Multiplicity Problem] Is there a characterization of connected graphs for which there exists a maximal magnetic potential $\hat \sigma$ such that $\lambda_1^{\hat\sigma}(G)$ has higher multiplicity? 
Trees are not in this class of graphs and also not $K_3$ with an attached dangling edge (Example \ref{example:suspension}). However, cycles, wheel graphs and complete graphs are in this class of graphs. This question is somehow in the same spirit as the \emph{Colin de Verdi{\`e}re Invariant} of graphs. More generally, one might investigate the $\lambda_1^\sigma$-surface of a finite graph $G$ over the space of all non-equivalent magnetic potentials and understand its critical points by using Morse Theory in the same spirit as in \cite{B13}. 
\end{problem}

\begin{problem}[Universal Lower Bound Problem] For $d$-regular graphs with $d \ge 3$, we have the lower Ramanujan bound $\nu(G) \ge d-2\sqrt{d-1}$ and the magneto-spectral height is non-decreasing by adding edges between existing vertices. We also know that we have for all $d \ge 7$,
\begin{equation} \label{eq:Cdlowbd}
0 < d-2\sqrt{2}\sqrt{d-1} \le \nu_d \le d-2\sqrt{d-1},
\end{equation}
where $\nu_d$ is the infimum of magneto-spectral heights of all finite graphs $G$ satisfying $d_{\min}(G) \ge d$  (Corollary \ref{cor:Cddef}). 
However, we do not know whether $\nu_3 > 0$, that is, there exists also a universal lower bound for all graphs with $d_{\min}(G) \ge 3$. It is also not clear whether \eqref{eq:Cdlowbd} can be improved to
$$ \nu_d = d-2\sqrt{d-1}, $$
since a graph with $d_{\min}(G) \ge d$ cannot be usually reduced to a $d$-regular graph by just removing edges between existing vertices.
\end{problem}

\begin{problem}[Spectral Gap Problem] Do there exist universal constants $c_d>0$ for $d\ge 2$, such that $\lambda_2(G) \le c_d \nu(G)$ for all connected graphs $G$ with $d_{\min}(G) \ge d$?
\end{problem}

\begin{problem}[Suspension Problem]
    \label{prob:susp}
    We know that the relation
    $$\nu(\tilde G) = \nu(G) + 1$$
    between a graph $G$ and its suspension $\tilde G$ holds true for many graphs. The only counterexamples we are aware of, are graphs with isolated vertices. Is it true that these are the only counterexamples for this relation? 
\end{problem}

\begin{problem}
[Algebraicity Problem] Does there exist a graph $G$ for which
$\nu(G)$ is transcendental, or is $\nu(G)$ for all graphs $G$ an algebraic number? 
\end{problem}

\begin{problem}[Regular Bipartite Graph Problem] We know that the magneto-spectral height of $d$-regular bipartite graphs is bounded above by $d-\sqrt{d}$ (see Corollary \ref{cor:Kddmagnetospecheight}). This upper bound is attained for hypercubes $Q^d$, complete bipartite graphs $K_{d,d}$ and the examples given in Remark \ref{rem:UWnd}. We know that this class of graphs is closed under taking Cartesian products as well as the $\odot$-product introduced in Remark \ref{rem:UWnd}. Note that it follows from \eqref{eq:sigmasigma} that a necessary condition is that any pair of distance-2 vertices must have at least $2$ common neighbours. Therefore, even cycles $C_d$ with $d \ge 6$ do not be belong to this class, and nor does any regular bipartite graph with girth $\ge 6$.

The complete classification of $d$-regular bipartite graphs $G$ with $\nu(G) = d-\sqrt{d}$ seems to be a hard but interesting problem, since it would agree with the classification of all unit weighing matrices. At present there are not too many papers in the literature about this particular class of matrices generalizing Hadamard matrices.

Another interesting question is whether all graphs in this class have (standard) non-negative Bakry-\'Emery or Ollivier Ricci curvature. The magnetic Bakry-\'Emery curvature of any maximal potential is non-negative, by Theorem \ref{thm:Chunyang}. 
\end{problem}

\begin{problem}[Petersen Graph Problem] The $\Delta$-spectrum of the Petersen graph $G$ is
\begin{center}
\begin{tabular}{cccccccccc}
$\lambda_1$ & $\lambda_2$ & $\lambda_3$ & $\lambda_4$ & $\lambda_5$ & $\lambda_6$ & $\lambda_7$ & $\lambda_8$ & $\lambda_9$ & $\lambda_{10}$\\
\hline
$0$ & $2$ & $2$ & $2$ & $2$ & $2$ & $5$ & $5$ & $5$ & $5$
 \end{tabular}
\end{center}
Since this graph is regular with $d=3$, $\lambda_1^{\bar \sigma}(G)$ for the antibalanced signature $\bar \sigma \equiv -1$ satisfies 
$$
\lambda_1^{\bar \sigma}(G) = 2d-\lambda_{10} = 1. $$
Numerical experiments suggest that $\nu(G)=1$, and that $\bar \sigma$ is the unique maximal potential, up to gauge equivalence. Can one give a rigorous proof of these statements? Proposition \ref{prop:alonbopp} provides
the ``almost optimal'' upper bound
$$ \nu(G) \le 3-2\sqrt{2} + \frac{2\sqrt{2}-1}{2} = 3 -2\sqrt{2} \approx 1.0858. $$
Moreover, it would be interesting to find further examples of connected graphs $G$ with $\nu(G)=1$, which are not suspensions of trees. 
\end{problem}

\begin{problem}[Cospectrality Problem] Recall that there is a cospectral graph $\widehat G$ to the wheel graph $W_6$ (see Figure \ref{fig:cospectral}), and that these two graphs have different magneto-spectral heights. Moreover, there are cospectral trees $T_1,T_2$ with $\nu(T_1)=\nu(T_2)=0$. Can one find a pair of cospectral finite connected graphs $G_1,G_2$, whose magneto-spectral heights are non-zero and coincide? If no such pair exists, a non-zero magneto-spectral height together with the spectrum would determine a graph, up to isomorphism.   
\end{problem}

\begin{problem} [Asymptotic Triviality Problem]
Does there exist criteria of graph sequences $G_n$ satisfying $\nu(G_n) \to 0$? Obviously, this is the case for sequences which end eventually in forests or for increasing cycles.
\end{problem}

\begin{problem}[Lift Problem] Does there exist an example of a finite lift $\widehat G \to G$ such that $\nu(\widehat G) > \nu(G)$? While eigenvalues of magnetic Laplacians are preserved under lifts of the magnetic potential and new eigenvalues may occur, which would suggest $\nu(\widehat G) \le \nu(G)$, a lift $\widehat G$ allows also for magnetic potentials which are not lifts of potentials of $G$, which may strictly increase the magneto-spectral height. However, we are not aware of any such example. 
\end{problem}

\begin{problem}[Average Magneto-spectral Height Problem] Instead of taking the supremum over all magnetic potentials in the definition \eqref{eq:magspecheight} of $\nu(G)$, one could also take the average where all magnetic potentials are chosen uniformly and independently at random. Does there exist a uniform lower bound of this average $\nu_{\rm{av}}(G)$ in terms of $\nu(G)$?  
\end{problem}

The final problem is about the magneto-spectral height of Riemannian manifolds, as introduced in Appendix A below.  

\begin{problem}[Small Eigenvalues Problem For Hyperbolic Surfaces]
Do we have for any sequence $S_n$ of compact hyperbolic surfaces with ${\rm{inj}}(S_n) \to \infty$ (where ${\rm{inj}}(S)$ is the injectivity radius of $S$) that $\limsup \nu(S_n) \le 1/4$? (This would be a smooth counterpart of the Alon-Boppana type result in Subsection \ref{subsec:ramanujan}.) Moreover, is it true that we have $\nu(S) \ge 1/4$ for any compact hyperbolic surface? The bound $1/4$ is the bottom of the spectrum of the hyperbolic plane, which is the universal cover of any hyperbolic surface. A related question is whether all compact (or finite volume) hyperbolic surfaces admit finite covers with no new eigenvalues in the interval $[0,1/4)$. For significant recent progress in this question, see \cite{HM23} and the references therein. Moreover, all these questions have natural extensions beyond constant curvature and also beyond 2-dimensional surfaces.
\end{problem}

{\bf{Acknowledgement:}} 
C. Hu, S. Kamtue and Sh. Liu are grateful for the hospitality of Durham University and N. Peyerimhoff is grateful for the hospitality of USTC and Chulalongkorn University. We also like to thank Wenbo Li and Joe Thomas for useful discussions regarding maximal abelian coverings and Ramanujan bounds. C. Hu is supported by the Outstanding PhD Students Overseas Study Support Program of the University of Science and Technology of China. B. Hua is supported by the National Natural Science Foundation of China No. 12371056. S. Kamtue is supported by grants for development of new faculty staff, Ratchadaphiseksomphot Fund, Chulalongkorn University. Sh. Liu is supported by the National Natural Science Foundation of China No. 12431004.

\appendix

\section{Magneto-spectral heights for Riemannian manifolds}

The magneto-spectral height has an analogue in the continuous setting of Riemannian manifolds. To explain this analogy, we first need to explain some further concepts in the setting of a connected graph $G=(V,E)$: Every magnetic potential $\sigma: E^{or}(G) \to S^1$ gives rise to a representation 
\begin{equation} \label{eq:reprhosigma} 
\rho_\sigma: \pi_1(G,x_0) \to S^1, 
\end{equation}
where $\pi_1(G,x_0)$ is the fundamental group of $G$ with base vertex $x_0$. Recall that $\pi_1(G,x_0)$ is the set of all equivalence classes of closed walks in $G$ starting and ending in $x_0$. This set has a natural group structure by concatenation and reversion, and the representation \eqref{eq:reprhosigma} is defined as follows:
$$ \rho_\sigma(\gamma) = \sigma(x_0,x_1)\sigma(x_1,x_2)\cdots\sigma(x_{m-1},x_m), $$
where $\gamma$ is the equivalence class (modulo backtracking) of the closed walk
$$ x_0 \sim x_1 \sim x_2 \sim \cdots \sim x_{m-1} \sim x_m=x_0. $$
Moreover, the universal cover of $G$ is a tree $T=(\widetilde V, \widetilde E)$ with covering map $\pi: T \to G$, on which the fundamental group $\Gamma = \pi_1(G,x_0)$ acts by deck transformations. Vertices in $T$ can be identified with equivalence classes of walks in $G$ starting at $x_0$. Let $\tilde x_0 \in \pi^{-1}(x_0)$ be the vertex in $\tilde V$ corresponding to the ``empty'' walk.

Let $\rho: \Gamma \to S^1$ an arbitrary representation.
We can identify functions $f \in C(V,\C)$ with $\rho$-equivariant functions $\vec{f} \in C(\tilde V,\C)$ as follows:
$$ \vec{f}(\tilde x_0) = f(x_0) $$
and
$$
\vec{f}(\tilde y)=\sigma(\gamma_{\tilde y})f(x_0),
$$
where $\gamma_{\tilde y}$ is a walk in the equivalence class corresponding to $\tilde y \in T$. In particular, we have
$$
\vec{f}(\gamma x) =\rho_\sigma(\gamma) \vec{f}(x) \quad \text{for all $x \in \tilde V$ and $\gamma \in \Gamma$.} $$
The set of all $\rho$-equivariant functions on $T$ is denoted by
$$ C_\rho(\widetilde V,\C) = \{ g: \widetilde V \to \C: g(\gamma x) = \rho(\gamma)g(x) \quad \text{for all $\gamma \in \Gamma$ and $x \in \widetilde V$} \}, $$
and the bijective map $f \mapsto \vec{f}$ is henceforth denoted by $\mathcal{F}_\rho: C(V,\C) \to C_\rho(\widetilde V, \C)$. Then we have the following commutative diagram relating the magnetic Laplacian $\Delta^\sigma: C(V,\C) \to C(V,\C)$ and the standard Laplacian $\Delta_T: C(\tilde V,\C) \to C(\tilde V,\C)$ on the tree $T$ with $\rho=\rho_\sigma$:
$$\begin{CD}
   C(V,\C) @>\Delta^\sigma>> C(V,\C) \\
   @VV{\mathcal{F}_\rho}V          @VV{\mathcal{F}_\rho}V \\
   C_\rho(\widetilde V,\C) @>>\Delta_T> C_\rho(\widetilde V,\C)
\end{CD}
$$
In the special case that the representation $\rho=\rho_\sigma$ factors through the finite subgroup $S_k$ of powers of the $k$-th root of unity $\xi_k=e^{2\pi i/k}$, that is, $\rho: \pi_1(G,x_0) \to S_k$, the diagram descends to a corresponding commutative diagram between $\Delta^\sigma$ on $G$ and the standard Laplacian $\Delta_{\widehat G}$ on the cyclic $k$-lift $\widehat G=(\widehat V,\widehat E)$ corresponding to $\sigma$:
$$\begin{CD}
   C(V,\C) @>\Delta^\sigma>> C(V,\C) \\
   @VV{\mathcal{F}_\rho}V          @VV{\mathcal{F}_\rho}V \\
   C_\rho(\widehat V,\C) @>>\Delta_{\widehat G}> C_\rho(\widehat V,\C))
\end{CD}
$$
where $\mathcal{F}_\rho$ maps $f \in C(V,\C)$ now to a $\rho$-equivariant function $\vec{f}$ on $\widehat V$.

The spectrum of $\Delta_{\widehat G}$ is then the multiset disjoint union of the spectra of the magnetic Laplacians $\Delta^{\sigma^j}$ on $G$ for $j=0,1,\dots,k-1$.

\medskip

We now discuss the analogous setting in the case of a compact connected Riemannian manifold $(M,g)$ with base point $x_0 \in M$. Let $\pi: \widetilde M \to M$ be the universal covering map and $\tilde x_0 \in \pi^{-1}(x_0) \subset \tilde M$. Then the fundamental group $\Gamma = \pi_1(M,x_0)$ acts by deck transformations on the universal cover $\widetilde M$. Any representation $\rho: \pi_1(M,x_0) \to S^1$ gives rise to a flat complex line bundle $L_\rho \to M$ and to a (one-dimensional) \emph{twisted Laplacian} $\Delta^\rho: C^\infty(M,L_\rho) \to C^\infty(M,L_\rho)$ with the following properties: 
\begin{itemize}
\item[(a)] There exists a one-one correspondence $\mathcal{F}_\rho$ between sections $f \in C^\infty(M,L_\rho)$ and $\rho$-equivariant functions $\vec{f} = \mathcal{F}_\rho f$ in $C^\infty_\rho(\widetilde M,\C)$, that is, $\vec{f}$ satisfies 
$$ \vec{f}(\gamma x) = \rho(\gamma)\vec{f}(x) \quad \text{for all $x \in \widetilde M$ and $\gamma \in \Gamma$.} $$
\item[(b)] We have the following commutative diagram:
$$\begin{CD}
   C^\infty(M,L_\rho) @>\Delta^\rho>> C^\infty(M,L_\rho) \\
   @VV{\mathcal{F}_\rho}V          @VV{\mathcal{F}_\rho}V \\
   C^\infty_\rho(\widetilde M,\C) @>>\Delta_{\widetilde M}> C^\infty_\rho(\widetilde M,\C)
\end{CD}
$$
where $\Delta_{\widetilde M}$ is the standard Laplacian on $\widetilde M$.
\end{itemize}
If $\rho$ factors through the finite subgroup $S_k$, there exists a finite $k$-cover $\widehat M$ of $M$ such that the spectrum of the standard Laplacian $\Delta_{\widehat M}$ agrees again with the multiset disjoint union of the spectra of the twisted Laplacians $\Delta^{\rho^j}: C^\infty(M,L_{\rho^j}) \to C^\infty(M,L_{\rho^j})$ on $M$ for $j=0,1,\dots,k-1$. For more information about (higher dimensional) twisted Laplacians see, e.g., \cite{Su89} or \cite[Chapter 3]{CP23}. 

Let us briefly discuss the relation between one-dimensional twisted Laplacians and \emph{magnetic Laplacians} $\Delta^\alpha: C^\infty(M,\C) \to C^\infty(M,\C)$ on a compact connected Riemannian manifold $(M,g)$ associated to (real-valued) magnetic potentials $\alpha \in \Omega^1(M)$ (see \cite[Definition 1.1]{G24} for a real-valued version of twisted Laplacians). The $\alpha$-magnetic Laplacian is given by (see, e.g., \cite[(1.1)]{Shi87})
\begin{equation} \label{eq:Deltaalpha} 
\Delta^\alpha f = \Delta_M f - 2 i \, \langle {\rm{grad}} f,\alpha^\# \rangle + \left( i \, d^* \alpha + \Vert \alpha^\# \Vert^2 \right) f, 
\end{equation}
where $\alpha^\#$ is the vector field determined by the relation $\alpha(Y)= \langle Y,\alpha^\# \rangle$ and $d^*: \Omega^1(M) \to C^\infty(M)$ is the formal adjoint of the exterior derivative $d: C^\infty(M) \to \Omega^1(M)$. The operator $\Delta^\alpha$ on a compact Riemannian manifold is self-adjoint with non-negative discrete spectrum, and it can also be viewed as the operator $\Delta^{\tilde \alpha}$, acting on $\Gamma$-periodic functions $\tilde f \in C_\Gamma(\tilde M,\C)$ with $\tilde \alpha = \pi^* \alpha \in \Omega^1(\widetilde M)$. 

Any \emph{harmonic} magnetic potential $\alpha \in \Omega^1(M)$ (that is, $d \alpha = d^* \alpha = 0$) gives rise to a representation $\rho_\alpha: \pi_1(M,x_0) \to S^1$ as follows:
$$ \rho_\alpha(\gamma) = \exp(-i \, \int_{\tilde x_0}^{\gamma \tilde x_0} \tilde \alpha), $$
where the path integral $\int_{\tilde x_0}^{\gamma \tilde x_0} \tilde \alpha$ is independent on the path by Stokes and the fact that $d \tilde \alpha = 0$. Moreover, for any harmonic magnetic potential $\alpha$ and associated representation $\rho=\rho_\alpha$, the operators $\Delta^{\tilde \alpha}: C_\Gamma^\infty(\widetilde M,\C) \to C_\Gamma^\infty(\widetilde M,\C)$ and $\Delta^\rho: C_\rho^\infty(\widetilde M,\C) \to C_\rho^\infty(\widetilde M,\C)$ are unitarily equivalent via the gauge function $\tau \in C^\infty(\widetilde M, S^1)$, given by
$$ \tau(\tilde x) = \exp\left( -i \int_{\tilde x_0}^{\tilde x} \tilde \alpha \right). $$
This follows from ${\rm{grad}}\, \tau = -i \tilde \alpha^\#$, $\Delta_{\widetilde M}\, \tau = \Vert \tilde \alpha^\# \Vert^2$, $d^* \tilde \alpha=0$ and 
$$ \tau^{-1} \Delta_{\widetilde M} (\tau \tilde f) = \Delta_{\widetilde M} \tilde f - 2 i \langle {\rm{grad}} \tilde f, \tilde \alpha^\# \rangle + \Vert \tilde \alpha^\# \Vert^2 \tilde f = \Delta^{\tilde \alpha} \tilde f, $$
for any $\Gamma$-periodic function $\tilde f$. This shows $\tau^{-1} \Delta^\rho \tau = \Delta^{\tilde \alpha}$ and, therefore, every magnetic Laplacian with a harmonic potential is a one-dimensional twisted Laplacian. However, the converse is not true: There are no harmonic $1$-forms on the real projective plane $\mathbb{R} P^2$, but there is a non-constant representation $\rho: \pi_1(\mathbb{R} P^2,x_0) \to \{\pm 1\}$. 

Our magneto-spectral invariant in the  manifold setting is defined as follows:

\begin{definition}[Magneto-spectral height for manifolds]
\label{def:msheightformanifold}
Let $(M,g)$ be a compact Riemannian manifolds. The \emph{magneto-spectral height} of $M$ is
$$ \nu(M) = \sup_\rho \lambda_1^\rho(M), $$
where $\rho: \pi_1(M,x_0) \to S^1$ runs through all $S^1$-representations and $\lambda_1^\rho(M)$ is the smallest eigenvalue of the twisted Laplacian $\Delta^\rho$. 
\end{definition}

It follows directly from its definition, that we have $\nu(M) = 0$ for all simply connected manifolds $M$. This is the analogue of the property $\nu(T) = 0$ for finite trees. Let us finish with the following remark explaining the reason behind our Definition \ref{def:msheightformanifold}.

\begin{remark} Another potential continuous counterpart of the magneto-spectral height for compact Riemannian manifolds $(M,g)$ could be 
\begin{equation} \label{eq:nualtcont}
\sup_{\alpha \in \Omega^1(M)} \lambda_1^\alpha(M), 
\end{equation}
where $\alpha$ runs through all magnetic potentials on $M$ and $\lambda_1^\alpha(M)$ is the smallest eigenvalue of the magnetic Laplacian $\Delta^\alpha$, given in \eqref{eq:Deltaalpha}. Readers may ask why we did not choose this alternative definition.

Since, by Hodge Theory, every $\alpha \in \Omega^1(M)$ can be uniquely written as the sum $\alpha_0 + d f + d^* \omega$ with harmonic $\alpha_0 \in \Omega^1(M)$, $f \in C^\infty(M)$ and $\omega \in \Omega^2(M)$, the term $d f$ can be gauged away and \eqref{eq:nualtcont} agrees therefore with
$$ 
\sup_{\alpha_1} \lambda_1^{\alpha_1}(M),
$$
where $\alpha_1$ runs through all magnetic potentials on $M$ with $d^* \alpha_1 = 0$. However, any term $d^* \omega \neq 0$ in the Hodge decomposition of $\alpha$ cannot be gauged away, which implies that we cannot restrict the supremum to run only over harmonic magnetic potentials. 

Due to this fact, this alternative definition has the following significant disadvantage: It is no longer true that the expression \eqref{eq:nualtcont} vanishes for simply connected compact Riemannian manifolds. Even worse, the expression \eqref{eq:nualtcont} is infinite for the simply connected $3$-sphere (see \cite[Section 5.2]{CGHP25}). 
\end{remark}

\end{document}